\documentclass[11pt,reqno,a4paper]{amsart}
\usepackage{a4wide}
\usepackage{amsmath,amsthm,amssymb,amsrefs} %fullpage creates more narrow margins -> less pages to print
\usepackage{graphicx}
\usepackage{stmaryrd} % \llbracket and \rrbracket from the stmaryrd package
\usepackage{tikz-cd} % for commutative diagrams
\usepackage{verbatim}
\newtheorem{thm}{Theorem}[section]
\newcommand{\bt}{\begin{thm}}
\newcommand{\et}{\end{thm}}

\newtheorem{ex}[thm]{Example}

\newtheorem{cor}[thm]{Corollary} 
\newcommand{\bc}{\begin{cor}}
\newcommand{\ec}{\end{cor}}

\newtheorem{lem}[thm]{Lemma}
\newcommand{\bl}{\begin{lem}}
\newcommand{\el}{\end{lem}}

\newtheorem{prop}[thm]{Proposition}
\newcommand{\bp}{\begin{prop}}
\newcommand{\ep}{\end{prop}}

\newtheorem{defn}[thm]{Definition}

\newtheorem{rmrk}[thm]{Remark}

\newtheorem{innercustomthm}{Theorem}
\newenvironment{customthm}[1]
  {\renewcommand\theinnercustomthm{#1}\innercustomthm}
  {\endinnercustomthm}

\newcommand{\GHto}{\stackrel { \textrm{GH}}{\longrightarrow} }
\newcommand{\Fto}{\stackrel {\mathcal{F}}{\longrightarrow} }

\newcommand{{\set}}{\rm{set}}

\newcommand{\bee}{\begin{equation*}}

\newcommand{\ee}{\end{equation}}
\newcommand{\eee}{\end{equation*}}

\newcommand{\N}{\mathbb{N}}
\newcommand{\R}{\mathbb{R}}
\newcommand{\M}{\mathbb{M}} % Minkowski space
\newcommand{\E}{\mathbb{E}} % Euclidean space
\newcommand{\Z}{\mathbb{Z}}

\newcommand{\diam}{\operatorname{Diam}}

\newcommand{\Fm}{{\mathcal F}}
\newcommand{\image}{\operatorname{Im}}

\newcommand{\id}{\operatorname{id}} % identity map
\newcommand{\Lip}{\operatorname{Lip}}
\newcommand{\mass}{{\mathbf M}}

\def\g{\mathbf{g}}
\def\h{\mathbf{h}}
\def\bsigma{{\boldsymbol{\sigma}}}
\def\etab{{\boldsymbol{\eta}}}

\newcommand{\A}{\hat{\mathcal{A}}}

\makeatletter
\@namedef{subjclassname@2020}{%
  \textup{2020} Mathematics Subject Classification}
\makeatother

\begin{document}

\title[Properties of the null distance and spacetime convergence]
{Properties of the null distance and\\ spacetime convergence}

\author[B.~Allen]{Brian Allen$^\dagger$}
\thanks{$^\dagger$Department of Mathematics, University of Hartford, 200 Bloomfield Avenue, West Hartford, CT 06117-1599, USA, \textit{Email}: \texttt{brianallenmath@gmail.com}}

% Corresponding author
\author[A.~Burtscher]{Annegret Burtscher$^\ddagger$}
\thanks{$^\ddagger$Department of Mathematics, Radboud University, PO Box 9010, Postvak 59, 6500 GL Nijmegen, The Netherlands, \textit{Email}: \texttt{burtscher@math.ru.nl}}

% \date{\today}
% 
% \subjclass[2020]{53C50 (primary), and 53C23, 53B30, 53Z05, 83C20, 49Q15 (secondary)}
% 
% \keywords{Lorentzian geometry, Gromov--Hausdorff convergence, Sormani--Wenger intrinsic flat convergence, spacetime convergence, uniform convergence, null distance, length spaces, intrinsic metric, integral current spaces, globally hyperbolic spacetimes, warped products, time functions, general relativity}

\begin{abstract} 
 The null distance for Lorentzian manifolds was recently introduced by Sormani and Vega. Under mild assumptions on the time function of the spacetime, the null distance gives rise to an intrinsic, conformally invariant metric that induces the manifold topology. We show when warped products of low regularity and globally hyperbolic spacetimes endowed with the null distance are (local) integral current spaces. This metric and integral current structure sets the stage for investigating convergence analogous to Riemannian geometry. Our main theorem is a general convergence result for warped product spacetimes relating uniform, Gromov--Hausdorff and Sormani--Wenger intrinsic flat convergence of the corresponding null distances. In addition, we show that non-uniform convergence of warping functions in general leads to distinct limiting behavior, such as limits that disagree.
\end{abstract}

\maketitle

%\newpage
%\tableofcontents

%%%%%%%%%%%%%%%%%%%%%%%%%%%%%%%%%%%%%%%%%%%%%%%%%%%%%%%%%%%%%%%%%%%%%%%%%%%%%%%%%%%%%%%%%%%%%%%%%%%%%%%%%%%%%%%%%%%%%%%

\section{Introduction}\label{Intro}

%%%%%%%%%%%%%%%%%%%%%%%%%%%%%%%%%%%%%%%%%%%%%%%%%%%%%%%%%%%%%%%%%%%%%%%%%%%%%%%%%%%%%%%%%%%%%%%%%%%%%%%%%%%%%%%%%%%%%%%

 What is a good notion of convergence for Lorentzian manifolds? Riemannian manifolds naturally carry the structure of metric spaces, and standard notions of metric convergence such as Gromov--Hausdorff (GH) and Sormani--Wenger intrinsic flat (SWIF) convergence \cite{SW} naturally interact with the Riemannian structure and (weak) curvature bounds. 
 Despite the increasing need to understand Lorentzian manifolds of low regularity, a comprehensive metric theory is still in its infancy. 

 Shing-Tung Yau first suggested that a notion of spacetime SWIF convergence is needed in light of geometric stability questions arising from spacelike observations in connection with the positive mass theorem by Lee and Sormani~\cite{LeeS}. The most natural analogue to the Riemannian distance function, the Lorentzian distance, is based on the maximization of curve lengths. But this distance does not even give rise to a metric structure (see \cite{BEE}*{Ch.\ 4} for more details), let alone an integral current space needed for SWIF convergence. With the aim to obtain a metric theory suitable for spacetime convergence Sormani and Vega defined the null distance for Lorentzian manifolds with time functions \cite{SV}. In upcoming work, announced in \cite{S-Ob}, Sormani and Vega analyze big bang spacetimes using the null distance and Sakovich and Sormani \cite{SakS2} consider the null distance with respect to the cosmological time function on maximal developments. Note that Noldus~\cite{No} has a different approach towards defining a GH convergence based on the Lorentzian distance. 
 
 In this paper, we investigate in detail the metric and integral current structure of Lorentzian manifolds with respect to the null distance and provide spacetime convergence results for warped product spacetimes of low regularity. While as soon as the null distance is shown to be definite it can be readily used for GH convergence, to consider SWIF convergence we also need to prove bi-Lipschitz bounds to show that the spacetimes under consideration are (local) integral current spaces.

 A comprehensive spacetime convergence theory will open the door to an exciting new direction of geometric spacetime stability and low-regularity Lorentz\-ian geometry. Causality theory for continuous Lorentzian metrics has already been investigated extensively (see Sorkin and Woolgar \cite{SWo}, Chru\'{s}ciel and Grant \cite{CG}, Galloway, Ling, and Sbierski \cite{GLS}, Graf and Ling \cite{GL}, S\"amann \cite{Sae}, see also Minguzzi \cite{Min}, and Minguzzi and Suhr \cite{MSu} for related work). Paired with weak notions of curvature (bounds) new applications in both mathematics and physics are expected. In particular in connection with notions of distributional curvature (see, for instance, Chen and Li~\cite{CL}, LeFloch and Mardare~\cite{LM}) and curvature bounds (see, for instance, Alexander and Bishop~\cite{AB}, Andersson and Howard~\cite{AH}, Harris \cite{H}, Kunzinger and S\"{a}mann \cite{KS}, McCann \cite{McC}, Mondino and Suhr \cite{MoSu}, Yun \cite{Y}) in general relativity, GH and SWIF spacetime convergence has the potential to transform Lorentzian geometry the same way the metric approach has led to a new era in Riemannian geometry a few decades ago.

\subsection{Spacetime metric structure and convergence}\label{subsec: Spacetime Metric Structure Intro}
 
 In search for a far-reaching metric structure on Lorentzian manifolds Sormani and Vega \cite{SV} recently defined the null distance $\hat d_\tau$ on spacetimes with a (generalized) time function $\tau$. From a purely geometric point of view the null distance is thus immediately defined on any noncompact manifold, since any such manifold automatically admits a spacetime structure with a smooth time (or even temporal) function \cite{Kok}*{Theorem 2}. The null distance is based on the minimization of the null length of broken causal curves. More precisely, if $\beta \colon [a,b] \to M$ is a continuous, piecewise smooth causal curve that is either future-directed or past-directed causal on its pieces $a=s_0 < s_1 < \ldots < s_k =b$, then the null length of $\beta$ is given by
 \[
  \hat L_\tau (\beta) = \sum_{i=1}^k |\tau(\beta(s_i))-\tau(\beta(s_{i-1}))|,
 \]
 and the null distance between $p,q \in M$ is
 \[
  \hat d_\tau(p,q) = \inf \{ \hat L_\tau(\beta) : \beta ~\text{is a piecewise causal curve from $p$ to $q$} \}.
 \]
 This null distance turns a time-oriented Lorentzian manifold with a time function $\tau$, satisfying mild assumptions, into a metric space in a conformally invariant way. Indeed, in Section~\ref{sect-Length Spaces} we show that spacetimes with suitable time functions possess additional structure.
 
\begin{thm}\label{thm:lengthspace}
 If $(M,\g)$ is a spacetime with locally anti-Lipschitz time function $\tau$, then $\hat d_\tau$ is an intrinsic metric and hence $(M,\hat d_\tau)$ is a length space. 
\end{thm}

 The local influence of different (suitable) time functions on the metric structure obtained by the null distances is minor, in particular, when temporal functions are used (which is no additional restriction \cite{Min}*{Theorem 4.100}). In this paper we are mostly interested in globally hyperbolic spacetimes and sequences of compact warped product spacetimes. For the latter one can easily verify the global equivalence of null distances $\hat d_\tau$ for all time functions $\tau$ as defined in Theorem~\ref{thm:warpedcausality} (the more general case is discussed in Section~\ref{sect-Background}).

 With the null distance metric space structure at hand some particularly important applications in general relativity can be addressed, such as questions of geometric stability.  For instance, a question posed by Sormani in \cite{S-Ob} is ``In what sense is a black hole spacetime of small mass close to Minkowski space?''

 The question of geometric stability has already been studied extensively for spacelike slices of spacetimes by Allen \cite{BA1,BA2,BA3}, Bray and Finster \cite{BF}, Bryden, Khuri, and Sormani \cite{BKS}, Cabrera Pacheco \cite{ACP}, Finster \cite{FelF}, Finster and Kath \cite{FK}, Huang and Lee \cite{HL}, Huang, Lee, and Sormani \cite{HLS}, Jauregui and Lee \cite{JL}, Lee \cite{DL}, Lee and Sormani \cite{LeeS}, LeFloch and Sormani \cite{LFS}, Sakovich and Sormani \cite{SakS}, Sormani and Stavrov Allen \cite{SStav}. Lee and Sormani \cite{LeeS} conjecture that Sormani--Wenger intrinsic flat (SWIF) convergence is the correct notion of convergence for stability. This conjecture is supported by one of their examples with increasingly many, increasingly small gravity wells of fixed depth which does not have a smooth or Gromov--Hausdorff (GH) limit but whose SWIF limit is Euclidean space, which justifies this claim. In the full Lorentzian case the notion of null distance is crucial in order to turn a spacetime into a metric space so that standard metric notions of convergence such as GH and SWIF convergence can by applied to spacetimes. 
 
 Since the SWIF notion of distance is defined on integral current spaces (intuitively a metric space with countable many bi-Lipschitz charts to $\R^n$) one of our goals of this paper is to show that globally hyperbolic spacetimes are (local) integral current spaces (recall that GH distance is defined as soon as the null distance is definite). Currents are geometric objects used in the calculus of variations and geometric measure theory that generalize the notion of integration. The theory of currents was developed by Federer and Fleming for Euclidean space \cite{FF} and extended to metric spaces by Ambrosio and Kirchheim \cite{AK}. The theory of rectifiable currents played a crucial role in establishing the existence of minimal surfaces with given boundary in the higher dimensional Plateau problem. An integral current structure, in particular, guarantees that a metric space is a countably $\mathcal{H}^{n+1}$ ($n+ 1$-dimensional Hausdorff measure) rectifiable space \cite{SW} which is an important additional property of a metric space and useful for studying spacetimes of low regularity.
 
\begin{thm}\label{thm: Warp are integral currents}
 Let $I$ be an interval and $(\Sigma,\bsigma)$ be a connected complete Riemannian manifold. Suppose $f \colon I \to (0,\infty)$ is a bounded function that is bounded away from $0$. There is a natural local integral current space structure on the warped product spacetime $M = I \times_f \Sigma$ with respect to the null distance $\hat d_f$.  If $I \times \Sigma$ is compact then $(M,\hat d_f)$ carries an integral current space structure. 
\end{thm}

 In order to prove Theorem \ref{thm: Warp are integral currents} we derive explicit expressions for the null distance of Lorentzian products as well as detailed estimates of the null distance on warped products with bounded warping functions. These prototype classes of spacetimes are important to better understand the null distance, and we also use this insight in obtaining a more general result for globally hyperbolic spacetimes.
 
\begin{thm}\label{thm: Globally Hyperbolic Are Integral Current Spaces}
 Let $(M,\g)$ be a globally hyperbolic spacetime with (smooth) time function $\tau$. Suppose $M$ admits Cauchy hypersurfaces on which the ambient Lorentzian metric restricts as a complete Riemannian metric and $(M,\hat d_\tau)$ is complete as a metric space. Then $(M,\hat{d}_\tau)$ is a local integral current space.
If $M$ is compact then it is an integral current space.
\end{thm}

 The main ingredient that allows us to extend Theorem~\ref{thm: Warp are integral currents} to the globally hyperbolic setting in Theorem~\ref{thm: Globally Hyperbolic Are Integral Current Spaces} is that a (locally) bi-Lipschitz map from a (local) integral current space to a metric space naturally induces a (local) integral current space structure on the target. Since this is an important result in its own right it is stated separately for integral current spaces in Theorem~\ref{thm:intcurrent} and local integral current spaces in \ref{thm:localintcurrent}.

 The metric and integral current structure of spacetimes establishes the basis for proving spacetime convergence results. In this paper we initiate this investigation of spacetime convergence for the important class of warped product spacetimes. In general relativity, for instance, the cosmological Friedmann--Lema\^{\i}tre--Robertson--Walker (FLRW) metrics describing a homogeneous and isotropic universe are of this class. Our goal is to identify conditions on the sequence of warping functions which guarantee GH and/or SWIF convergence with respect to the null distance with the aim of generalizing these conditions to arbitrary sequences of spacetimes in the future. Due to the physical relevance of Lorentzian geometry it is also crucial to know when limits of spacetimes disagree (or do not exist at all). In the final section of this paper we therefore present examples of warped product spacetimes with distinct limiting behavior. For instance, we construct an example with degenerate behavior of the limiting warping function which leads to different GH and SWIF limits. Two more examples illustrate that the way in which the warping functions converge and the properties of the limiting warping function itself influence the structure of the limiting integral current space. For instance, $L^p$-convergence in the spacetime setting also appears to contain useful information, something that was recently already observed for Riemannian warped products by Allen and Sormani \cite{AS}, however, with important distinctions coming from the fact that the null lengths are not as directly related to the warping functions as in the Riemannian case. Motivated by these examples we prove our main theorem on the convergence of warped product spacetimes.

\begin{thm}\label{thm:fuconv}
 Let $(\Sigma,\bsigma)$ be a connected compact Riemannian manifold, and $I$ be a closed interval. Suppose $(f_j)_j$ is a sequence of continuous functions $f_j \colon I \to (0,\infty)$ (uniformly bounded away from $0$) and $M_j = I \times_{f_j} \Sigma$ are warped products with Lorentzian metrics
 \[
  \g_j = -dt^2 + f_j(t)^2 \bsigma.
 \]
 Assume that $(f_j)_j$ converges uniformly to a limit function $f_\infty$.
 Then the null distances corresponding to the integral current spaces $M_j$ converge in the uniform, GH, and SWIF sense to the null distance defined with respect to the warped product $M_{\infty} = I \times_{f_\infty}\Sigma$.
\end{thm}

 Note that $\Sigma$ is explicitly allowed to have a boundary (see Remark~\ref{rmrk-Boundary Allowed in HLS}).
 
 Though this theorem involves a stronger notion of convergence then one expects for stability of spacetime questions with gravity wells and black holes, results similar to Theorem~\ref{thm:fuconv} have already successfully been used in the stability of the Riemannian positive mass theorem~\cite{BA4,HLS}. In \cite{BA4}, Allen used a related Riemannian convergence result in order to efficiently handle singular sets, which can be done by showing uniform or smooth convergence on compact sets away from the singular set (see, for instance, the work of Lakzian \cite{Lak} and Lakzian and Sormani \cite{LS}). In \cite{HLS}, Huang, Lee, and Sormani needed a stronger convergence on the boundary of interest in order to show that no points on the outer boundary disappear in the limit. Besides stability questions, uniform convergence can also be employed to analyze structural properties of continuous (semi-)Riemannian metrics (see, for instance, the work of Burtscher \cite{B} and Chru\'{s}ciel and Grant \cite{CG}).

\subsection{Causality and completeness}

 Sormani and Vega already proved that the null distance is always a pseudometric and, as long as the time function is anti-Lipschitz, definite. As discussed earlier, in Theorem~\ref{thm:lengthspace} we show that the (local) anti-Lipschitz property of $\tau$ automatically also yields a length space $(M,\hat d_\tau)$. 
 
 Since there are, however, many artificial ways to turn a Lorentzian manifold into a metric space it is important that such a metric is connected to the causal structure of the Lorentzian manifold on which it is defined.It is clear that the null distance satisfies
 \[
  p \leq q \Longrightarrow \hat d_\tau (p,q) = \tau(q) - \tau(p).
 \]
 If the converse also holds $\hat d_\tau$ is said to encode causality, which is thus a property that enables the null distance to capture whether or not points of a spacetime are causally related. Fully understanding when the null distance encodes causality is mentioned as an open problem in \cite{SV}. For warped product spacetimes, the answer is already known.
 
\begin{thm}[\cite{SV}*{Theorem 3.25}]\label{thm:warpedcausality}
 Let $I$ be a interval, $f \colon I \to (0,\infty)$ smooth and $(\Sigma,\bsigma)$ a complete Riemannian manifold. Let $M = I \times_f \Sigma$ be the warped spacetime with metric tensor
 \[
  \g = -dt^2 + f(t)^2 \bsigma,
 \]
 and let $\tau(t,x) = \phi(t)$ be a smooth time function on $M$ with $\phi'>0$. Then the induced null distance $\hat d_\tau$ is definite and encodes causality of $M$, i.e., for all $p,q \in M$,
 \[
  p \leq q \Longleftrightarrow \hat d_\tau (p,q) = \tau(q) - \tau(p).
  \]
\end{thm}

 We provide two examples in Section~\ref{sect-Encoding Causality} which show how the null distance can fail to encode causality when it is not complete.

 Since metric completeness is generally a useful property, it is desirable to be able to characterize completeness in terms of conditions on the spacetime and time function alone. In this direction we provide the following sufficient conditions for the null distance to be complete.

\begin{thm}\label{Conditions Implying Null Complete 2}
 Let $(M,\g)$ be a spacetime with time function $\tau$. If $\tau$ is anti-Lipschitz on $M$ with respect to a complete metric on $M$ (that induces the manifold topology), then $(M,\hat d_\tau)$ is complete. 
\end{thm}
 
 Note that the notion of null distance is relatively new and also many interesting and deep questions about its structural properties other than those mentioned here are yet to be fully explored. The conformal invariance and causality-encoding property of the null distance already show that this notion can play a crucial and viable rule in general relativity not only in connection with spacetime stability but also in, for instance, Lorentzian causality theory.
 
\subsection{Null distance and weak spacetimes in general relativity}

 Because the null distance at essence just requires a notion of causality (via the admissible class of piecewise causal curves) and a definition of time function, we note that the null distance can be generalized to very weak notions of spacetimes.
 
 In general relativity, very weak notions of Lorentzian metrics emerge naturally as solutions to the Einstein equations. This is most evident in matter models involving compressible fluids since discontinuities occur both at the matter-vacuum boundary and in the form of shocks within the fluid (see, for instance, Barnes, LeFloch, Schmidt and Stewart~\cite{BLSS}, Burtscher and LeFloch~\cite{BLF}, Groah, Smoller and Temple \cite{GST}, Le Floch and LeFloch~\cite{LLF}, Rendall and St\aa{}hl~\cite{RS}). Also for the vacuum Einstein equations and other matter models weak solutions play a prominent role (see, for instance, Christodoulou~\cite{Chr}, Sbierski~\cite{Sb}, Dafermos and Luk~\cite{DL}, Burtscher, Kiessling and Tahvildar-Zadeh~\cite{BKTZ}, Tahvildar-Zadeh~\cite{TZ}).
 
 From a geometric as well as causality theoretic perspective, continuous Lorentz\-ian metrics have been studied extensively by Sorkin and Woolgar \cite{SWo}, Chru\'{s}ciel and Grant \cite{CG}, Galloway, Ling, and Sbierski \cite{GLS}, Graf and Ling \cite{GL}, and S\"{a}mann \cite{Sae}. It will be particularly insightful to use the null distance in place of the commonly used background Riemannian metric on weak Lorentzian manifolds, for instance, in recent work of Kunzinger and S\"amann \cite{KS} who investigated a Lorentzian notion of length spaces based on a time-separation function on a given metric space. Using the null distance as background metric for sufficiently regular Lorentzian manifolds in place of an auxiliary Riemannian distance, naturally provides a link to the causal and length space structure.
 
 Riemannian GH and SWIF convergence results have been particularly important in connection with lower curvature bounds (recall the influential work of Gromov~\cite{G}, Grove and Petersen~\cite{GP} as well as Burago, Gromov and Perel\cprime man~\cite{BGP} on lower sectional curvature bounds, and Cheeger and Colding \cite{CC,CC1} on lower Ricci curvature bounds). Besides, the Cheeger--Gromov collapse theory has already been applied by Anderson~\cite{A1,A2} to study the long-term behavior of solutions to the vacuum Einstein equations.
 In the Lorentzian context, distributional notions of curvature in Lorentzian geometry have been considered already by Lichnerowicz~\cite{Lic} and Taub~\cite{Ta}, and more systematically by LeFloch and Mardare~\cite{LM} and Chen and Li~\cite{CL}. Lorentzian analogues for sectional curvature bounds have been studied by Harris \cite{H}, Andersson and Howard~\cite{AH}, Alexander and Bishop~\cite{AB}, and Kunzinger and S\"{a}mann \cite{KS}, and relations to Ricci curvature bounds by Yun \cite{Y}, McCann \cite{McC}, and Mondino and Suhr \cite{MoSu}. Paired with (weak) notions of curvature bounds, spacetime convergence results with respect to the null distance can provide new tools to analyze the topology of spacetimes and spacetime singularities.

\subsection{Outline} 

 We proceed to give a brief description of each section of this paper.

 In Section \ref{sect-Background}, we recall the definition of the null distance and review important basic properties obtained by Sormani and Vega. Warped product spacetimes are introduced since they will be a rich source of examples throughout this paper. We then review the notions and some results of Gromov--Hausdorff (GH) convergence, (local) integral current spaces and prove that this structure can be pushed forward using (locally) bi-Lipschitz maps in Theorem~\ref{thm:intcurrent} and Theorem~\ref{thm:localintcurrent}, and Sormani--Wenger Intrinsic Flat (SWIF) convergence of integral current spaces, all of which will be used in Sections~\ref{sect-Integral Current Spaces} and \ref{sect-Convergence}. 
 
 In Section \ref{sect-Length Spaces}, we explore in detail the length space structure of the metric space defined by the null distance, including Theorem~\ref{thm:lengthspace} mentioned above. In addition, a wealth of examples illustrates what can go wrong when certain properties of the time function or null distance are not assumed. For instance, we give complete examples for which the null distance between certain points cannot be achieved by a piecewise causal curve. We  state Theorem~\ref{Conditions Implying Null Complete 2} relating the completeness of the null distance to properties of the time function and a background metric. We conclude the section with a discussion about the relation of the null distance to causality. Two simple examples illustrate that causality generally is not encoded when completeness of the null distance is not satisfied. 
  
 In Section \ref{sect-Integral Current Spaces}, we state precisely when the null distance on a spacetime defines a (local) integral current space. The basic idea is to define Lipschitz maps from Minkowski space with the null distance to a particular spacetime of interest with the null distance. Then, since we can show that Minkowski space with the null distance is a local integral current space, we push forward this (local) integral current space structure through the (locally) bi-Lipschitz map to a particular spacetime of interest using our results of Section~\ref{sect-Background}. In particular, this section culminates by showing that warped product spacetimes (of low regularity) and globally hyperbolic spacetimes with a complete Cauchy surface are (local) integral current spaces in Theorem~\ref{thm: Warp are integral currents} and Theorem~\ref{thm: Globally Hyperbolic Are Integral Current Spaces}, respectively.
    
 In Section \ref{sect-Convergence}, we state and prove our main Theorem~\ref{thm:fuconv} for warped product spacetimes which relates uniform convergence of the warping functions to uniform, GH, and SWIF convergence of the sequence of integral current spaces with the null distance. In addition, we investigate what happens if uniform convergence of the warping functions (and hence Lorentzian manifolds) is not assumed. To this end we provide three examples of sequences of warped product spacetimes with distinct limiting behavior. In particular, we construct a sequence of warped product spacetimes with the corresponding null distances whose GH and SWIF limits disagree (and no pointwise limit exists).

%%%%%%%%%%%%%%%%%%%%%%%%%%%%%%%%%%%%%%%%%%%%%%%%%%%%%%%%%%%%%%%%%%%%%%%%%%%%%%%%%%%%%%%%%%%%%%%%%%%%%%%%%%%%%%%%%%%%%%%
 
\section{Background}\label{sect-Background}
 
%%%%%%%%%%%%%%%%%%%%%%%%%%%%%%%%%%%%%%%%%%%%%%%%%%%%%%%%%%%%%%%%%%%%%%%%%%%%%%%%%%%%%%%%%%%%%%%%%%%%%%%%%%%%%%%%%%%%%%%
 
 We recall various relevant definitions and results related to convergence on metric spaces. We first introduce the null distance, (local) integral current spaces, Sormani--Wenger intrinsic flat distance and Gromov--Hausdorff convergence for (Riemannian and) Lorentzian manifolds and then explain what is known about their relation in the Riemannian setting.

\subsection{Null distance}\label{subsec-nulldistance}
 
 Let $(M,\g)$ be a spacetime, that is, a time-oriented connected Lorentzian manifold of dimension $n+1$. For our purposes $M$ is always second countable and Hausdorff. Note that these topological assumptions are even superfluous for solutions to the Einstein equations and Lorentzian manifolds with a second countable Cauchy surface due to \cite{W}*{Lemma 15 \& Corollary 22}. In our convention the metric tensor $\g$ has signature $(-,+,+,\ldots,+)$. By $\tau$ we denote a \emph{time function} of $(M,\g)$, i.e., a continuous function $\tau \colon M \to \R$ that strictly increases along future-directed causal curves \cite{BEE}*{p.\ 64}. Whenever continuity is not assumed, we call $\tau$ a \emph{generalized time function}.

 We recall the notions of null length and null distance associated to a particular generalized time function $\tau$ of $(M,\g)$ as introduced in \cite{SV}. Let $\beta \colon [a,b] \to M$ be a piecewise causal curve, i.e., a piecewise smooth curve that is either future-directed or past-directed causal on its pieces $a=s_0 < s_1 < \ldots < s_k =b$ (moving forward and backward in time is allowed). The \emph{null length} of $\beta$ is given by
 \begin{align*}
  \hat L_\tau (\beta) = \sum_{i=1}^k |\tau(\beta(s_i))-\tau(\beta(s_{i-1}))|.
 \end{align*}
 If $\tau$ is differentiable along $\beta$ we can compute the null length of $\beta$ also via the integral
 \begin{align*}
  \hat L_\tau (\beta) = \int_a^b |(\tau \circ \beta)'(s)| ds.
 \end{align*}
 For any $p,q \in M$, the \emph{null distance} is given by
 \[
  \hat d_\tau (p,q) = \inf \{ \hat L_\tau (\beta) : \beta \text{ is a piecewise causal curve from } p \text{ to } q \},
 \] 
 where we use that there always exists a piecewise causal curve between two points in $M$ \cite{SV}*{Lemma 3.5}.
 Unlike the Lorentzian distance function (see \cite{BEE}) which is not a true distance because it is neither definite, symmetric nor satisfies the triangle inequality, the null distance $\hat d_\tau$ is a pseudometric that in most cases encodes causality (see \cite{SV}*{Lemmas 3.6 and 3.8} and the discussion in Section \ref{sect-Encoding Causality}). Note that $\hat d_\tau$ is continuous on $M \times M$ iff $\tau$ is continuous \cite{SV}*{Proposition 3.14}. The topology induced by $\hat d_\tau$ coincides with the manifold topology if and only if $\tau$ is continuous and $\hat d_\tau$ is definite \cite{SV}*{Proposition 3.15}. We discuss the relation of $\hat d_\tau$ to causality in detail in Section \ref{sect-Encoding Causality}. Note that if temporal functions $\tau$ are used instead of time functions, the corresponding null distances $\hat d_\tau$ are equivalent on compact sets (for the warped product case see below, the general case is treated separately in \cite{BGH} and can be compared to the local bi-Lipschitz property of Riemannian distances \cite{B}*{Theorem 4.5}; conformal invariance and scaling with $\tau$ was already shown in \cite{SV}*{Proposition 3.9 \& Proposition 3.18}).

\subsection{Warped product spacetimes}\label{subsec: Warped Product Spacetimes}

 Most importantly for our work, definiteness of the null distance holds for spacetimes with ``regular'' cosmological time function \`{a} la Andersson, Galloway and Howard \cite{AGH}, and warped spacetimes $I \times_{f} \Sigma$. Such warped spacetimes are $n+1$ dimensional manifolds with metric tensor
 \[
  \g = - dt^2 + f^2(t) \bsigma,
 \]
 where $I$ is an open interval, $f$ is positive and smooth and $(\Sigma,\bsigma)$ is a complete Riemannian manifold. The result for smooth time functions $\tau (t,x) = \phi(t)$ with $\phi'>0$ is contained in \cite{SV}*{Theorem 3.25}. For continuous warping functions $f$ with the canonical time function $\tau(t,x) = t$ it is \cite{SV}*{Lemma 3.23}. If $\tau =t$ we write $p=(t_p,p_\Sigma)$ for a point $p \in I \times_{f} \Sigma$. Lorentzian manifolds of this type are also referred to as \emph{generalized Robertson--Walker (GRW) spacetimes} (see \cite{AW,MM,Z} for an overview of some properties). It is easy to see that the null distances $\hat d_{\tau_i}$ with respect to different time functions $\tau_i(t,x) =\phi_i(t)$, $\phi'_i>0$, are equivalent metrics on compact sets. % In the global compact case needed for our results it is simply the mean value theorem. 
 
 \begin{lem}\label{lem:timeequiv}
 Let $I$ be a closed interval, $(\Sigma,\bsigma)$ be a compact Riemannian manifold and $M = I \times_f \Sigma$ a warped product spacetime. If $\tau$ is the time function defined by $\tau(t,x) = \phi(t)$ with $\phi'>0$ then there exists a constant $C>0$ such that
 \[
  \frac{1}{C} \hat d_t(p,q) \leq \hat d_\tau (p,q) \leq C \hat d_t(p,q), \qquad p,q \in M.
 \]
 Moreover, if $\phi'$ is bounded above and below away from zero, the equivalence holds globally on any (also noncompact) $M$.
\end{lem}

\begin{proof}
 Without loss of generality we assume that $p \leq q$. We can write $p=(t_p,p_\Sigma)$ and $q=(t_q,q_\Sigma)$, and assume that $t_p < t_q$ (if equality holds then $p=q$, where $\hat d_t$ and $\hat d_\tau$ both vanish). By the mean value formula, there exists $t \in (t_p,t_q)$ such that
 \[
  \phi(t_q) - \phi(t_p) = \phi'(t) (t_q - t_p).
 \]
 Since $\phi'$ is bounded on compact sets, say $0 < A \leq \phi'(t) \leq B$, we immediately obtain that
 \[
  A \, \hat d_t(p,q) \leq \hat d_\tau (p,q) \leq B \, \hat d_t(p,q). \qedhere
 \]
\end{proof}

\subsection{Gromov--Hausdorff distance}
  
 Gromov--Hausdorff convergence \cites{BBI,E,G} is a notion for the convergence of metric spaces based on the Hausdorff distance. For two subsets $A,B$ of a common metric space $(Z,d_Z)$ the \emph{Hausdorff distance} between $A$ and $B$ is defined by
 \[
  d_\mathrm{H}^Z (A,B) = \inf \{ r > 0 : A \subseteq U_r(B) \text{ and } B \subseteq U_r(A) \},
 \]
 where $U_r(S)$ is the $r$-neighborhood of $S$, i.e., the set of all points $z \in Z$ such that $d_Z (z,S) < r$. In general, any pair of (compact) metric spaces $(X_i,d_i)$, $i=1,2$, need not be subsets of the same metric space $Z$. However, for any two separable metric spaces $A$ and $B$ such a common metric space $Z$ can always be constructed, for example, via the Fr\'echet embedding $A$ and $B$ can be isometrically embedded into the Banach space $l^\infty$ \cite{H}*{Section 1.22}. Using such isometric embeddings the concept of the Hausdorff distance can be extended to that of the \emph{Gromov--Hausdorff distance} $d_{\mathrm{GH}}$ between two metric spaces $(X,d_X)$ and $(Y,d_Y)$, given by
 \begin{align*}
   d_{\mathrm{GH}}((X,d_X),(Y, d_Y)) & \\
   = \inf \{ r > 0 : \,& \text{there is a metric space } Z \text{ and subspaces } X',Y' \\
   &\text{isometric to } X,Y \text{ such that } d^Z_H(X',Y') < r \}.
 \end{align*}
 One can show that $d_{\mathrm{GH}}((X,d_X),(Y,d_Y)) < \infty$ if $X,Y$ are bounded metric spaces. % BBI, Ex. 7.3.13

 A sequence of compact metric spaces $(X_j,d_j)$ is said to \emph{Gromov--Hausdorff converge} to a compact metric space $(X,d)$ if
 \[
  d_{\mathrm{GH}}((X_j,d_j),(X,d)) \to 0, \qquad \text{as } j \to \infty.
 \]
 In our case, all sets $X_j$ are the same manifold $M$, only the null distances $\hat d_j := \hat d_{\tau,j}$ differ. 
 Note that the Gromov--Hausdorff distance can also be introduced for pointed (noncompact) metric spaces.

\subsection{Integral currents} \label{subsec: Integral Currents}
 
 In Euclidean space integer rectifiable currents arise by integrating differential forms over rectifiable sets, and integral currents are those currents with compatible boundary currents $\partial T(\omega) = T(d\omega)$ (defined by Stokes' Theorem). This definition naturally extends to smooth manifolds.
 Currents on metric spaces were first introduced by Ambrosio and Kirchheim \cite{AK} and integral current spaces were used by Sormani and Wenger to define a notion of distance analogous to the flat distance in $\R^n$ of Federer and Fleming \cite{FF}. 
 In what follows we introduce a working definition of integer rectifiable currents on complete metric spaces using a common short-cut via parametrizations which we will then use to define integral current spaces in Section~\ref{subsec: Integral Current Spaces}. See the work of Ambrosio and Kirchheim \cite{AK} for a thorough treatment of  currents in metric spaces, Sormani and Wenger \cite{SW} for a thorough description of the theory of integral current spaces, and \cite{S} for a recent survey. We begin with the definition of a current on a metric space of Ambrosio and Kirchheim.
 
 \begin{defn}
 Let $n \in \N \cup \{ 0 \}$. An \emph{$n$-dimensional  current} $T$ on a complete metric space $(X,d)$ is a $\R$-valued multilinear functional
  \[
  (f,\pi_1,\ldots,\pi_n)\mapsto T(f,\pi_1,\ldots,\pi_n)
 \]
 on the space of $n+1$ tuples $(f,\pi_1,\ldots,\pi_n)$, where $f$ is a bounded Lipschitz function from $X \to \R$ and $\pi_i$ are Lipschitz functions from $X \to \R$ satisfying
 \begin{enumerate}
  \item \emph{Locality:} $T(f,\pi_1,\ldots,\pi_m)=0$  if  $\exists i \in \{1,\ldots,m\}$  such that  $\pi_i$ is constant on a neighborhood of $\{f \not = 0\}$.
   \item \emph{Continuity:} $T$ is continuous with respect to the pointwise convergence of the $\pi_i$ such that $\Lip(\pi_i) \le 1$.
   \item \emph{Finite mass:} there exists a finite Borel measure $\mu$ on $X$ such that
 \begin{align}\label{finitemass}
 |T(f,\pi_1,\ldots,\pi_m) \le \prod_{i=1}^m \Lip(\pi_i) \int_X |f| d\mu   
 \end{align}
 for all $n+1$ tuples $(f,\pi_1,\ldots,\pi_n)$.
 \end{enumerate}
 \end{defn}
 
  The mass of a current is defined as follows.
 
 \begin{defn}
 For a current $T$ on the complete metric space $(X,d)$, the \emph{mass measure} of $T$, which we will denote $\|T\|_d$, is defined as the smallest Borel measure $\mu$ such that
 \begin{align}\label{AKmass}
  T(f, \pi_1,\ldots,\pi_m) \le \prod_{i=1}^m \Lip_{d_j}(\pi_i) \int_X f \, d\mu.
 \end{align}
  The mass of $T$ is then defined as $\mass(T) = \|T\|_d(X)$.
\end{defn}
 
 We start with defining what it means to restrict a current to a Borel set. Although the first slot of a current is defined to be a bounded Lipschitz function, Ambrosio and Kirchheim explain that one can actually uniquely extend the definition of the current to allow any bounded Borel function $f$ still satisfying \eqref{finitemass} with $\mu = \| T\|$ \cite{AK}*{p.\ 12}. This justifies the following definition of the restriction operator.
 \begin{defn}\label{Restricted Current}
 Let $T$ be a current on a complete metric space $(X, d)$, $A\subseteq X$ a Borel measurable set and $\chi_A \colon X \to \R$ the characteristic function of $A$. Then the current
 \begin{align*}
 (T \llcorner A)(f,\pi_1,...,\pi_n) = T(f \chi_A,\pi_1,...,\pi_n)
 \end{align*}
 is called the \emph{restriction of $T$ to $A$}.
 \end{defn}
 
  \begin{ex}[Euclidean space]\label{exIntegralCurrents}
 Let $h: U \subseteq \R^n \rightarrow \Z$ be an $L^1$ function then one can define the $n$-dimensional current $\llbracket h \rrbracket$ as
 \begin{align*}
  \llbracket h \rrbracket (f,\pi_1,\ldots,\pi_n) = \int_U hf \, d \pi_1 \wedge\ldots\wedge \pi_n.
 \end{align*}
 \end{ex}

 A common way to define a current structure on a metric space $(X,d)$ is to cover $X$ by images of bi-Lipschitz maps $\varphi_i \colon U_i  \rightarrow X$, $U_i \subseteq \R^n$, and induce the current structure on $X$ through the pushforward of the usual current structure on $\R^n$ (see \cite{S}*{Section 9.3.2}). We now review the necessary definitions required to understand this approach.
 
\begin{defn}\label{Def pushforward}
 Let $X$ and $Y$ be complete metric spaces and $\varphi \colon X \to Y$ a Lipschitz map. If $T$ is a current $T$ on $X$ then the \emph{pushforward current} $\varphi_\# T$ on $X$ is given by 
 \begin{align*}
 \varphi_{\#}T(f,\pi_1,\ldots,\pi_n) = T(f\circ\varphi , \pi_1\circ \varphi,\ldots,\pi_n \circ \varphi ),
 \end{align*}
 where $\pi_i \colon Y \rightarrow \R$ are Lipschitz maps and $f \colon Y \rightarrow \R$ is a bounded Lipschitz function.
\end{defn}
 
 A useful example of an integral current on a metric space which makes use of Definition~\ref{Def pushforward} is the following.
 
\begin{ex}[Metric spaces]\label{Ex current}
 If $(X,d)$ is a complete metric space, $\varphi\colon\R^n \rightarrow X$ a bi-Lipschitz map, and a Lebesgue function $\theta \in L^1(U,\Z)$ where $U \subseteq \R^n$, then $\varphi_{\#} \llbracket \theta \rrbracket$ is an $n$-dimensional current on $X$ such that
 \begin{align*}
 \varphi_{\#} \llbracket \theta \rrbracket (f,\pi_1,\ldots,\pi_n)= \int_U \theta (f \circ\varphi) \,d(\pi_1\circ \varphi) \wedge \ldots\wedge d(\pi_n\circ \varphi),
 \end{align*}
 where by Rademacher's Theorem $d(\pi_i\circ\varphi)$ is well defined since $\pi_i \circ \varphi$ is differentiable almost everywhere.
\end{ex}

\begin{rmrk}\label{RemarkMass}
In general, for pushforward currents defined via Definition~\ref{Def pushforward}  Ambrosio and Kirchheim, in equation (2.4) of \cite{AK},  showed that
\begin{align*}
\| \varphi_{\#}T \|_d &\le (\Lip(\varphi))^n \varphi_{\#}\|T \|_d, \\  \mass( \varphi_{\#}T) &\le  (\Lip(\varphi))^m \mass(T).
\end{align*}
 For a current given as in Example~\ref{Ex current} we can deduce
 \begin{align*} 
\| \varphi_{\#}\llbracket \theta \rrbracket\|_d \le  (\Lip(\varphi))^n|\theta| \, d \mathcal{L}^n ,
 \end{align*}
 which implies the mass of $\varphi_{\#}T$ is bounded by
 \begin{align*}
 \mass( \varphi_{\#}\llbracket \theta \rrbracket) \le  (\Lip(\varphi))^m\int_U \theta \, d \mathcal{L}^n,
 \end{align*}
 where $d \mathcal{L}^n$ is the Lebesgue measure on $\R^n$. See the discussion around \cite{SW}*{Remark 2.21 and Lemma 2.42} for more details.
\end{rmrk}

 By combining Definition~\ref{Def pushforward} with Example~\ref{Ex current} one can define a parametrization of a current structure on a metric space by an atlas of bi-Lipschitz charts. We will use this as the definition of a integer rectifiable current on a metric space, as is commonly done.
 
\begin{defn}\label{Def current on metric space}
 Let $(X,d)$ be a complete metric space. By choosing a countable collection of bi-Lipschitz maps $\varphi_i \colon U_i \rightarrow X$, $U_i \subseteq \R^n$ precompact Borel measurable such that
 \begin{align*}
 \varphi_i(U_i) \cap \varphi_j(U_j) = \emptyset, \qquad \text{for } i \not = j,
 \end{align*}
 and weight functions $\theta_i \in L^1(U_i,\Z)$, we define an \emph{integer rectifiable current} $T$ on $(X,d)$ by
 \begin{align*}
 T= \sum_{i=1}^{\infty} \varphi_{i \#}\llbracket\theta_i\rrbracket,
 \end{align*}
 with mass
 \begin{align*}
 \mass(T) = \sum_{i=1}^{\infty} \mass(\varphi_{i \#}\llbracket\theta_i\rrbracket).
 \end{align*}
\end{defn}

 In particular, the current in Example~\ref{Ex current} is an integer rectifiable current.

 The boundary of a current can also be defined by mimicking and extending the smooth case (based on differential forms and Stokes' Theorem).

\begin{defn}\label{defBoundary}
 The \emph{boundary} $\partial T$ of an $n$-dimensional current $T$ is defined as
\begin{align*}
\partial T(f,\pi_1,...,\pi_{n-1})= T(1,f,\pi_1,...,\pi_{n-1}).
\end{align*}
\end{defn}

 By construction, $\varphi_{\#}(\partial T) = \partial (\varphi_{\#} T)$ and $\partial \partial T = 0$.

\begin{defn}\label{defIntegralCurrent}
 An \emph{integral current} $T$ on a complete metric space $(X,d)$ is an integer rectifiable current $T$ such that the boundary $\partial T$ is a current itself.
\end{defn}

 We now also review the definition of the set of a current which will be important to defining integral current spaces in Section~\ref{subsec: Integral Current Spaces} below.

\begin{defn}\label{defSet}
 Given an $n$-dimensional current on a complete metric space $(X,d)$ the \emph{canonical set} of $T$ is the collection of points 
\begin{align*}
 {{\set}}_d(T) = \left\{ p \in X:\liminf_{r \rightarrow 0} \frac{\| T \|_d(B^d_r(p))}{\omega_n r^n} > 0 \right\},
\end{align*}
 where $\omega_n$ is the volume of the unit ball in $n$-dimensional Euclidean space.
\end{defn}

\subsection{Integral current spaces} \label{subsec: Integral Current Spaces}
 
 Based on the work by Ambrosio and Kirchheim these spaces were introduced by Sormani and Wenger \cite{SW} and are important for defining the SWIF distance in Section~\ref{subsec: SWIF}. We also briefly discuss locally integer rectifiable currents introduced by Lang and Wenger~\cite{LW} and local integral current spaces by Jauregui and Lee~\cite{JL} in Section~\ref{subsec: Local Integral Current Spaces}, which are relevant in the noncompact setting. Our main results show that integral current structures can be pushed forward from one metric space to another via (locally) bi-Lipschitz maps.

 \begin{defn}\label{defIntegralCurrentSpaces}
 We say that $(X,d,T)$ is an \emph{integral current space} if $(X,d)$ is a metric space and $T$ is an $n$-dimensional integral current on the completion, $(\overline{X},\overline{d})$, of $X$ so that ${{\set}}_d(T)=X$.
 \end{defn}

\begin{ex}[Riemannian manifolds]\label{exRiemanniancurrent}
 The standard current structure for an oriented compact smooth manifold $M$ (with boundary) is the standard integration of $n$-forms over $M$ and thus $T$ is given by
 \[
  T(f,\pi_1,\ldots,\pi_n) = \int_M f \, d\pi_1 \wedge \ldots \wedge d\pi_n.
 \]
 Together with a Riemannian metric $\g$, and induced distance function $d_\g$, this yields an integral current space with boundary current $\partial T$ agreeing with the boundary $\partial M$ of $M$ (due to Stokes' Theorem). Note that for precompact manifolds one can work with metric completions instead.
\end{ex}

 The following theorem allows one to induce an integral current space structure on one metric space from another metric space via a bi-Lipschitz map.

\begin{thm}\label{thm:intcurrent}
 Let $(X,d_X)$ be a metric space and $(X,d_X,T)$ be an integral current space. If $(Y,d_Y)$ is a metric space and there is a bi-Lipschitz map $\psi\colon X\rightarrow Y$, then  $(Y,d_Y,\psi_{\#}T)$ is also an integral current space.
\end{thm}

\begin{proof}
 Let $(X,d_X,T)$ be an integral current space and $\psi\colon X\to Y$ a bi-Lipschitz map. By Definition~\ref{defIntegralCurrentSpaces} and Definition~\ref{Def current on metric space} there exist bi-Lipschitz maps $\varphi_i\colon U_i \to X$, $U_i \subset \R^n$ precompact Borel measurable with pairwise disjoint images and weight functions $\theta_i \in L^1(U_i,Z)$. We can induce an integer rectifiable current on $(Y,d_Y)$ by considering $\widetilde{\varphi}_i = \psi \circ \varphi_i$ in Definition~\ref{Def current on metric space}. Since $\psi$ is bi-Lipschitz we know it is a homeomorphism and hence $\widetilde{U}_i=\psi(U_i)$ are precompact Borel measurable with pairwise disjoint images. Hence we define the integral current structure on $(Y,d_Y)$ by
 \begin{align*}
 \widetilde{T} = \sum_{i=1}^{\infty} \widetilde{\varphi}_{i\#}\llbracket\theta_i\rrbracket.
 \end{align*}
 We also notice that, since $(\psi \circ \varphi_i)_\# = \psi_\# \circ \varphi_{i\#}$,
 \begin{align*}
 \psi_{\#}(T)(f,\pi_1,\ldots,\pi_n)&= T( f\circ \psi , \pi_1\circ \psi,\ldots, \pi_n\circ \psi)
 \\&= \sum_{i=1}^{\infty} \varphi_{i \#}\llbracket\theta_i\rrbracket (f\circ \psi, \pi_1\circ \psi,\ldots, \pi_n\circ \psi)
 \\&=\sum_{i=1}^{\infty} \widetilde{\varphi}_{i \#}\llbracket\theta_i\rrbracket ( f, \pi_1,\ldots, \pi_n) = \widetilde{T}(f,\pi_1,\ldots,\pi_n).
 \end{align*}
 
 Since Cauchy sequences in $X$ are mapped to Cauchy sequences in $Y$, and vice versa, $\psi$ extends as a bi-Lipschitz map $\psi \colon \overline{X} \rightarrow \overline{Y}$ between the completions of $X$ and $Y$, and hence $\widetilde{T}$ is defined on $\overline{Y}$ since $T$ is defined on $\overline{X}$ by definition.

 From this, Definition~\ref{Def pushforward}, Example~\ref{Ex current}, Remark~\ref{RemarkMass}, and the fact that measures are additive over a countable disjoint collection of sets we can deduce that
 \begin{align*}%\label{MassMeasureInequality}
 \| \widetilde{T}\|_{d_Y} = \sum_{i=1}^{\infty} \|\widetilde{\varphi}_{i \#}\llbracket\theta_i \rrbracket\|_{d_Y} \le  (\Lip(\psi))^n\sum_{i=1}^{\infty} \|\varphi_{i \#}\llbracket\theta_i\rrbracket \|_{d_X}.
 \end{align*}
 The mass of $\widetilde{T}$ is finite since
  \begin{align*}
 \mass(\widetilde{T}) &= \sum_{i=1}^{\infty} \mass(\widetilde{\varphi}_{i \#}\llbracket\theta_i\rrbracket) 
 \\&\le (\Lip(\psi))^n\sum_{i=1}^{\infty} \mass(\varphi_{i \#}\llbracket\theta_i\rrbracket) = (\Lip(\psi))^n \mass(T) < \infty,
 \end{align*}
 where we take advantage of the fact that $\mass(T)$ is finite by definition.
 
 Using the fact that boundary and pushforwards commute for currents we note that
\begin{align*}
\partial (\widetilde{T}) = \psi_{\#}(\partial T),
\end{align*}
 and hence $\partial (\widetilde{T})$ is a current and has finite mass since
 \begin{align*}
 \mass(\partial(\widetilde{T}))&\le (\Lip(\psi))^{n-1} \mass(\partial T).
\end{align*}
 Hence, $\widetilde{T}$ is an integral current on $\widetilde{Y}$.
 
 It remains to be shown that ${\set}_{d_Y}(\widetilde{T})=Y$. Since $\psi$ is bi-Lipschitz, for any $p,q \in X$
 \begin{align*}
 (\Lip(\psi^{-1}))^{-1} d_X(p,q) \le d_Y(\psi(p),\psi(q)) \le \Lip(\psi) d_X(p,q),
 \end{align*}
 which implies that for any $r>0$
 \begin{align*}
 \psi^{-1}(B^{d_Y}_{\Lip(\psi^{-1})^{-1}r}(\psi(p))) \subseteq B^{d_X}_r(p) \subseteq \psi^{-1}(B^{d_Y}_{\Lip(\psi)r}(\psi(p))).
 \end{align*}
 By combining these inclusions with Remark \ref{RemarkMass} we obtain
 \begin{align}\label{MassInequality}
 \psi^{-1}_{\#} \| \widetilde{T}\|_{d_Y} \ge (\Lip(\psi^{-1}))^{-n} \| \psi^{-1}_{\#}\widetilde{T} \|_{d_X}= (\Lip(\psi^{-1}))^{-n} \| T \|_{d_X}.
 \end{align}
 where the second equality in \eqref{MassInequality} follows from the earlier observation $\psi^{-1}_{\#}\widetilde{T}= T$.
 Putting all of this together we find
 \begin{align*}
 {\set}_{d_Y}(\widetilde{T})&= \left\{\psi(p) \in Y: \liminf_{r \rightarrow 0} \frac{\| \widetilde{T}\|_{d_Y}(B^{d_Y}_r(\psi(p)))}{\omega_nr^n} > 0\right\}
 \\&= \psi\left(\left\{p \in X: \liminf_{r \rightarrow 0} \frac{\psi^{-1}_{\#}\| \widetilde{T}\|_{d_Y}(\psi^{-1}(B^{d_Y}_r(\psi(p))))}{\omega_nr^n} > 0\right\}\right)
 \\&\supseteq \psi\left(\left\{p \in X: \liminf_{r \rightarrow 0} \frac{\| T \|_{d_X}(B^{d_X}_{\Lip(\psi)^{-1}r}(p))}{\omega_nr^n} > 0\right\}\right)
 \\&=\psi\left(\left\{p \in X: \Lip(\psi)^{-n}\liminf_{r \rightarrow 0} \frac{\| T \|_{d_X}(B^{d_X}_{\Lip(\psi)^{-1}r}(p))}{\omega_n(\Lip(\psi)^{-1}r)^n} > 0\right\} \right)
 \\&= \psi({\set}_{d_X}(T))=\psi(X)=Y,
 \end{align*}
 and hence ${\set}_{d_Y}(\widetilde{T})=Y$. Thus $(Y,d_Y,\widetilde{T})$ is an integral current space.
\end{proof}

\subsection{Local integral current spaces}\label{subsec: Local Integral Current Spaces}

 When considering spacetimes one is mostly interested in the noncompact case and hence we also introduce the notion of locally integral currents defined by Lang and Wenger~\cite{LW} and local integral current spaces defined by Jauregui and Lee~\cite{JL}. Lang and Wenger define a notion of a \emph{$n$-dimensional metric functional} $T$ which satisfies properties analogous to those in Definition \ref{Def current on metric space} but for $f \colon X \to \R$ bounded Lipschitz functions with bounded support and $\pi_i \colon X \to \R$ functions that are Lipschitz on every bounded subset of $X$ (see \cite{LW}*{Definition 2.1}). They also define a notion of mass, $\mass_V(T) \in [0,\infty]$, for $T$ on any open set $V$. In \cite{LW}*{Definition 2.2} such an $n$-dimensional metric functional $T$ on a metric space $(X,d)$ is {called an \emph{n-dimensional metric current with locally finite mass} if for every bounded, open set $V \subseteq X$ 
\begin{align}\label{locproperty1}
\mass_V(T)< \infty,
\end{align}
 and for any $\varepsilon > 0$, there exists a compact set $C \subseteq V$ so that
\begin{align}\label{locproperty2}
\mass_{V\setminus C}(T)< \varepsilon.
\end{align}
 Lang and Wenger provide analogous definitions of locally integer rectifiable currents in \cite{LW}*{Definition 2.4} and locally integral currents in \cite{LW}*{Definition 2.5}. 

 Using the results of \cite{LW}, Jauregui and Lee~\cite{JL}*{Definition 13} recently gave the following definition of local integral current spaces.

 \begin{defn}\label{deflocalIntegralCurrentSpaces}
 We say that $(X,d,T)$ is a \emph{local integral current space} if $(X,d)$ is a metric space and $T$ is an $n$-dimensional locally integral current on the completion, $(\overline{X},\overline{d})$, of $X$ so that ${{\set}}_d(T)=X$.
 \end{defn}
 
 \begin{ex}
  A complete, connected, oriented manifold with continuous Riemannian metric with the canonical choice of $T$ defined as in Example~\ref{exRiemanniancurrent} is a local integral current space \cite{LW}*{Section 2.8}.
 \end{ex}

 Locally integral currents can be pushed forward by maps that are Lipschitz on bounded sets and which also satisfy that the preimages of bounded sets are bounded \cite{LW}*{p.\ 160}. Hence we can conclude this section by showing that a local integral current space structure can be induced via maps that are bi-Lipschitz on bounded sets. 

\begin{thm}\label{thm:localintcurrent}
 Let $(X,d_X)$, $(Y,d_Y)$ be metric spaces, $(X,d_X,T)$ be a local integral current space, and $\psi\colon X \to Y$ a map which is bi-Lipschitz on bounded subsets. Then $(Y,d_Y,\psi_{\#}T)$ is also a local integral current space.
\end{thm}

\begin{proof}
 Let $(X,d_X,T)$ be a local integral current space and $\psi$ be bi-Lipschitz on bounded sets, in particular, $\psi$ and $\psi^{-1}$ are Lipschitz on metric balls.
 By the discussion in \cite{LW}*{p.\ 171f} we know that $\psi_{\#}T$ is a metric current with locally finite mass. By \cite{JL}*{Lemma 14}, for any $p \in X$ and a.e.\ $r > 0$ the restriction $T \llcorner B_r(p)$ to the open ball $B_r(p)$ is an integral current on $(\overline{X},\bar{d}_X)$. By Theorem~\ref{thm:intcurrent},  \[ \psi_{\#}T \llcorner \psi(B_r(p)) = \psi_\# (T \llcorner B_r(p))\] is then an integral current on $\overline{Y}$ for a.e.\ $r > 0$ since $\psi$ is bi-Lipschitz on all balls $B_r(p)$. In particular, those integral currents $\psi_{\#}T \llcorner \psi(B_r(p))$ can be viewed as locally integral currents on $\overline{Y}$ by  \cite{LW}*{top of p.\ 175}.  Hence  \cite{LW}*{Lemma 2.3} implies that $\psi_{\#}T$ is a locally integral current on $(\overline{Y},\bar{d}_Y)$.
 
 Finally, checking whether a point in $Y$ is in ${\set}_{d_Y}(\psi_\# T)$ is a local calculation for metric balls in $Y$ of arbitrarily small size. Then, since $\psi^{-1}$ is continuous, we know that $\psi(B_r(p))$ is open in $Y$ for any $r>0$ and hence for all $\varepsilon > 0$ sufficiently small $B_{\varepsilon}^{d_Y}(\psi(p))\subseteq \psi(B_r(p))$. Since $\psi$ is bi-Lipschitz on bounded sets of $X$ we can repeat the same argument as in Theorem \ref{thm:intcurrent} to find that ${\set}_{d_Y}(\psi_\# T) = Y$.
\end{proof}
 
 \begin{rmrk}\label{rmrk:proper}
  On proper (and hence locally compact) metric spaces locally bi-Lipschitz maps $\psi$ can be used in place of maps that are bi-Lipschitz on bounded sets in Theorem~\ref{thm:localintcurrent}. This is because on locally compact metric spaces locally Lipschitz maps are Lipschitz on compact sets, and due to the properness the closure of every bounded set is compact.
  Recall, in particular, that any locally compact complete length space is proper (see, e.g., \cite{BBI}*{Proposition 2.5.22}).
 \end{rmrk}
  
 \begin{rmrk}\label{rmrk:Lang}
  For general locally compact metric spaces related by locally bi-Lipschitz maps it should be straightforward to prove a result analogous to Theorem~\ref{thm:localintcurrent} by using an earlier notion of locally integral currents proposed by Lang~\cite{L}*{Definition 8.6}. This is since the push forward for such currents on locally compact spaces only requires locally Lipschitz maps that are proper, i.e., the preimages of compact sets must be compact (see \cite{L}*{Definition 2.6}). As pointed out in \cite{LW}, however, local compactness alone may not persist in the limit and the local currents of Lang thus seem less suitable for general pointed convergence results.
 \end{rmrk}

 It is by using Theorem \ref{thm:intcurrent} and Theorem \ref{thm:localintcurrent} (together with Remark~\ref{rmrk:proper}) that we will induce a (local) integral current space structure on warped products and globally hyperbolic spacetimes endowed with the null distance in Section \ref{sect-Integral Current Spaces}.
 
\subsection{Sormani--Wenger intrinsic flat distance}\label{subsec: SWIF}

 Intrinsic flat convergence is a notion of convergence for sequences of integral current spaces which was coined by Sormani and Wenger~\cite{SW} to handle the convergence of ``hairy'' Riemannian manifolds.  
 
 The construction of Ambrosio and Kirchheim \cite{AK} allows one to define the flat distance between currents $T_1$, $T_2$ of a metric space $Z$ by
 \begin{align*}
  d_\mathrm{F}^Z(T_1,T_2) = \inf\{\mass^n(A)+\mass^{n+1}(B) : A + \partial B = T_1 - T_2 \},
 \end{align*}
 where $A$ and $B$ are $n$- and $(n+1)$-dimensional integral currents, respectively. The Sormani--Wenger intrinsic flat distance is then defined between pairs of integral current spaces $M_1=(X_1,d_1, T_1)$ and $M_2=(X_2,d_2,T_2)$ as 
 \begin{align*}
  d_{\mathcal{F}}(M_1,M_2) = \inf\{d_\mathrm{F}^Z(\varphi_{1\#}T_1,\varphi_{2\#}T_2) : \varphi_j \colon X_j \rightarrow Z\}
 \end{align*} 
 where the infimum is taken over all complete metric spaces $Z$ and all metric isometric embeddings $\varphi_j$, that is, $\varphi_j \colon X_j \to Z$ that satisfy
 \begin{align*}
  d_Z(\varphi_j(x),\varphi_j(y)) = d_{X_j}(x,y), \qquad x,y \in X_j.
 \end{align*}

 In \cite{SW} many important properties about Sormani--Wenger intrinsic flat (SWIF, or even shorter $\mathcal{F}$) convergence are shown. Since this work in 2011 a deeper understanding of this notion of convergence has been obtained. See \cites{LS,S,SW} for many interesting examples of SWIF convergence and its relationship to GH convergence.

\subsection{Metric convergence results}\label{subsec: GH and SWIF results}
 
 When considering sequences of integral current spaces our goal will be to show bi-Lipschitz bounds on the sequence in terms of a background metric space. This then allows us to use the following compactness result of Huang, Lee, and Sormani \cite{HLS}*{Appendix}, which guarantees that a subsequence will converge in the uniform, GH and SWIF sense to the same metric space. The Gromov--Hausdorff part of this theorem was already shown by Gromov in \cite{G}.
 
\begin{thm}[\cite{HLS}*{Theorem A.1}]\label{HLS-thm}
 Let $(X, d_0, T)$ be a precompact $n$-di\-men\-sional integral current space without boundary (e.g., $\partial T=0$) and fix $\lambda>0$. Suppose that $d_j$, $j \in \N$, are metrics on $X$ such that
 \begin{align}\label{d_j}
  \frac{1}{\lambda} \le \frac{d_j(p,q)}{d_0(p,q)} \le \lambda.
 \end{align}
 Then there exists a subsequence of $(d_j)_j$, also denoted $(d_j)_j$, and a length metric $d_\infty$ satisfying (\ref{d_j}) such that $d_j$ converges uniformly to $d_\infty$, that is,
 \bee
  \varepsilon_j= \sup\left\{|d_j(p,q)-d_\infty(p,q)|:\,\, p,q\in X\right\} \to 0.
 \eee 
 Furthermore
 \bee
  \lim_{j\to \infty} d_{\mathrm{GH}}\left((X, d_j), (X, d_\infty)\right) =0,
 \eee
 and
 \bee
  \lim_{j\to \infty} d_{\mathcal{F}}\left((X, d_j,T), (X, d_\infty,T)\right) =0.
 \eee
 In particular, $(X, d_\infty, T)$ is an integral current space and $\operatorname{set}(T)=X$ so there are no disappearing sequences of points $x_j\in (X, d_j)$.

 In fact we have
 \bee
  d_{\mathrm{GH}}\left((X, d_j), (X, d_\infty)\right) \le 2\varepsilon_j,
 \eee
 and 
 \bee
  d_{\mathcal{F}}\left((X, d_j, T), (X, d_\infty, T)\right) \le 2^{(n+1)/2} \lambda^{n+1} 2\varepsilon_j \mass_{(X,d_0)}(T).
 \eee
\end{thm}

\begin{rmrk}\label{rmrk-Boundary Allowed in HLS}
 In Theorem \ref{HLS-thm} it was assumed that the integral current spaces have no boundary. We note that this assumption is not necessary and the proof given in the appendix of \cite{HLS} goes through in the boundary case as well. In the proof, a metric space $(Z_j,d_j') = (I_{\varepsilon}\times X,d_j')$, $I_{\varepsilon}=[-\varepsilon_j,\varepsilon_j]$, is constructed into which one can embed $(X,d_j)$ and $(X,d_{\infty})$ isometrically. This space is used to estimate the SWIF distance
 \begin{align*}
     d_{\mathcal{F}}\left((X, d_j, T), (X, d_\infty, T)\right) &\le d_F^{Z_j}\left(\phi_{j\#}T, \phi_{j\#}'T\right)
     \\&\le \mass_{(Z_j,d_j')}(B)+\mass_{(Z_j,d_j')}(A),
 \end{align*}
 where $\phi_j, \phi_j'$ are the embedding maps. In the proof of Theorem A.1 in \cite{HLS} $A=0$ since there was no boundary. In the case where one allows a boundary, the term $A=I_{\varepsilon}\times \partial T$ is non-trivial, however,  $\mass_{(Z_j,d_j')}(A)$ can be estimated exactly in the same way as $\mass_{(Z_j,d_j')}(B)$ was estimated because the bi-Lipschitz bounds are global. By doing so one obtains the SWIF convergence in the case where $\partial T \not = 0$. In Section \ref{sect-Convergence} we can therefore apply Theorem \ref{HLS-thm} also in the case where the integral current spaces of the sequence have a boundary.
\end{rmrk}

 Theorem \ref{HLS-thm} guarantees that the limit space will be an integral current space since the assumed bi-Lipschitz bounds persist in the limit. In fact, due to these bi-Lipschitz bounds one can use the same current structure on the entire sequence and limit space as the background integral current space. The mass of each integral current space will be different though since this depends explicitly on the distance function. 

 Our goal in the analysis of warped product spacetimes in Section~\ref{sect-Convergence} of this paper is to prove pointwise convergence of the corresponding null distances to the desired limit space, which when combined with Theorem \ref{HLS-thm} will imply uniform, GH and SWIF convergence to this limit space.

%%%%%%%%%%%%%%%%%%%%%%%%%%%%%%%%%%%%%%%%%%%%%%%%%%%%%%%%%%%%%%%%%%%%%%%%%%%%%%%%%%%%%%%%%%%%%%%%%%%%%%%%%%%%%%%%%%%%%%%%
 
\section{$(M,\hat d_\tau)$ as a length space}\label{sect-Length Spaces}
 
%%%%%%%%%%%%%%%%%%%%%%%%%%%%%%%%%%%%%%%%%%%%%%%%%%%%%%%%%%%%%%%%%%%%%%%%%%%%%%%%%%%%%%%%%%%%%%%%%%%%%%%%%%%%%%%%%%%%%%%%

 In this section we assume that the spacetime $(M,\g)$ admits a time function $\tau$ such that the induced null distance $\hat d_\tau$ is a metric. We first discuss sufficient assumptions for this to hold.
 
\subsection{Metric structure}
 
 Certain causality assumptions imply the existence of time functions. By \cite{BEE}*{Proposition 3.24} and \cite{MS}*{Theorem 3.51} (see also discussion in \cite{SV}*{Sec.\ 2.3}), a sufficient condition for the existence of a generalized time function is a past-distinguishing spacetime (e.g., a strongly causal spacetime), while a stably causal spacetime furthermore admits a (continuous) time function. Since only continuous time functions lead to a null distance that induces the manifold topology \cite{SV}*{Proposition 3.15}, we restrict our attention to those.
 
 Furthermore, the null distance $\hat d_\tau$ is definite, and hence a metric if and only if the (generalized) time function $\tau$  is locally anti-Lipschitz \cite{SV}*{Proposition 4.5, Theorem 4.6}. Alternatively, definiteness follows if $\hat d_\tau$ encodes causality (see \cite{SV}*{Lemma 3.12} and Section~\ref{sect-Encoding Causality}).
  
\begin{defn}
 Let $(M,\g)$ be a spacetime and $f\colon M \to \R$. For any $U \subseteq M$ we say that $f$ is \emph{anti-Lipschitz on $U$} if there exists a (definite) distance function $d_U$ on $U$ so that for any $p,q\in U$ we have
 \begin{align}
 p \le q \Longrightarrow f(q) - f(p) \ge d_U(p,q).
 \end{align}
 We say that $f\colon M \to \R$ is \emph{locally anti-Lipschitz} if $f$ is anti-Lipschitz on a neighborhood $U$ of each point $p \in M$.
\end{defn}
 
\begin{rmrk}
  In order to obtain definiteness of $\hat d_\tau$ it is sufficient to simply demand the anti-Lipschitz property with respect to a \emph{any} metric $d_U$ on $M$ (see \cite{SV}*{Sec.\ 4.1}), however, for other global properties of $(M,\hat d_\tau)$ a background Riemannian structure may be more useful. If one (continuous) Riemannian metric exists with respect to which $\tau$ is locally anti-Lipschitz function, then it is automatically locally anti-Lipschitz with respect to \emph{any} Riemannian metric due to the fact that the associated distance functions are equivalent on compact sets \cite{B}*{Theorem 4.5}.
\end{rmrk}

 We recall an important result by Sormani and Vega.

\begin{thm}[\cite{SV}*{Theorem 4.6}]
 Let $M$ be a spacetime and $\tau$ a time function on $M$. If $\tau$ is locally anti-Lipschitz, then the induced null distance $\hat d_\tau$ is a definite and conformally invariant metric on $M$ that induces the manifold topology.
\end{thm}

\begin{rmrk}
 Time functions $\tau$ whose gradient vectors $\nabla \tau$ are defined almost everywhere and which are locally bounded away from the light cone \cite{SV}*{Def.\ 4.13} are anti-Lipschitz, and so is the regular cosmological time function of Andersson, Galloway and Howard~\cite{AGH}.
\end{rmrk}

\begin{rmrk}
Due to the conformal invariance of the null distance we note that many of the theorems of this paper are automatically true in the conformal to warped product case as well.
\end{rmrk}

\subsection{Length structure}
 
 We show that the null length $\hat L_\tau$ on the set $\A$ of piecewise (smooth) causal paths defines a length structure on $M$, and thus $M$ together with the null distance $\hat d_\tau$ is a length space. We follow the conventions of Burago, Burago and Ivanov \cite{BBI}.
 
\begin{customthm}{\ref{thm:lengthspace}}
 Let $(M,\g)$ be a spacetime with locally anti-Lipschitz time function $\tau$. Then $(M,\A,\hat L_\tau)$ defines a \emph{length structure} on $M$, i.e., $\A$ is an admissible class of paths according to \cite{BBI}*{Sec.\ 2.1.1} and
 \begin{enumerate}
   \item The length of paths is additive, i.e., $\hat L_\tau(\beta|_{[a,b]}) = \hat L_\tau (\beta|_{[a,c]}) + \hat L_\tau (\beta|_{[c,b]})$ for any $c \in [a,b]$.
   \item The length of a piece of a path continuously depends on the piece, i.e., for a fixed path $\beta \colon [a,b] \to M$ the function $t \mapsto \hat L_\tau (\beta|_{[a,t]})$ is continuous.
   \item The length is invariant under reparametrizations, i.e., $\hat L_\tau (\beta \circ \varphi) = \hat L_\tau (\beta)$ for any diffeomorphism $\varphi$ such that $\beta,\beta \circ \varphi \in \A$.
   \item The length structure agrees with the topology, i.e. for $U_p$ a neighborhood of $p$ there is
   $\inf \{ \hat L_\tau(\beta)  :  \beta(a) = p, \beta(b) \in M \setminus U_p \} > 0$.
 \end{enumerate}
 Thus $\hat d_\tau$ is the intrinsic metric of $M$ with respect to this length structure, and hence $(M,\hat d_\tau)$ is a \emph{length space}. % see \cite{BBI}*{Def.\ 2.1.6}
\end{customthm}
 
\begin{proof}
 It is clear that the class $\A$ of piecewise causal paths is closed under restrictions, concatenations and (smooth) reparametrizations. Thus $\A$ is admissible.
  
 (i) Suppose $\beta \colon [a,b] \to M$ is a piecewise causal curve, i.e., on the pieces $a=s_0 < s_1 < \ldots < s_k =b$ the path is smooth and either future- or past-directed causal and $\hat L_\tau (\beta) = \sum_{i=1}^k |\tau(\beta(s_i))-\tau(\beta(s_{i-1}))|$. If $c=s_j$, then the result is trivial, so assume that $c \in (s_{j-1},s_j)$. Since $\beta$ is either future- or past-directed on $[s_{j-1},s_j]$ it is either future- or past-directed on both pieces $s_{j-1} < c < s_j$ and the result follows by artificially introducing a breaking point $c$ and subsequent cancellation of $\tau(\beta(c))$ in the null length.

 (ii) This follows immediately from the continuity of $\tau$ (and $\beta$).

 (iii) This is clear because a diffeomorphism maps an interval to another interval in a bijective smooth way. A change of orientation does not matter because the length is computed using the absolute value.

 (iv) Suppose $p \in M$, $U_p$ is a neighborhood of $M$ and $q \in M \setminus U_p$. By \cite{SV}*{Theorem 4.6}, $\hat d_\tau$ then encodes the manifold topology. Hence there exists an $r>0$ such that $\hat B_r(p) = \{ x \in M : \hat d_\tau (p,x) < r \} \subseteq U_p$, and therefore $\inf \{ \hat L_\tau(\beta) : \beta(a) = p, \beta(b) \in M \setminus U_p \} \geq r > 0$.
\end{proof}

\begin{rmrk}\label{rem:counterex}
 Note that property (ii) in Theorem~\ref{thm:lengthspace} does not hold for discontinuous generalized time functions (see Example \ref{DiscontinuousLengthOfCurves} and Example \ref{DistanceNotAchieved} below). Hence only time functions $\tau$ yield a length structure. On the other hand, property (iv) requires an interaction with the manifold topology, so $\hat d_\tau$ needs to be furthermore definite, that is, a metric (see Example \ref{TimeFunctionNotAntiLipschitz}). To express the metric property of $\hat d_\tau$ in terms of $\tau$ we have assumed that $\tau$ is furthermore locally anti-Lipschitz.
\end{rmrk}
 
\begin{rmrk}
 Note that for every globally hyperbolic manifold with $C^{2,1}$ metric a continuously differentiable, locally anti-Lipschitz time functions exists due to Chru\'{s}ciel, Grant and Minguzzi \cite{CGM}*{Theorem 1.1, Cor.\ 4.4}. Furthermore, for every spacetime metric with a time function there exists an $\varepsilon$-close locally anti-Lipschitz time function \cite{CGM}*{Proposition 5.3}. For a more general setup using cone fields see the work of Bernard and Suhr \cite{BSu}.
\end{rmrk}

 By construction $\hat d_\tau$ is an intrinsic metric. Thus it coincides with the metric induced by the length structure $L_{\hat d_\tau}$ induced by $\hat d_\tau$.
 
\begin{prop}\label{prop:distances}
 Let $M$, $\g$ and $\tau$ be as in Theorem~\ref{thm:lengthspace}. Then the null distance $\hat d_\tau$ is an \emph{intrinsic metric}, that is, it coincides with the metric
 \[
   d(p,q) = \inf \left\{ L_{\hat d_\tau} (\gamma) : \gamma \text{ is a rectifiable curve between } p \text{ and } q \right\},
 \]
 induced by the length structure
 \[
   L_{\hat d_\tau} (\gamma) = \sup \bigg\{ \sum_{i=1}^k \hat d_\tau (\gamma(s_i),\gamma(s_{i-1})) : a=s_0 < s_1 < \ldots < s_k = b \bigg\}.
 \]
 Moreover, $L_{\hat d_\tau} = \hat L_\tau$ holds on the set $\A$ of piecewise causal curves.
\end{prop}
 
\begin{proof}
 We follow the proof of \cite{BBI}*{Proposition 2.4.1}.
 Let $\beta \colon [a,b] \to M$ be a path in $\A$ and $a=s_0 < s_1 < \ldots < s_k = b$ be an arbitrary partition of $[a,b]$. By definition of the null distance, for each $i =1,\ldots,k$,
 \[
  \hat d_\tau(\beta(s_i),\beta(s_{i-1})) \leq L_\tau (\beta|_{[s_{i-1},s_i]}),
 \]
 and hence we obtain
 \begin{align*}
  \sum_{i=1}^k \hat d_\tau (\beta(s_i),\beta(s_{i-1})) \leq \sum_{i=1}^k \hat L_\tau(\beta|_{[s_{i-1},s_i]}) = \hat L_\tau(\beta)
 \end{align*}
 due to the additivity of null lengths obtained in Theorem~\ref{thm:lengthspace} (i). Since the partition was arbitrary, the inequality holds also for the supremum, i.e.,
 \[
  L_{\hat d_\tau}(\beta) \leq \hat L_\tau (\beta).
 \]
 A smaller length function immediately implies a smaller metric, i.e., $d \leq \hat d_\tau$.
 
%  On the other hand, 
%  
%  It remains to show equality of lengths for admissible curves $\beta$.
%   
%  As for any length space, $L_{\hat d_\tau}(\beta) \leq \hat L_\tau (\beta)$ by construction. 
  
  The converse inequality $\hat d_\tau \leq d$ holds automatically since for any $p,q \in M$ and any rectifiable curve $\gamma$ between these endpoints %and any $\beta \in \A$ connecting them (which exists by \cite{SV}*{Lemma 3.5}) thus
  \begin{align*}
   \hat d_\tau(p,q) \leq L_{\hat d_\tau}(\gamma).
  \end{align*}
  Thus the metrics $\hat d_\tau$ and $d$ coincide.
  
  It remains to show equality of lengths for curves $\beta \in \A$. If $a = s_0 < s_1 < \ldots < s_k =b$ are the future- and past-directed causal pieces  of $\beta$, we have that
  \begin{align*}
   \hat L_\tau(\beta) &= \sum_{i=1}^k |\tau(\beta(s_i)) - \tau(\beta(s_{i-1}))| \\
    &= \sum_{i=1}^k \hat d_\tau (\beta(s_i),\beta(s_{i-1})) \leq L_{\hat d_\tau} (\beta).
  \end{align*}
  This proves the final statement on the equality of lengths for piecewise causal curves $\beta$.
\end{proof}

 To conclude our discussion of Theorem~\ref{thm:lengthspace} we provide simple (counter)examples mentioned in Remark~\ref{rem:counterex} for more general time functions in 2-dimensional Minkowski space $\M^{1+1} = \R^{1,1}$.

\begin{ex}[Discontinuous time function not satisfying \ref{thm:lengthspace}(iii)]\label{DiscontinuousLengthOfCurves}
 Consider $\M^{1+1}$ with discontinuous time function 
 \[
  \tau(t)= \begin{cases}
              t+1 & t > 0
            \\0 & t = 0
            \\t-1 & t < 0
 \end{cases}
 \]
 then any causal curve, $\beta: [a,b] \rightarrow \M^{1+1}$, connecting a point $p \in \{(x,t) \in \M^{1+1}: t < 0\}$ to a point $q\in \{(x,t) \in \M^{1+1}: t > 0\}$ will be such that $t \mapsto \hat{L}_{\tau}(\beta|_{[a,t]})$ is not continuous. We also observe that $\hat{d}_{\tau}$ is not complete since the sequence $p_i = \left(0,-\frac{1}{i}\right)$ is a Cauchy sequence since
 \begin{align}
  \hat{d}_{\tau}(p_i,p_j) = \left| \frac{i-j}{ij}\right|,
 \end{align}
 which does not converge to $p=(0,0)$ since $\hat{d}_{\tau}(p_i,p) = 1+\frac{1}{i}$.
\end{ex}

\begin{ex}[Non-anti-Lipschitz time function]\label{TimeFunctionNotAntiLipschitz}
 Consider $\M^{1+1}$ with time function 
 \[
  \tau(t)=t^3
 \]
 which is continuous but not anti-Lipschitz. Sormani--Vega \cite{SV}*{Proposition 3.4} have shown that the resulting null distance $\hat{d}_{\tau}$ is not definite and does not encode causality in this case. We just also note that this metric is not a length space either since the resulting length structure would not agree with the manifold topology.
\end{ex}

 Since $\hat L_\tau = L_{\hat d_\tau}$ on $\A$ by Proposition~\ref{prop:distances}, and $L_{\hat d_\tau}$ is lower semicontinuous on the larger set of rectifiable paths, by \cite{BBI}*{Proposition 2.3.4} it follows immediately that $\hat L_\tau$ is lower semicontinuous on $\A$. 

\begin{lem}
 Let $\tau$ be a time function on the Lorentzian manifold $(M,\g)$. The null length functional $\hat L_\tau$ is lower semicontinuous, that is, for any sequence $(\beta_i)_i$ of piecewise causal path that converge pointwise to a piecewise causal path $\beta$ we have that
 \[
 \pushQED{\qed} 
  \hat L_\tau (\beta) \leq \liminf_{i \to \infty} \hat L_\tau(\beta_i). \qedhere
  \popQED
 \]
\end{lem}

 One can construct selfsimilar ``fractal'' counterexamples in order to show that upper semicontinuity does \emph{not} hold in general.

\begin{ex}[Timelike counterexample]\label{ex:timelikecounterexample}
 We define piecewise causal curves $\beta_i \colon [0,1] \to \M^{1+1}$ recursively by the graphs of
 \begin{align*}
  \tilde \beta_1(s) = \begin{cases}
                  3s & s \in [0,\frac{1}{2}], \\
                  2-s & s \in [\frac{1}{2},1],
                 \end{cases}
 \end{align*}
 and
 \begin{align*}
  \tilde \beta_{i+1}(s) = \begin{cases}
                \frac{1}{2} \tilde \beta_i(2s) & s \in [0, \frac{1}{2}], \\
                \frac{1}{2} \tilde \beta_i(2s-1) + \frac{1}{2} & s \in [\frac{1}{2}, 1].
               \end{cases}
 \end{align*}
 See also Figure~\ref{fig:timelikecounterexample}. It follows that for all $i \in \N$,
 \[
  \hat L_\tau (\beta_{i+1}) = \hat L_\tau (\beta_i) = \ldots = \hat L_\tau(\beta_1) = 2.
 \]
 On the other hand, the limit curve $\beta(s) = s$ satisfies $\hat L_\tau (\beta) = 1 < 2$.
\end{ex}

\begin{figure}[ht]
 \centering
 \begin{tikzpicture}[scale=4]
  \draw[->] (0,0) -- (1.3,0) node[anchor=north west] {$x$};
  \draw[->] (0,0) -- (0,1.7) node[anchor=south east] {$t$};
  \draw[dotted] (0,0) -- (1,1);
   \draw[fill] (0,0) circle [radius=0.025];
   \node[left, outer sep=2pt] at (0,0) {$(0,0)$};
   \draw[fill] (1,1) circle [radius=0.025];
   \node[right, outer sep=2pt] at (1,1) {$(1,1)$};
   \draw[thick] (0,0) -- (1/2,3/2) -- (1,1);
   \node[left, outer sep=2pt] at (3/8,9/8) {$\beta_1$};
   \draw[ultra thick] (0,0) -- (1/4,3/4) -- (1/2,1/2) -- (3/4,5/4) -- (1,1);
   \node[left, outer sep=2pt] at (23/32,9/8) {$\beta_2$};
   \draw[thick,dashed] (1/8,3/8) -- (1/4,1/4) -- (3/8,5/8) -- (1/2,1/2) -- (5/8,7/8) -- (3/4,3/4) -- (7/8,9/8); 
   \node[left, outer sep=2pt] at (11/32,4/8) {$\beta_3$};
 \end{tikzpicture}
 \caption{A sequence of piecewise timelike curves $\beta_i$ in $\M^{1+1}$ converges uniformly to a null curve $\beta$ such that $\liminf_{i \to \infty} \hat L_\tau(\beta_i) > \hat L_\tau(\beta)$ in Example~\ref{ex:timelikecounterexample}.}
 \label{fig:timelikecounterexample}
\end{figure}
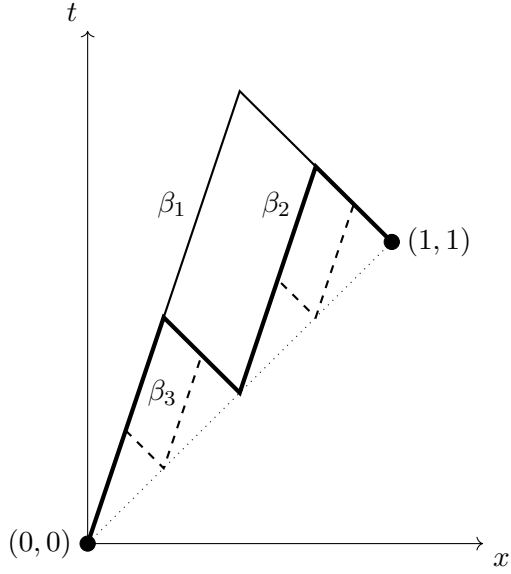

\begin{ex}[Null counterexample]\label{ex:nullcounterexample}
 In a similar fashion one can define a sequence of piecewise null curves $\beta_i$ on $\M^{1+1}$ that uniformly converge to the curve $\beta(s) = s$ so that
 \[
  \hat L_\tau (\beta_i) = 5 > 1 = \hat L_\tau(\beta).
 \]
 See Figure~\ref{fig:nullcounterexample}.
\end{ex}

\begin{figure}[ht]
 \centering
 \begin{tikzpicture}[scale=4]
  \draw[->] (-1.3,0) -- (1.3,0) node[anchor=north west] {$x$};
  \draw[->] (0,0) -- (0,2.2) node[anchor=south east] {$t$};
  \draw[dotted] (0,0) -- (1,1);
  \draw[fill] (0,0) circle [radius=0.025];
  \node[below, outer sep=2pt] at (0,0) {$(0,0)$};
  \draw[fill] (1,1) circle [radius=0.025];
  \node[right, outer sep=2pt] at (1,1) {$(1,1)$};
  \draw[thick] (0,0) -- (-1,1) -- (-2/3,4/3) -- (1/3,1/3) -- (2/3,2/3) -- (-1/3,5/3) -- (0,2) -- (1,1);
  \node[above, outer sep=2pt] at (-1+1/6,7/6) {$\beta_1$};
  \draw[ultra thick] (0,0) -- (-1/2,1/2) -- (-1/2+1/9,1/2+1/9) -- (1/9,1/9) -- (2/9,2/9) -- (2/9-1/2,2/9+1/2) -- (1/3-1/2,1/3+1/2) -- (1/3,1/3) -- (2/3,2/3) -- (-1/2+2/3,1/2+2/3) -- (-1/2+1/9+2/3,1/2+1/9+2/3) -- (1/9+2/3,1/9+2/3) -- (2/9+2/3,2/9+2/3) -- (2/9-1/2+2/3,2/9+1/2+2/3) -- (1/3-1/2+2/3,1/3+1/2+2/3) -- (1,1);
   \node[above, outer sep=2pt] at (-1/3,2/3) {$\beta_2$};
   \draw[thick, dashed] (0,0) -- (-1/4,1/4)-- +(1/27,1/27) -- (1/27,1/27) -- + (1/27,1/27) -- (2/27-1/4,2/27+1/4) -- + (1/27,1/27) -- (1/9,1/9) -- (2/9,2/9) -- (2/9-1/4,2/9+1/4)-- +(1/27,1/27) -- (7/27,7/27) -- + (1/27,1/27) -- (8/27-1/4,8/27+1/4) -- + (1/27,1/27) -- (1/3,1/3) -- (2/3,2/3) -- (2/3-1/4,2/3+1/4)-- +(1/27,1/27) -- (2/3+1/27,2/3+1/27) -- + (1/27,1/27) -- (2/3+2/27-1/4,2/3+2/27+1/4) -- + (1/27,1/27) -- (2/3+1/9,2/3+1/9) -- (2/3+2/9,2/3+2/9) -- (2/3+2/9-1/4,2/3+2/9+1/4)-- +(1/27,1/27) -- (2/3+7/27,2/3+7/27) -- + (1/27,1/27) -- (2/3+8/27-1/4,2/3+8/27+1/4) -- + (1/27,1/27) -- (1,1);
   \node[above, outer sep=2pt] at (-1/4,1/4) {$\beta_3$};
 \end{tikzpicture}
 \caption{A sequence of piecewise null curves $\beta_i$ in $\M^{1+1}$ converges uniformly to a null curve $\beta$ such that $\liminf_{i \to \infty} \hat L_\tau(\beta_i) > \hat L_\tau(\beta)$ in Example~\ref{ex:nullcounterexample}.}
 \label{fig:nullcounterexample}
\end{figure}
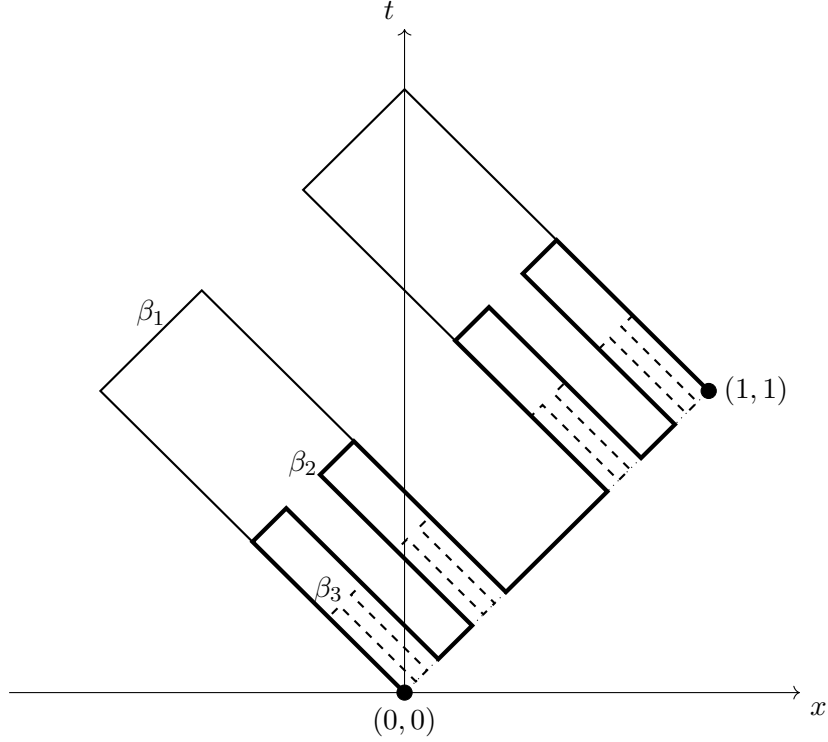

\subsection{Completeness}\label{subsec: Completeness}
 
 Length spaces with complete metrics have additional features.
 We provide conditions on the Lorentzian manifold and the continuous time function $\tau$ which imply that $(M,\hat d_\tau)$ is a complete metric space. 

 \begin{prop}\label{Conditions Implying Null Complete}
 Let $(M,\g)$ be a spacetime with time function $\tau$. If there exists a complete background metric $d$ on $M$ (that induces the manifold topology) and a constant $C>0$ such that
 \[
  \hat d_\tau(p,q) \geq C d (p,q), \qquad p,q \in M,
 \]
 then $(M,\hat{d}_{\tau})$ is complete.
\end{prop} 

\begin{proof} 
 Let $(p_i)_i$ be a Cauchy sequence of points in $M$ with respect to $\hat d_\tau$. By assumption,
 \[
  \hat d_\tau (p_i,p_j) \geq C d(p_i,p_j),
 \]
 where $d$ is the distance function associated with a complete metric $d$. In particular, since $(p_i)_i$ is also a Cauchy sequence with respect to $d$ there exists $p \in M$ so that
 \[
  \lim_{i\to\infty}  d (p_i,p) =0.
 \]
 It remains to be shown that $\displaystyle\lim_{i\to\infty}\hat d_\tau(p_i,p) = 0$ to prove completeness of $M$ with respect to $\hat d_\tau$.
 
 Let $\varepsilon >0$. Since both $d$ and $\hat d_\tau$ are metrics that induce the manifold topology (in the latter case it is due to \cite{SV}*{Proposition 3.15 and Proposition 4.5}) there exists a $\varepsilon'_\varepsilon\in (0,\varepsilon)$ such that $B^{\hat d_\tau}_{\varepsilon'}(p) \subseteq B^{d}_\varepsilon(p)$, and a $\varepsilon''_\varepsilon>0$ such that
 \[
  B^{d}_{\varepsilon''}(p) \subseteq B^{\hat d_\tau}_{\varepsilon'}(p) \subseteq B^{d}_\varepsilon(p).
 \]
 Since $p_i \to p$ with respect to $d$, for all $i$ sufficiently large we have that $p_i \in B^{d}_{\varepsilon''}(p) \subseteq B^{\hat d_\tau}_{\varepsilon'}(p)$. In particular, for $i$ sufficiently large we have $p_i \in B^{\hat d_\tau}_{\varepsilon}(p)$ since $\varepsilon' < \varepsilon$, and therefore
 \[
  \lim_{i\to\infty} \hat d_\tau (p_i,p) = 0. \qedhere
 \]
\end{proof}

 We now relate the anti-Lipschitz time function and background metric assumptions of Proposition \ref{Conditions Implying Null Complete} more directly.

\begin{customthm}{\ref{Conditions Implying Null Complete 2}}
 Let $(M,\g)$ be a spacetime with time function $\tau$. If $\tau$ is anti-Lipschitz on $M$ with respect to a complete metric $d$ (that induces the manifold topology), then $(M,\hat d_\tau)$ is complete. 
\end{customthm}

\begin{proof}
 We show that the assumptions of Proposition \ref{Conditions Implying Null Complete} are satisfied. If $p$ and $q$ are causally related, then 
 \begin{align*}
 |\tau(p) - \tau(q)| = \hat d_\tau (p,q).
 \end{align*}
 Otherwise, for every $\varepsilon>0$, we can choose a piecewise null curve $\beta$ so that
 \begin{align*}
    \hat{L}(\beta) \le \hat d_\tau(p,q) + \varepsilon,
 \end{align*}
 with breaking points $\beta(s_i)=p_i$ so that
 \begin{align*}
  \hat d_\tau(p,q) \geq \hat{L}(\beta)-\varepsilon = \sum_i \hat d_\tau (p_i,p_{i_1})- \varepsilon 
  = \sum_i |\tau(p_i)-\tau(p_{i-1})| - \varepsilon.
 \end{align*}
 Since $\tau$ is globally anti-Lipschitz with respect to $d$ we find
 \begin{align*}
  \hat{d}_{\tau}(p,q)  &> C \sum_i d(p_i,p_{i-1}) - \varepsilon
  \geq C d(p,q) - \varepsilon.
 \end{align*}
 Since this inequality holds for all $\varepsilon$, we have
 \begin{align*}
    \hat{d}_{\tau}(p,q)   \geq C d(p,q)
 \end{align*}
 and can apply Proposition \ref{Conditions Implying Null Complete} to finish the proof.
\end{proof}
 
 The Hopf--Rinow Theorem implies the following result.
 
 \begin{cor}\label{cor:complete}
  Let $(M,\g)$ be a spacetime with time function $\tau$. If $\tau$ is anti-Lipschitz on $M$ with respect to a complete Riemannian metric $\h$, then $(M,\hat d_\tau)$ is complete.\qedhere
 \end{cor}

\begin{rmrk}
 Recall that the standard Hopf--Rinow Theorem does not hold for Lorentzian manifolds (see the example of the Clifton--Pohl torus in \cite{O}*{p.\ 193}). The Hopf--Rinow--Cohn-Vossen Theorem for length spaces (see, for instance, \cite{BBI}*{Theorem 2.5.28} or \cite{G}*{Section 1.9}) implies that if the length space $(M,\hat d_\tau)$ is complete then the Heine--Borel property holds (every closed and bounded set is compact) and that between any two points a shortest rectifiable curves exists. Examples \ref{DistanceNotAchievedWarped} and \ref{DistanceNotAchieved} show that completeness does not guarantee that the distance is achieved by a piecewise causal curve.
\end{rmrk}
 
\begin{ex}[Warped product with distance not achieved in $\A$]\label{DistanceNotAchievedWarped}
 Con\-sid\-er the Lorentzian warped product $\R \times_f \R$ with $f(t) =t^2+1$, i.e.,
\begin{align*}
    \g=-dt^2+(t^2+1)^2 dx^2.
\end{align*}
 We argue that the null distance between $p = (x,t) = (-1,0)$ and $q = (1,0)$ is $\hat{d}_{\tau}(p,q) = 2$, and that it is not achieved by any piecewise causal curve. Indeed, this is true for any two points at the $\{ t=0 \}$ level.
 
 In order to see that $\hat d_\tau(p,q) = 2$ we refer to Proposition~\ref{prop:biLip}, which implies that on each time interval $[-t,t]$ we have the upper and lower bounds for the null distance $\hat d_f = \hat d_t$ of the warped product compared to the null distance $\hat d_\M$ of the 2-dimensional Minkowski space $\M$ given by
 \[
  \hat d_\M (p,q) \leq \hat d_f(p,q) \leq (t^2 +1) \hat d_\M(p,q).
 \]
 Since $t$ can be made arbitrarily small if we restrict our attention to piecewise causal curves within $[-\frac{1}{i},\frac{1}{i}] \times \R$ (see Figure~\ref{fig:DistanceNotAchievedWarped} for a visualization of null curves emanating from $t=0$), we obtain for $p=(-1,0)$ and $q=(1,0)$ that
 \[
  \hat d_f(p,q) = \hat d_\M(p,q) = 2.
 \]
 
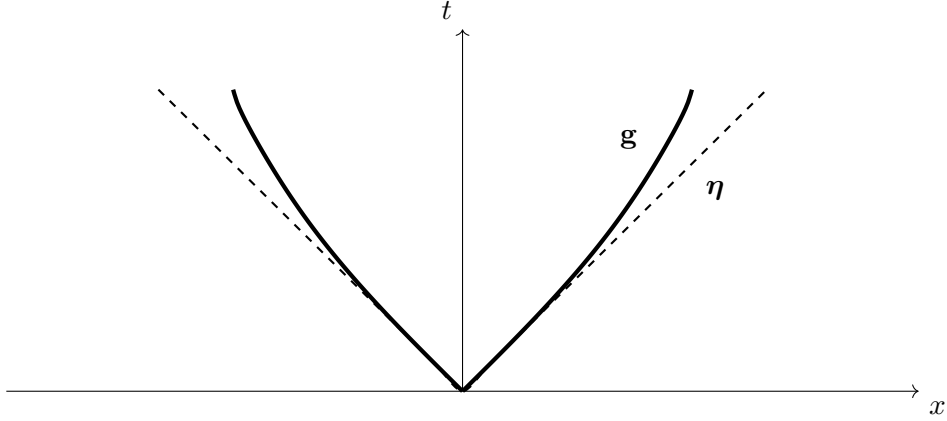
\begin{figure} [h]
  \centering
 \begin{tikzpicture}[scale=4]
  \draw[->] (-1.5,0) -- (1.5,0) node[anchor=north west] {$x$};
  \draw[->] (0,0) -- (0,1.2) node[anchor=south east] {$t$};
   \draw[domain=0:1,smooth,variable=\t,ultra thick] plot ({\t - \t^3/3+\t*\t^5/5-\t^7/7+\t^9/9-\t^11/11+\t^13/13-\t^15/15},{\t} );
   \draw[domain=0:1,smooth,variable=\t,ultra thick] plot ({-\t + \t^3/3-\t*\t^5/5+\t^7/7-\t^9/9+\t^11/11-\t^13/13+\t^15/15},{\t} );
   \node[left, outer sep=2pt] at (3/4-1/8,5/6) {$\g$};
   \draw[domain=-1:1,smooth,variable=\t,dashed,thick] plot ({\t},{abs(\t)} );
   \node[right, outer sep=2pt] at (3/4,2/3) {$\etab$};
 \end{tikzpicture}
  \caption{Light cones with respect to $\g$ and the Minkowski metric $\etab$ in Example~\ref{DistanceNotAchievedWarped}.}
  \label{fig:DistanceNotAchievedWarped}
 \end{figure}
 
 % the null curves are of the form $\beta(t) = (t,\arctan(t))$
 
This distance cannot be achieved by any piecewise causal curve, because any such curve $\beta$ leaves the $\{t=0\}$ line and beyond the $\{t = 0 \}$ line the light cones of Minkowski space are strictly larger than that of $\g$.
 
Note that, however, the metric $\g$ still encodes causality by Theorem \ref{thm:warpedcausality} of Sormani and Vega \cite{SV} stated in Section \ref{sect-Encoding Causality}. (We would like to thank Christina Sormani for sharing this example with us.)
\end{ex}
 
\begin{ex}[Minkowski space and distance not achieved in $\A$]\label{DistanceNotAchieved}
 Consider $\M^{1+1}$ with discontinuous time function 
 \[
 \tau(t)= \begin{cases}
            \sqrt{t}+1 &  t> 1
         \\ \sqrt{t} & 0\le t \le 1
         \\t & t \le 0
          \end{cases}
 \]
 then the distance between $p = (-1,1)$ and $q = (1,1)$, $\hat{d}_{\tau}(p,q) = 1$, is not achieved by any piecewise causal curve. One can define a sequence of piecewise causal curves $\beta_i \colon [-1,1] \to \M^{1+1}$ connecting $p$ and $q$ as graph of $\tilde \beta_i$ (see Figure \ref{fig:CurveSeqDiscts}), given inductively by
 \begin{align*}
  \tilde \beta_1 (s) &= |s|, \\
  \tilde \beta_i (s) &= \begin{cases}
                  \frac{1}{2}+\frac{1}{2} \tilde \beta_{i-1} (2s+1) & s \in [-1,0], \\
                  \frac{1}{2}+\frac{1}{2} \tilde \beta_{i-1} (2s-1) & s \in [0,1].
                 \end{cases}
 \end{align*}
 Then $\hat L_\tau (\beta_1) = 2$ and
\begin{align*}
 \hat{L}_{\tau}(\beta_i) &= 2^i \left( 1 - \frac{(2^{i-1}-1)^{\frac{1}{2}}}{2^{\frac{i-1}{2}}}\right) \\
                         &= 2^{\frac{i+1}{2}} \left( 2^{\frac{i-1}{2}} - (2^{i-1}-1)^{\frac{1}{2}}\right) \to 1
\end{align*}

 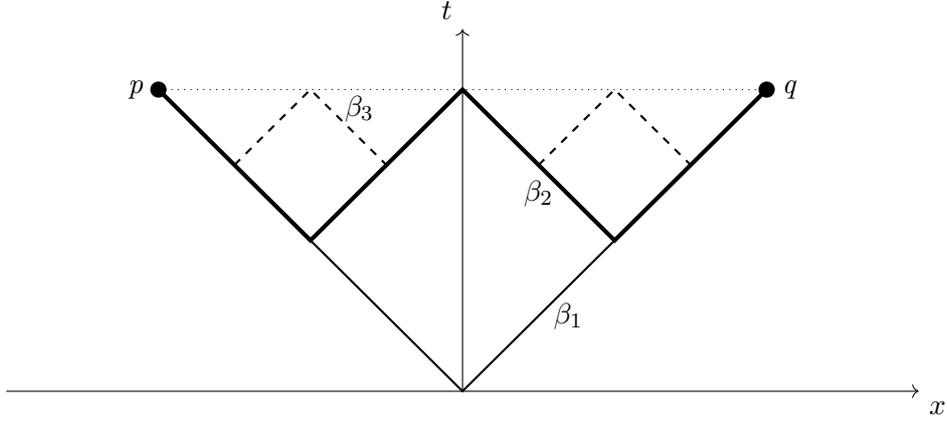
\begin{figure} [h]
  \centering
 \begin{tikzpicture}[scale=4]
  \draw[->] (-1.5,0) -- (1.5,0) node[anchor=north west] {$x$};
  \draw[->] (0,0) -- (0,1.2) node[anchor=south east] {$t$};
  \draw[dotted] (-1,1) -- (1,1);
  \draw[fill] (-1,1) circle [radius=0.025];
  \node[left, outer sep=2pt] at (-1,1) {p};
  \draw[fill] (1,1) circle [radius=0.025];
  \node[right, outer sep=2pt] at (1,1) {q};
  \draw[thick] (-1,1) -- (0,0) -- (1,1);
  \node[right, outer sep=2pt] at (1/4,1/4) {$\beta_1$};
  \draw[ultra thick] (-1,1) -- (-1/2,1/2) -- (0,1) -- (1/2,1/2) -- (1,1);
  \node[below, outer sep=2pt] at (1/4,3/4) {$\beta_2$};
  \draw[thick,dashed] (-3/4,3/4) -- (-1/2,1) -- (-1/4,3/4) -- (0,1)(1/4,3/4) -- (1/2,1) -- (3/4,3/4); 
  \node[right, outer sep=2pt] at (-7/16,15/16) {$\beta_3$};
 \end{tikzpicture}
  \caption{A sequence of piecewise causal curves $\beta_i$ between $p$ and $q$ showing that the null distance is not achieved in Example~\ref{DistanceNotAchieved}.}
  \label{fig:CurveSeqDiscts}
 \end{figure}

 Note that if, on the other hand, $\bar{\tau}(t) = \sqrt{t}$ for all $t \ge 0$ then $\hat{d}_{\bar{\tau}}(p,q) = 2(\sqrt{2}-1) < 1$ which is achieved by a piecewise causal curve connecting $p $ to $(0,2)$ and subsequently to $q$. In the initial example, the discontinuity of $\tau$ forces the sequence of distance-approximating curves to eventually lie below $t = 1$. The distance cannot be achieved because the limit curve $\beta(s)=(s,1)$ is not piecewise causal. Note that the time function $\tau$ is locally anti-Lipschitz and hence the issue is truly caused by the discontinuity.
\end{ex}

\begin{rmrk}
 We conjecture that having a (possibly locally anti-Lipschitz but) generalized time function $\tau$ that is discontinuous along any causal curve already implies that distances are not achieved by piecewise causal curves in $(M,\hat d_\tau)$.
\end{rmrk}

%%%%%%%%%%%%%%%%%%%%%%%%%%%%%%%%%%%%%%%%%%%%%%%%%%%%%%%%%%%%%%%%%%%%%%%%%%%%%%%%%%%%%%%%%%%%%%%%%%%%%%%%%%%%%%%%%%%%%%%%
 
\subsection{Causality-encoding property} \label{sect-Encoding Causality}

%%%%%%%%%%%%%%%%%%%%%%%%%%%%%%%%%%%%%%%%%%%%%%%%%%%%%%%%%%%%%%%%%%%%%%%%%%%%%%%%%%%%%%%%%%%%%%%%%%%%%%%%%%%%%%%%%%%%%%%% 

 One of the major advantages of the null distance over auxiliary metrics (for instance, induced via a background Riemannian metric) is its relationship to the causal structure of a spacetime, a feature already investigated in \cite{SV}. In this Section, we recall the concept and discuss known results and counterexamples.
 
\begin{defn}
 Let $(M,\g)$ be a spacetime with generalized time function $\tau$. We say that the null distance $\hat{d}_{\tau}$ \emph{encodes causality} if
 \begin{align}
  p \le q \Longleftrightarrow \hat{d}_{\tau}(p,q) = \tau(p) - \tau(q).
 \end{align}
\end{defn}
 
 In some cases, for instance, Minkowski space $\M^{n+1} = \R^{n,1}$ with time function $\tau = t$ and smooth warped products, it was already shown by Sormani and Vega that causality is encoded. Since we are mostly interested in warped product spacetimes, we recall this result here.
 
\begin{customthm}{\ref{thm:warpedcausality}}[\cite{SV}*{Theorem 3.25}]
 Let $I$ be an interval, $f \colon I \to (0,\infty)$ smooth and $(\Sigma,\bsigma)$ a complete Riemannian manifold. Let $M = I \times_f \Sigma$ be the warped spacetime with metric tensor
 \[
  \g = -dt^2 + f(t)^2 \bsigma,
 \]
 and let $\tau(t,x) = \phi(t)$ be a smooth time function on $M$ with $\phi'>0$. Then the induced null distance $\hat d_\tau$ is definite and encodes causality of $M$.
\end{customthm}

\begin{rmrk}
 Note that, if one drops completeness of $(\Sigma,\bsigma)$ in Theorem \ref{thm:warpedcausality}, one can still prove that
 \[
  q \in \overline{I^+(p)} \Longleftrightarrow \hat d_\tau (p,q) = \tau(q)-\tau(p),
 \]
 see \cite{SV}*{Theorem 3.28}.
\end{rmrk}

 On more general spacetimes, it is clear that
 \[
  p \leq q \Longrightarrow \hat d_\tau (p,q) = \tau(q) - \tau(p)
 \]
 is always satisfied, while the converse
 \begin{align}\label{encodecausality}
  p \leq q \Longleftarrow \hat d_\tau (p,q) = \tau(q) - \tau(p)
 \end{align}
 guarantees that $\hat d_\tau$ is definite \cite{SV}*{Lemma 3.12}, and hence $\hat d_\tau$ is a metric that induces the manifold topology for continuous $\tau$ \cite{SV}*{Proposition 3.15}.
 In \cite{SV}*{Proposition 4.5} it has been shown that locally anti-Lipschitz time functions $\tau$ are exactly those for which the corresponding null distances $\hat d_\tau$ are definite. 
 
\begin{rmrk}
 Suppose that $(M,\hat d_\tau)$ is such that for any two points $p,q \in M$, the distance is achieved by a piecewise causal curve $\beta$. Then $\hat L(\beta) = \hat d_\tau(p,q) = \tau(q) - \tau(p)$ implies by \cite{SV}*{Lemma 3.6} that $p \leq q$, i.e., the converse \eqref{encodecausality} holds and $\hat d_\tau$ encodes causality. 
\end{rmrk}
 
 Summing up the above the discussion, the following relations for (generalized) time functions $\tau$ are known:
  
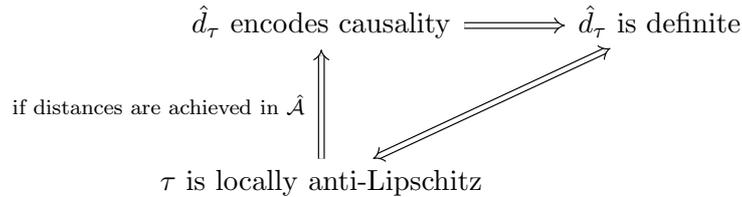
\[
 \begin{tikzcd}
  \hat d_\tau \text{ encodes causality} \arrow[Rightarrow]{r}{}%\text{ \cite{SV}*{Lemma 3.12} }}
  & \hat d_\tau ~\text{is definite}
 \\ [2em]
 \tau \text{ is locally anti-Lipschitz} \arrow[Rightarrow]{u}{\text{if distances are achieved in }\A~} \arrow[Leftrightarrow]{ur}{}%\text{ \cite{SV}*{Proposition 4.5} }}
 \end{tikzcd}
\]

 In particular, for spacetimes where all null distances between points are realized by piecewise causal curves, all three important notions are equivalent. In Riemannian geometry, such a condition would be guaranteed by completeness of the length space via the Hopf--Rinow Theorem. In Section \ref{subsec: Completeness} we have observed that completeness of $(M,\hat d_\tau)$ does in general not guarantee that a shortest path is piecewise causal, however, the null distances of the counterexamples still encode causality. As such, completeness still plays an important role in connection with the causality encoding property. In Examples \ref{IncompleteDoesNotEncodeCausality1} and \ref{IncompleteDoesNotEncodeCausality2} we discuss two simple incomplete (and not globally hyperbolic) spacetimes with globally anti-Lipschitz function that do not encode causality.
 
\begin{ex}\label{IncompleteDoesNotEncodeCausality1}
 Consider 2-dimensional Minkowski space $\M^{1+1}$ with the point $(1,1)$ removed and time function $\tau = t$ (see \cite{O}*{Ex.\ 4, p.\ 403}). For $p=(0,0)$ and $q=(2,2)$ we obtain that $\hat{d}_{\tau}(p,q) = \tau(q)-\tau(p)=2$ which can be seen by taking a sequence of piecewise causal curves in the future light cone of $p$ (see Figure~\ref{IncompleteDoesNotEncodeCausality1}). On the other hand, it is clear that $p \not \le q$.

 \begin{figure}[h]
 \centering
 \begin{tikzpicture}[scale=3]
  \draw[->] (0,0) -- (2.2,0) node[anchor=north west] {$x$};
  \draw[->] (0,0) -- (0,2.7) node[anchor=south east] {$t$};
  \draw[dotted] (0,0) -- (2,2);
  \draw[fill] (0,0) circle [radius=0.025];
  \node[left, outer sep=2pt] at (0,0) {$p$};
  \draw[fill] (2,2) circle [radius=0.025];
  \node[right, outer sep=2pt] at (2,2) {$q$};
  \draw[draw=black, fill=white] (1,1) circle [radius=0.025];
  \node[right, outer sep=2pt] at (1,1) {$\mathrm{point}~\mathrm{removed}$};
  \draw[thick] (0,0) -- (2-1/2,2+1/2) -- (2,2);
  \node[left, outer sep=2pt] at (2-1/2,2+1/2) {$\beta_1$};
  \draw[ultra thick] (0,0) -- (2-1/4,2+1/4) -- (2,2);
  \node[left, outer sep=2pt] at (2-1/4,2+1/4) {$\beta_2$};
  \draw[thick, dashed] (0,0) -- (2-1/8,2+1/8) -- (2,2);
  \node[left, outer sep=2pt] at (2-1/8,2+1/8) {$\beta_3$};
 \end{tikzpicture}
 \caption{A sequence of piecewise null curves $\beta_i$ in $\M^{1+1}$ whose length converges to $\hat{d}_{\tau}(p,q)=2$ in an  incomplete spacetime, described in Example~\ref{IncompleteDoesNotEncodeCausality1}.}
 \label{fig:IncompleteDoesNotEncodeCausality1}
 \end{figure}
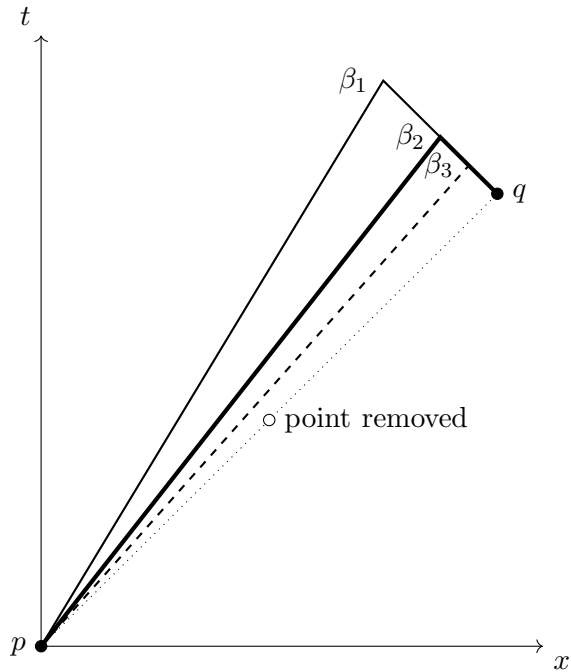
\end{ex}

\begin{ex}\label{IncompleteDoesNotEncodeCausality2}
 Consider again $\M^{1+1}$ with time function $\tau=t$, but now we remove the whole line $\{(x,1) : -1 \leq x \leq 1 \}$.  Let $p=(0,0)$ and $q=(0,2)$ which satisfy $\hat{d}_{\tau}(p,q) = \tau(q)-\tau(p)=2$ as can be seen by taking (approximating) paths which connect $p$ via null geodesics to $p_\varepsilon = (\varepsilon,0)$, through $(1+\varepsilon,1)$ and then proceed to $q_\varepsilon = (\varepsilon,2)$ (see Figure~\ref{IncompleteDoesNotEncodeCausality2}). It is also clear in this case that $p \not \le q$.
 
 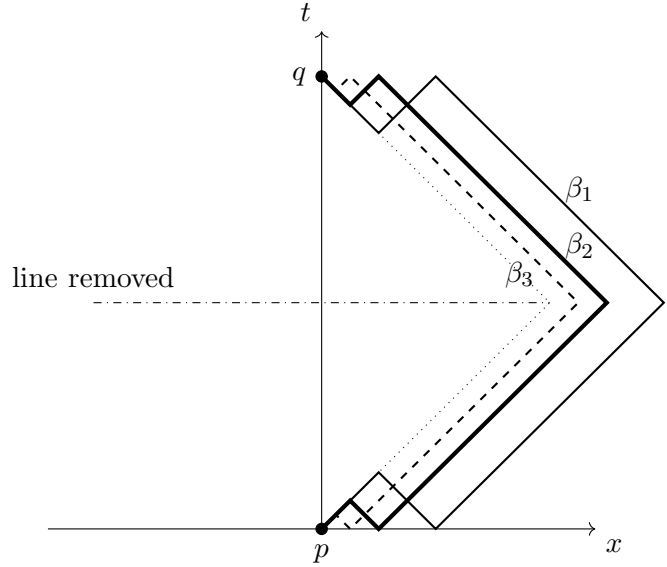
\begin{figure}
  \begin{tikzpicture}[scale=3]
  \draw[->] (-1.2,0) -- (1.2,0) node[anchor=north west] {$x$};
  \draw[->] (0,0) -- (0,2.2) node[anchor=south east] {$t$};
  \draw[dashdotted] (-1,1) -- (1,1);
  \node[above, outer sep=2pt] at (-1,1) {$\mathrm{line}~\mathrm{removed}$};
  \draw[fill] (0,0) circle [radius=0.025];
  \node[below, outer sep=2pt] at (0,0) {$p$};
  \draw[fill] (0,2) circle [radius=0.025];
  \node[left, outer sep=2pt] at (0,2) {$q$};
  \draw[dotted] (0,0) -- (1,1) -- (0,2);
  \draw[thick] (0,0)--(1/4,1/4) -- (1/2,0) -- (3/2,1) -- (1/2,2) -- (1/4,7/4) -- (0,2);
  \node[right, outer sep=2pt] at (1,3/2) {$\beta_1$};
  \draw[ultra thick] (0,0)--(1/8,1/8) -- (1/4,0) -- (5/4,1) -- (1/4,2) -- (1/8,1+7/8) -- (0,2);
  \node[right, outer sep=2pt] at (1,5/4) {$\beta_2$}; 
  \draw[thick, dashed] (0,0)--(1/16,1/16) -- (1/8,0) -- (9/8,1) -- (1/8,2) -- (1/16,1+15/16) -- (0,2);
  \node[left, outer sep=2pt] at (1,9/8) {$\beta_3$};
  \end{tikzpicture}
  \caption{A sequence of piecewise null curves $\beta_i$ in $\M^{1+1}$ whose null length converges to $\hat{d}_{\tau}(p,q)=2$ in an incomplete spacetime, described in Example~\ref{IncompleteDoesNotEncodeCausality2}.}
 \label{fig:IncompleteDoesNotEncodeCausality2}
 \end{figure}
\end{ex}

%%%%%%%%%%%%%%%%%%%%%%%%%%%%%%%%%%%%%%%%%%%%%%%%%%%%%%%%%%%%%%%%%%%%%%%%%%%%%%%%%%%%%%%%%%%%%%%%%%%%%%%%%%%%%%%%%%%%%%%%
 
\section{$(M,\hat d_\tau)$ as an integral current space}\label{sect-Integral Current Spaces}

%%%%%%%%%%%%%%%%%%%%%%%%%%%%%%%%%%%%%%%%%%%%%%%%%%%%%%%%%%%%%%%%%%%%%%%%%%%%%%%%%%%%%%%%%%%%%%%%%%%%%%%%%%%%%%%%%%%%%%%% 

 Our goal in this section is to rigorously introduce the current structure on spacetimes (including those of low regularity), which carry a metric structure obtained via the null distance. In particular, we focus here on warped products and globally hyperbolic spacetimes, since they constitute the most important classes of Lorentzian manifolds in general relativity. The strategy is to induce a (local) integral current space structure by using (locally) bi-Lipschitz maps to Lorentzian products and then defining bi-Lipschitz maps from regions of Lorentzian products to regions of a suitable Riemannian space which allows us to apply Theorem \ref{thm:intcurrent} or Theorem \ref{thm:localintcurrent}. If the Lorentzian metrics are continuous a null distance analogue of \cite{B}*{Theorem 4.5} suffices, however, we also allow metrics of less regularity if they have a warped product structure. Thus we prove a Lorentzian analogue of the Riemannian result in \cite{AS}*{Lemma 2.3} instead. We start by investigating properties of Lorentzian products, and then move on to warped products and globally hyperbolic spacetimes.

\subsection{Lorentzian product manifolds}\label{ssec:genMink}
   
 We analyze Lorentzian products $\M_\bsigma = \R \times \Sigma$, where $(\Sigma,\bsigma)$ is a suitable Riemannian manifold, and show that its canonical null distance $\hat d_\bsigma = \hat d_t$  has the same features as the null distance of Minkowski space $\M^{n+1} = \R^{n,1}$. Moreover, we prove that $\M_\bsigma$ is a (local) integral current space.
   
 Lorentzian products naturally include all Lorentzian manifolds for which a Lorentzian splitting theorem applies (see, for instance, \cite{E,G1,G2,N}, and \cite{F} for an overview of Lorentzian splitting theorems). In particular, they include globally hyperbolic (or timelike geodesically complete) spacetimes with nonnegative Ricci curvature bounds and a timelike line.
   
 \smallskip
 We recall the metric and locally integral current structure of the simplest Lorentzian product. 
  
\begin{ex}[Minkowski space]\label{exMinkowski}
  By  $\M^{n+1}$ we denote the ($n+1$)-dimensional Min\-kow\-ski space. The Minkowski metric can be written as
  \[
   \etab = -dt^2 + \g_{\E^n},
  \]
 where $\g_{\E^n}$ is the standard metric on the Euclidean $n$-space. For points $p \in \M^{n+1}$ we write $p = (t_p,p_\E)$, where $t_p = t(p) \in \R$ and $p_\E \in\R^n$. The null distance of $\M^{n+1}$ (with respect to the canonical time function $\tau = t$) is denoted by $\hat d_{\M}$. Several properties of this null distance have already been established by Sormani and Vega in \cite{SV}. Amongst others, it was shown in \cite{SV}*{Proposition 3.3} that $\hat d_{\M}$ is definite, translation invariant and satisfies
  \begin{align}\label{ndMink}
   \hat d_{\M}(p,q) = \begin{cases}
                                |t(q) - t(p)| & q \in J^\pm(p), \\
                                \| q_\E - p_\E \| & q \not\in J^\pm(p),
                               \end{cases}
  \end{align}
 where $\|.\|$ denotes the Euclidean norm and $J^\pm(p)$ the causal future/past of $p$. Note that many of these properties are lost if one uses another time function \cite{SV}*{Proposition 3.4}.
 The standard locally integral current $T$ on $\M^{n+1}$ is that of $\R^{n+1}$, that is,
  \[
   T(f,\pi_1,\ldots,\pi_{n+1}) = \int_{\R^{n+1}} f \, d\pi_1 \wedge \ldots \wedge d\pi_{n+1}.
  \]
\end{ex}

\begin{defn}\label{Def generalized Minkowski}
  Let $(\Sigma,\bsigma)$ be a Riemannian manifold. The \emph{Lorentzian product manifold} $\M_\bsigma$ is the Lorentzian manifold $\R \times \Sigma$ with metric tensor
 \begin{align}
 \etab_\bsigma = -dt^2 + \bsigma.
 \end{align} 
\end{defn}

\begin{rmrk}[Regularity]
 In Definition~\ref{Def generalized Minkowski} we do not need to assume that $\bsigma$ is smooth, however, the induced Riemannian distance function $d_\bsigma$ should be well-defined to obtain a meaningful null distance $\hat d_\bsigma$. We can, for instance, assume that $\bsigma$ is continuous \cite{B}.
\end{rmrk}

 The form \eqref{ndMink} of the null distance in Minkowski space extends to (weak) Lorentzian products in the obvious way.
  
\begin{lem}\label{lem:nulld1}
 Let $(\Sigma,\bsigma)$ be a connected Riemannian manifold and $(\M_\bsigma,\etab_\bsigma)$ be the corresponding Lorentzian product manifold of Definition~\ref{Def generalized Minkowski}. Then the null distance $\hat d_\bsigma$ induced by the canonical time function $t$ satisfies
 \begin{align}\label{d:genMink}
  \hat d_\bsigma (p,q) = \begin{cases}
                           |t(p) - t(q)| & q \in J^\pm(p), \\
                           d_\bsigma(p_\Sigma,q_\Sigma) & q \not\in J^\pm(p),
                          \end{cases}
 \end{align}
 where $d_\bsigma$ is the Riemannian distance function induced by $\bsigma$ on $\Sigma$ and $J^\pm(p)$ is the causal future/past of $p$ with respect to $\etab_\bsigma$.
\end{lem}
  
\begin{proof}
 Let $p \in \M_\bsigma$. If $q \in J^\pm(p)$, then any causal curve $\beta$ between $p$ and $q$ is minimal and hence $\hat d_\bsigma(p,q) = \hat L_t(\beta) = |t(q) - t(p)|$.
   
 Suppose $q \not\in J^\pm(p)$. By \cite{SV}*{Lemma 3.5} there is a piecewise causal curve $\beta$ joining $p$ and $q$. The curve $\beta$ thus consists of a finite number of (unbroken) causal curves which we call $\beta_1,\ldots,\beta_k$ and assume to be parametrized by time and without loss of generality of the form $\beta_i \colon [t_{i-1},t_i] \to M$ (if $t_{i-1} \geq t_i$ we may change orientation). That is, they are of the form
 \[
  \beta_i(t) = (t,\alpha_i(t)), \qquad t \in [t_{i-1},t_i],
 \]
 where $\alpha_i([t_{i-1},t_i]) \subseteq \Sigma$. Causality of $\beta_i$ implies that
 \[
  \etab_\bsigma(\beta_i',\beta_i') = - 1 + \bsigma(\alpha_i',\alpha_i') \leq 0,
 \]
 hence
 $
  \bsigma(\alpha_i',\alpha_i') \leq 1.
 $
 Thus
 \[
  L_\bsigma(\alpha) = \sum_{i=1}^k L_\bsigma(\alpha_i) \leq  \sum_{i=1}^k |t_{i}-t_{i-1}| = \hat L_t(\beta).
 \]
 Since
 $
  d_\bsigma (p_\Sigma,q_\Sigma) \leq  L_\bsigma(\tilde\alpha)
 $
 holds for any spatial curve $\tilde\alpha$, this implies that
 \begin{align}\label{eq:d1}
  d_\bsigma(p_\Sigma,q_\Sigma) \leq \hat d_\bsigma(p,q).
 \end{align}

 It remains to be shown that equality holds. Let $\varepsilon>0$ and $\alpha \colon [0,L] \to \Sigma$ be a (non-selfintersecting) curve from $p_\Sigma$ to $q_\Sigma$, parametrized by arc length, such that
 \begin{align}\label{Lalpha}
  L_\bsigma(\alpha) < d_\bsigma (p_\Sigma,q_\Sigma) + \varepsilon.
 \end{align}
 We construct a broken null curve $\beta$ from $p$ to $q$ by suitably lifting $\alpha$ to $\M_\bsigma$ such that $\hat L_t(\beta) = L_\bsigma (\alpha)$. Consider the connected manifold $\M_{\tilde\bsigma} = \R \times \image(\alpha)$ with induced Lorentzian metric $\tilde\bsigma = \bsigma|_{\image(\alpha)}$ and the same canonical time function $t$. By \cite{SV}*{Lemma 3.5 \& Remark 3.7} (making use of \cite{FS}) we know there exists a piecewise null curve $\beta$ in $\M_{\tilde\bsigma}$ connecting $p$ and $q$. Since $\image(\beta) \subseteq \R \times \image(\alpha)$ we can assume without loss of generality that $\beta$ is parametrized such that
 \[
  \beta(s) = (t(s),\alpha(s))
 \]
 and has breaking points $s_i$, $i=0,\ldots,k$ (since we can always restrict to $\R \times \image(\alpha|_{[s_i,b]})$ in the $i$-th step, we can assume that $\beta$ does not go ``back and forth'' along $\alpha$). Since $\beta$ is piecewise null with respect to $\tilde\bsigma$, and hence with respect to $\bsigma$, we have that
 \begin{align*}
  \hat L_t(\beta) &= \sum_{i=1}^k |t(\beta(s_{i}))-t(\beta(s_{i-1}))| \\
                &= \sum_{i=1}^k L_{\tilde\bsigma} (\alpha|_{[s_{i-1},s_i]}) = L_\bsigma (\alpha).
 \end{align*}
 Together with \eqref{Lalpha}, we obtain
 \[
  \hat d_\bsigma(p,q) \leq \hat L_t(\beta) = L_\bsigma(\alpha) < d_\bsigma(p_\Sigma,q_\Sigma) + \varepsilon,
 \]
 and therefore, together with \eqref{eq:d1},
 \begin{align}\label{dd}
 \hat d_\bsigma(p,q) = d_\bsigma(p_\Sigma,q_\Sigma)
 \end{align}
 whenever $p$ and $q$ are not causally related, i.e., if $q \not\in J^\pm(p)$ and vice versa. This proves the second case for \eqref{d:genMink}.
\end{proof} 
 
\begin{rmrk}
 In the above proof of Lemma~\ref{lem:nulld1} we considered $\M_\bsigma = \R \times \Sigma$. The same result holds for $\widetilde \M_\bsigma = I \times \Sigma$, where $I$ is an interval. Note that $\Sigma$ does not necessarily need to be complete.
\end{rmrk}

\begin{rmrk}[Causality encoding property]
 Lorentzian products $\M_\bsigma = I \times (\Sigma,\bsigma)$ always satisfy the anti-Lipschitz property for the canonical time function due to the explicit structure \eqref{d:genMink} of the null distance obtained in Lemma~\ref{lem:nulld1}.
\end{rmrk}

 In order to prove that we have a (locally) integral current structure on a warped product $M_f = I \times_f \Sigma$ it therefore remains to be shown that the Lorentzian product manifold $\M_\bsigma$ with null distance $\hat d_\bsigma$ is a (local) integral current space, and to establish a connection between $M_f$ and $\M_\bsigma$. This is shown in Lemma~\ref{lem:intcurrent}, Proposition~\ref{prop:biLip} and Theorem \ref{thm: Warp are integral currents}.

\begin{lem}\label{lem:intcurrent}
 Let $I \subseteq \R$ be an interval and $(\Sigma,\bsigma)$ be a connected complete Riemannian manifold. Let $\M_\bsigma$ be the corresponding Lorentzian product manifold, endowed with the null distance $\hat d_\bsigma$ with respect to the canonical time function $\tau = t$. Let $\bar d_\bsigma$ be the distance function with respect to the Riemannian warped product
 \[
  \bar \etab_\bsigma = dt^2 + \bsigma.
 \]
 Then the identity map
 \[
  \id \colon (I \times \Sigma, \hat d_\bsigma) \to (I \times \Sigma, \bar d_\bsigma),
 \]
 is bi-Lipschitz and a natural local integral current space structure is induced on $\M_\bsigma$. If, in addition, $I \times \Sigma$ is compact then $\M_\bsigma$ is an integral current space.
\end{lem}

\begin{rmrk}\label{rem:Pyt}
 The proof requires knowledge of the metric structure of products of Riemannian manifolds. Let $M = \Sigma_1 \times \Sigma_2$ be the product of the Riemannian manifolds $(\Sigma_1,\h_1)$ and $(\Sigma_2,\h_2)$ with induced distance functions $d_{\h_1\times\h_2},d_{\h_1},d_{\h_2}$. Suppose $\Sigma_1$ is a uniquely geodesic metric space, that is, any two points are the endpoints of a unique minimizing geodesic. Then the distance function on $M$ is given by the Pythagorean formula
\[
 d_{\h_1\times\h_2} (p,q)^2 = d_{\h_1} (p_1,q_1)^2 + d_{\h_2} (p_2,q_2)^2, \qquad p,q \in M.
\]
 See \cite{DM}*{Section 3.1} for a proof when $\Sigma_1$ and $\Sigma_2$ are complete Riemannian manifolds.
\end{rmrk}

\begin{proof}
 We first assume that $q \in J^{\pm}(p)$. Then by (the proof of) Lemma \ref{lem:nulld1} we have that
 \bee
  \hat{d}_\bsigma (p,q) = |t(q)-t(p)| \geq d_{\bsigma}(p_{\Sigma},q_{\Sigma}),
 \eee
 which together with Remark~\ref{rem:Pyt} implies that
 \begin{align}
  \hat{d}_\bsigma(p,q) &=  |t(q)-t(p)| \nonumber
  \\&\le \bar d_\bsigma(p,q) = \sqrt{|t(q)-t(p)|^2+d_{\bsigma}(p_{\Sigma},q_{\Sigma})^2} \label{FirstBiLip}
  \\&\le \sqrt{2} \, |t(q)-t(p)|=\sqrt{2}\,\hat{d}_\bsigma(p,q).\nonumber
 \end{align}

 Next assume that $q \not \in J^{\pm}(p)$. Then by (the proof of) Lemma \ref{lem:nulld1} we know that
 \begin{align*}
  \hat{d}_\bsigma (p,q) = d_{\bsigma}(p_\Sigma,q_{\Sigma}) \ge |t(q)-t(p)|,
 \end{align*}
 which implies that
 \begin{align}
  \hat{d}_\bsigma (p,q)&=d_{\bsigma}(p_{\Sigma},q_{\Sigma}) \nonumber
  \\&\le \bar d_\bsigma(p,q) = \sqrt{|t(q)-t(p)|^2+d_{\bsigma}(p_{\Sigma},q_{\Sigma})^2} \label{SecBiLip}
  \\&\le \sqrt{2}\,d_{\bsigma}(p_\Sigma,q_\Sigma) = \sqrt{2}\,\hat{d}_\bsigma (p,q).\nonumber
 \end{align}
 Note that \eqref{FirstBiLip} and \eqref{SecBiLip} imply that
 \[
  \hat{d}_\bsigma (p,q) \leq \bar d_\bsigma(p,q) \leq \sqrt{2}\, \hat{d}_\bsigma (p,q).
 \]
 for any $p,q \in I \times \Sigma$.
 Hence the identity map on $I \times \Sigma$ is bi-Lipschitz from $\hat{d}_\bsigma$ to $\bar d_\bsigma$.
 
 Recall that connected complete Riemannian manifolds with continuous metric are local integral current spaces. Let $T$ be this locally integral current on $(I \times \Sigma, \bar{d}_\sigma)$. Then by Theorem~\ref{thm:localintcurrent} we find that $(\M_\bsigma,\hat d_\bsigma, \id_\#T)$ is a local integral current space. By Theorem~\ref{thm:intcurrent} it is an integral current space when $I \times \Sigma$ is compact.
\end{proof}

\subsection{Spacetime warped products}

 The formula \eqref{d:genMink} can be generalized to warped products with merely bounded warping function $f$ if we estimate the Riemannian distance function using the bounds of $f$. Since we always use the canonical time function $\tau(t,x) =t$, we simply write $\hat d_f$ for the null distance $\hat d_t$ corresponding to the causal structure of the warped product $I \times_f \Sigma$. We also write $\hat L$ for the corresponding null length $\hat L_t$. %(bounded) warping function.

\begin{lem}\label{d:warped}
 Let $I$ be an interval, $(\Sigma,\bsigma)$ be a connected Riemannian manifold, and $f$ be a bounded function such that
 \[
  0 < f_{\min} \leq f(t) \leq f_{\max}, \qquad t \in I.
 \]
 Then the warped product $M_f = I \times_f \Sigma$ with Lorentzian metric
  \[
  \g_f = -dt^2 + f(t)^2 \bsigma
 \]
 is such that the null distance $\hat d_f$ of $\g_f$ with canonical time function $t$ satisfies
 \begin{equation}\label{eqd:warped}
 \begin{split}
  \hat d_f(p,q) &= |t(p) - t(q)|, \qquad \qquad \quad \, \, q \in J^\pm(p), \\ %\label{d:warped1} \\
  f_{\min} d_\bsigma(p_\Sigma,q_\Sigma) \leq \hat d_f&(p,q) \leq f_{\max} d_\bsigma(p_\Sigma,q_\Sigma), \qquad q \not\in J^\pm(p).
 \end{split}
 \end{equation}
\end{lem}

\begin{proof}
 Essentially we can apply the same proof as in Section~\ref{ssec:genMink} above. Let $p,q \in M$ and $\beta(s)=(t(s),\alpha(s))$ be a null causal curve joining $p$ and $q$. The fact that
 \begin{align}\label{gfest}
  \g_f(\beta'(t),\beta'(t)) = -1 + f(t)^2 \bsigma(\alpha'(t),\alpha'(t)) = 0,
 \end{align}
 however, implies that $\bsigma(\alpha',\alpha') \leq \frac{1}{f_{\min}^2}$ and thus
 \[
  L_\bsigma(\alpha) \leq \frac{1}{f_{\min}} \hat L(\beta).
 \]
 Therefore,
 \[
 f_{\min} d_\bsigma(p_\Sigma,q_\Sigma) \leq \hat d_f(p,q).
 \]
 In the second part of the proof, where a null curve $\beta$ is constructed from an (almost) $\Sigma$-minimizing $\alpha$, \eqref{gfest} implies that $\bsigma(\alpha',\alpha') \geq \frac{1}{f_{\max}^2}$. Therefore,
 \[
   \frac{1}{f_{\max}} \hat d_{f}(p,q) \leq \frac{1}{f_{\max}} \hat L(\beta) \leq L_\bsigma(\alpha) \leq d_\bsigma(p,q) +\varepsilon,
 \]
 which shows that
 \[
  \hat d_{f}(p,q) \leq f_{\max} d_\bsigma(p,q)
 \]
 if $q \not\in J^\pm(p)$.
\end{proof}

\begin{prop}\label{prop:biLip}
 Let $(\Sigma,\bsigma)$ be a connected Riemannian manifold. For a fixed closed interval $I$, let $f$ be a bounded function on $I$ such that
  \[
   0< f_{\min} \leq f(t) \leq f_{\max} < \infty, \qquad t\in I.
  \]
  Let $M_f = I \times_f \Sigma$ be the warped product with Lorentzian metric
  \[
   \g_f = -dt^2 + f(t)^2 \bsigma.
  \]
 Then for all $p,q \in I \times \Sigma$
 \begin{align}\label{dfestimate}
 \min \{1, f_{\min} \} \, \hat d_\bsigma (p,q) \leq \hat d_f (p,q) \leq  \max \{1,
  f_{\max} \} \, \hat d_\bsigma (p,q), 
 \end{align}
 where $\hat d_\bsigma$ and $\hat d_f$ are the null distances of the Lorentzian product $\M_\bsigma$ and the warped product $M_f$, respectively, both with respect to the canonical time function $\tau=t$ (see Definition~\ref{Def generalized Minkowski}).
\end{prop}

\begin{rmrk}\label{rmrk:biLip}
 A more refined estimate than \eqref{dfestimate} based on the precise causal relation of $p$ and $q$ with respect to both $\g_f$ and $\etab_\bsigma$ is obtained in the proof of Proposition~\ref{prop:biLip}. Depending on $f$, not all cases may occur (in which case the argument would lead to a contradiction). However, since case 1 always occurs in a small neighborhood, one cannot omit the constants $1$ in the estimate \eqref{dfestimate}.
\end{rmrk} 
 
\begin{proof}
 Depending on the causal relation of $p$ and $q$ we have to distinguish several cases. Without loss of generality we can assume that $t(q)\geq t(p)$. By $J^\pm_1(p)$ and $J^\pm_f(p)$ we denote the causal past/future of $p$ with respect to $\etab_\bsigma$ and $\g_f$, respectively. We use $\hat L_1$ and $\hat L_f$ to denote the null length of the respective metrics with respect to the canonical time function $\tau = t$, and $L_\bsigma$ and $d_\bsigma$ to denote the Riemannian length and distance in $(\Sigma,\bsigma)$.
 
 \textit{Case 1: $q \in J^+_1(p) \cap J^+_f(p)$.} In this case $p \leq q$ in both metrics and therefore,
 \[
  \hat d_f(p,q) = t(q)-t(p) = \hat d_\bsigma(p,q),
 \]
 because all causal curves are distance-realizing.
 
 \textit{Case 2: $q \in J^+_f(p) \setminus J^+_1(p)$.} Suppose $\beta$ is a future-directed causal curve in $J^+_f(p)$ connecting $p$ and $q$. Then
 \begin{align}\label{eqdf1}
  \hat d_f(p,q) = \hat L_f(\beta) = t(q)-t(p) \leq \hat d_\bsigma(p,q).
 \end{align}
 Since $\beta$ is causal with respect to $\g_f$, we may assume that it is parametrized by time, i.e., $\beta(t) = (t,\alpha(t))$ for some curve $\alpha$ in $\Sigma$. Thus
 \[
  \g_f(\beta'(t),\beta'(t)) = -1 + f(t)^2 \bsigma(\alpha'(t),\alpha'(t)) \leq 0,
 \]
and hence $\|\alpha'\|_\bsigma \leq \frac{1}{f_{\min}}$. This implies
 \[
  d_\bsigma(p_\Sigma,q_\Sigma) \leq L_\bsigma(\alpha) \leq \frac{1}{f_{\min}} |t(q)-t(p)|.
 \]
 By Lemma \ref{lem:nulld1} and Lemma \ref{d:warped} together with \eqref{eqdf1} we obtain that
 \[
  \hat d_f (p,q) \leq \hat d_\bsigma(p,q) = d_\bsigma(p_\Sigma,q_\Sigma) \leq \frac{1}{f_{\min}} \hat d_f(p,q).
 \]
 
 \textit{Case 3: $q \in J^+_1(p) \setminus J^+_f(p)$.} Then for a future-directed causal curve $\beta(t) = (t,\alpha(t))$ in $J^+_1(p)$ we have
 \begin{align}\label{eqds1}
  \hat d_\bsigma (p,q) = \hat L_{1}(\beta) = t(q) - t(p) \leq \hat d_f(p,q) .
 \end{align}
 Causality of $\beta$ implies that
 \[
  \etab_\bsigma (\beta'(t),\beta'(t)) = -1 + \bsigma (\alpha'(t),\alpha'(t)) \leq 0,
 \]
 and therefore that $\| \alpha' \|_\bsigma \leq 1$. Thus
 \[
  d_\bsigma(p_\Sigma,q_\Sigma) \leq L_\bsigma(\alpha) \leq |t(q)-t(p)|.
 \]
 By Lemma \ref{lem:nulld1} and Lemma \ref{d:warped} together with \eqref{eqds1} we hence obtain
 \[
  \frac{1}{f_{\max}} \hat d_f(p,q) \leq d_\bsigma (p_\Sigma,q_\Sigma) \leq \hat d_\bsigma(p,q) \leq \hat d_f(p,q).
 \]
 
 \textit{Case 4: $q \not\in J^+_1(p) \cup J^+_f(p)$.} If $q$ is in neither of the light cones, then by Lemma~\ref{lem:nulld1}
 \begin{align*}
  \hat d_\bsigma(p,q) = d_\bsigma(p_\Sigma,q_\Sigma),
 \end{align*}
 and by Lemma \ref{d:warped} we know that
 \begin{align*}
  f_{\min} d_\bsigma(p_\Sigma,q_\Sigma) \leq \hat d_f&(p,q) \leq f_{\max} d_\bsigma(p_\Sigma,q_\Sigma).
 \end{align*}
 Hence the statement follows immediately in this case.
\end{proof}

\begin{rmrk}
 We expect that Proposition \ref{prop:biLip} can be extended to multi-warped products and to spacetimes $N_h = I \, {}_h \! \times (\Sigma, \bsigma)$ of the form
  \[
   \g_h = - h^2 dt^2 + \bsigma,
  \]
 where $h$ is a positive bounded function on $\Sigma$. Singular static metrics appear, for instance, when the Einstein equations are coupled to matter fields \cite{AnB,BKTZ}. Using spacetime convergence results, it may be possible to study their spacetime stability using techniques analogous to \cite{BKS,HLS,LeeS,LFS}.
\end{rmrk}

 Combining the above results, Proposition \ref{prop:biLip}, Theorem \ref{thm:intcurrent}, and Theorem \ref{thm:localintcurrent} yields that we can induce a natural (local) integral current space structure on warped product spacetimes $I \times_f (\Sigma,\bsigma)$ with warping functions $f$ of low regularity and continuous Riemannian metrics $\bsigma$.
   
\begin{customthm}{\ref{thm: Warp are integral currents}}
 Let $I$ be an interval and $(\Sigma,\bsigma)$ be a connected complete Riemannian manifold. Suppose $f \colon I \to (0,\infty)$ is a bounded function that is bounded away from $0$. There is a natural local integral current space structure on the warped product spacetime $M = I \times_f \Sigma$ with respect to the null distance $\hat d_f$.  If $I \times \Sigma$ is compact then $(M,\hat d_f)$ carries an integral current space structure. 
\end{customthm}
  
\begin{proof}
 By Lemma~\ref{lem:intcurrent} the Lorentzian product $I \times \Sigma$ is a (local) integral current space with respect to $\hat d_\bsigma$. Since the identity $\id \colon (I\times\Sigma,\hat{d}_\bsigma) \to (I \times \Sigma, \hat d_f)$ is bi-Lipschitz by Proposition~\ref{prop:biLip}, the pushforward Theorems~\ref{thm:intcurrent} and \ref{thm:localintcurrent} imply that the warped product spacetime $M$ with null distance $\hat d_f$ can be induced with a (local) integral current structure as well.
\end{proof}

\subsection{Globally hyperbolic spacetimes}

 We  use the results of the previous subsections to show that globally hyperbolic spacetimes are (local) integral current spaces with respect to their corresponding null distances. We begin with a result which allows us to compare null distance functions between general Lorentzian metrics and warped product metrics. We adapt the following (slightly modified) $\preccurlyeq$ relation from \cite{CG}*{Section 1.2}.
 
\begin{defn}\label{def:relation}
 Let $M$ be a connected manifold with Lorentzian metrics $\g_1,\g_2$. We say that $\g_1$ is \emph{smaller} than $\g_2$, denoted by $\g_1 \preccurlyeq \g_2$, if for all tangent vectors $v \neq 0$ the implication
  \[
   \g_1(v,v) \leq 0 \Longrightarrow \g_2(v,v) \leq 0
  \]
 is satisfied.
\end{defn}
 
 In other words, if $\g_1 \preccurlyeq \g_2$ then the light cones with respect to $\g_2$ are wider than those of $\g_1$.
 
\begin{ex}\label{ex:order}
 If $(\Sigma,\bsigma)$ is a connected Riemannian manifold and $f_1,f_2$ are two warping functions that satisfy $0 <f_1 \leq f_2$ everywhere, then the warped product metrics on $I \times \Sigma$ satisfy $\g_{f_2} \preccurlyeq \g_{f_1}$
 and the corresponding null distances satisfy $\hat d_{f_1}(p,q) \leq \hat d_{f_2} (p,q)$.
\end{ex}

The observation regarding the null distance in Example~\ref{ex:order} can be generalized.

\begin{lem}\label{lem: biLip Inequality}
 Let $M$ be a connected manifold and $\g_1, \g, \g_2$ be Lorentzian metrics on $M$ such that 
 \begin{align}\label{crucialInequality1}
  \g_1 \preccurlyeq \g \preccurlyeq \g_2.
 \end{align}
 Then for a time function $\tau \colon M \to \R$ the corresponding null distances satisfy
 \begin{align*}
 \hat{d}_{\g_2}(p,q) \le \hat{d}_{\g}(p,q) \le \hat{d}_{\g_1}(p,q), \qquad p,q \in M.
 \end{align*}
\end{lem}
   
\begin{proof}
 Property \eqref{crucialInequality1} makes sure that the light cones of the different Lorentzian manifolds are contained in one another. In particular,  if $\alpha$ is a piecewise causal curve for $\g_1$ then $\alpha$ is a piecewise causal curve for $\g$. Similarly, if $\beta$ is a piecewise causal curve for $\g$, then $\beta$ is a piecewise causal curve for $\g_2$. Since the null distance between $p$ and $q$ is defined as an infimum over piecewise causal curves these inclusions imply the desired conclusion,
 \[
  \hat{d}_{\g_2}(p,q) \le \hat{d}_{\g}(p,q) \le \hat{d}_{\g_1}(p,q). \qedhere
 \]
\end{proof}
  
 We generalize Lemma~\ref{lem: biLip Inequality} to a local version which is utilized below and also interesting in its own right. The proof is based on a similar argument as the proof of \cite{B}*{Theorem 4.5} in the Riemannian case, however, the assumptions are necessarily stronger.
  
\begin{prop}\label{prop: dist comparison}
 Let $M$ be a connected manifold, equipped with Lorentzian metrics $\g$ and $\h$. Let $U$ be an open set such that $\g\preccurlyeq \h$ holds on $U$. Suppose $\tau$ is a time function with respect to both $\g$ and $\h$ such that the null distances $\hat d_\g$ and $\hat d_\h$ with respect to $\tau$ are metrics. Then for every compact set $K \subset U$ there exists a constant $C>0$ (possibly depending on $K$) such that
 \begin{align}\label{bi-Lipestimate}
  \hat d_\h(p,q) \leq C \hat d_\g (p,q), \qquad p,q \in K.
 \end{align}
\end{prop}

\begin{proof}
 We proceed by contradiction. Suppose \eqref{bi-Lipestimate} does not hold. Then for all $n \in \N$ there exist points $p_n,q_n \in K$ such that
 \begin{align}\label{assumption1}
  \hat d_\h(p_n,q_n) > n \hat d_\g(p_n,q_n).
 \end{align}
 Since $K$ is compact, and both $\hat d_\g, \hat d_\h$ are metrics that induce the manifold topology, we can assume that the sequences $(p_n)_n$ and $(q_n)_n$ converge to the same limit $p \in K$ by passing to subsequences.
 
 Since $M$ is locally compact there exists an $r_0>0$ such that the ball $\overline{B_{r_0}^\g (p)} = \{ q \in M : \hat d_\g(p,q) \leq r_0\}$ is compact and contained in $U$. Consider $r := \frac{r_0}{4}$ and arbitrary points $x,y \in B_r^\g(p)$. Then, for every $\varepsilon \in (0,r)$, there exists a (piecewise) causal curve $\beta_\varepsilon$ such that
 \[
  \hat L_\g(\beta_\varepsilon) < \hat d_\g (x,y) + \varepsilon.
 \]
 This curve $\beta_\varepsilon$ does not leave $B^\g_{r_0}(p)$ since for all $s$
 \begin{align*}
  \hat d_\g(p,\beta_\varepsilon(s)) &\leq \hat d_\g(p,x) + \underbrace{\hat d_\g(x,\beta_\varepsilon(s))}_{\leq \hat L_\g(\beta_\varepsilon)} \\
   &\leq \hat d_\g (p,x) + \hat d_\g(x,y) + \varepsilon \\
   &< r + 2r+ r = 4r = r_0.
 \end{align*}
 By assumption, $\g \preccurlyeq \h$ on $B_{r_0}^\g(p) \subseteq U$, hence $\beta_\varepsilon$ is also piecewise causal with respect to $\h$ and therefore $\hat L_\g(\beta_\varepsilon) = \hat L_\h (\beta_\varepsilon)$ and
 \begin{align*}
  \hat d_\h(x,y) \leq \hat L_\h(\beta_\varepsilon) = \hat L_\g(\beta_\varepsilon) < \hat d_\g(x,y) +\varepsilon.
 \end{align*}
 Since this estimate holds for all $\varepsilon \leq r$ this implies that
 \[
  \hat d_\h (x,y) \leq \hat d_\g(x,y), \qquad x,y \in B^\g_r(p).
 \]
 In particular, since the points $p_n$ and $q_n$ converge to $p$ we have that for sufficiently large $n$
 \[
  \hat d_\h (p_n,q_n) \leq \hat d_\g (p_n,q_n),
 \]
 which contradicts \eqref{assumption1}.
\end{proof}

 We apply an orthogonal splitting and use the above relation on chart neighborhoods and for particular Lorentzian product metrics $\g_1$ and $\g_2$. This induces a current on $(M,\hat{d}_\tau)$ via locally bi-Lipschitz maps.
  
\begin{prop}\label{prop: Globally Hyperbolic Are Integral Current Spaces}
 Let $(M,\g)$ be a globally hyperbolic spacetime. Then there exists a (smooth) time function $\tau$,  $(M,\g)$  is isometric to $(\R\times \Sigma, \tilde{\g})$ where $\tilde{\g}$ is a Lorentzian metric with Cauchy surfaces $(\Sigma_\tau,\bsigma_\tau=\bsigma|_{\Sigma_\tau})$, and 
  $\id \colon (\R \times \Sigma,\hat d_{\bsigma_0}) \to (\R \times \Sigma, \hat d_{\tilde \g})$ is locally bi-Lipschitz.
\end{prop}

\begin{proof}
 Bernal and S\'{a}nchez prove in \cite{BS} that any globally hyperbolic spacetime admits a smooth time function $\tau$ such that $\nabla \tau$ is everywhere past-pointing timelike. More precisely, $(M,\g)$  is isometric to $(\R\times \Sigma, \tilde{\g})$ with
 \begin{align}
 \tilde{\g} = -h^2 d \tau^2 + \bsigma,
 \end{align}
 where $\Sigma$ is a smooth spacelike Cauchy hypersurface, $\tau \colon \R \times \Sigma \rightarrow \R$ is the natural projection (and time function), $h^2 \colon \R\times \Sigma \rightarrow (0,\infty)$ is a smooth function, and $\bsigma$ is a symmetric 2-tensor field on $\R\times\Sigma$. Moreover, on each constant-$\tau$ hypersurface $\Sigma_\tau$ the restriction $\bsigma_\tau := \bsigma|_{\Sigma_\tau}$ is a Riemannian metric.
 
 Let $(U_i,u_i)_i$ be a countable atlas of precompact open sets on $\R \times \Sigma$ and let $V_i$ be precompact open sets so that $\overline{U}_i \subset V_i$. Since $h^2 >0$ and $V_i$ is precompact, there exists a constant $c_i$ such that $h^2|_{V_i} \in [\frac{1}{c_i^2},c_i^2]$. Moreover, we can compare each Riemannian metric $\bsigma_\tau$ on $V_i\cap \Sigma_\tau$ to $\bsigma_0$ since there exist continuous functions $\lambda_0,\mu_0 > 0$ such that
 \[
  \lambda_0(\tau)^2 \bsigma_0 (v,v) \leq \bsigma_\tau (v,v) \leq \mu_0(\tau)^2 \bsigma_0(v,v),
 \]
 for all $v \in T_{p'} \Sigma$, $(\tau,p') \in U_i \cap \Sigma_\tau$, which follows immediately from the proofs in \cite{B}*{Proposition 4.1 \& (4.2)}. For $a,b\in \R \setminus \{ 0 \}$ we define the reference product metric
 \[
  \tilde\g_{a,b} := - a^2d\tau^2 + b^2 \bsigma_0,
 \] 
 Due to the pre-compactness of $V_i$ we have positive constants
 \begin{align*}
  \lambda_i &:= \inf \lambda_0(\tau(V_i)) >0, \\
  \mu_i &:= \sup \mu_0(\tau(V_i)) \geq \lambda_i > 0,
 \end{align*}
 which implies that for $w \in T_p(\R\times \Sigma)$, $p \in V_i$,
 \begin{align*}
  \tilde\g_{c_i,\lambda_i} (w,w) \leq \tilde\g (w,w)  \leq  \, \tilde\g_{\frac{1}{c_i},\mu_i}(w,w),
 \end{align*}
 which in the notation of Definition~\ref{def:relation} says that locally on $V_i$ we have \[\tilde\g_{\frac{1}{c_i},\mu_i}\preccurlyeq\tilde \g\preccurlyeq\tilde\g_{c_i,\lambda_i}.\]
 Since $U_i$ is precompact, by Proposition~\ref{prop: dist comparison} there exists a constant $C_i>0$ such that
 \[
  \frac{1}{C_i} \hat d_{\tilde\g_{c_i,\lambda_i}}(p,q) \leq \hat d_{\tilde\g}(p,q) \leq C_i \hat d_{\tilde\g_{\frac{1}{c_i},\mu_i}}(p,q), \qquad p,q \in U_i.
 \]
 By \cite{SV}*{Proposition 3.9} the null distance is invariant under conformal changes, hence
  \[
  \frac{1}{C_i} \hat d_{\tilde\g_{1,\frac{\lambda_i}{c_i}}}(p,q) \leq \hat d_{\tilde\g}(p,q) \leq  C_i \hat d_{\tilde \g_{1,c_i\mu_i}}(p,q), \qquad p,q \in U_i.
 \]
By Proposition~\ref{prop:biLip}, we can estimate the left and right null distances further to obtain for the constant $A_i = C_i \max \{ c_i\mu_i,\frac{c_i}{\lambda_i},1 \} > 0$ that
 \[
  \frac{1}{A_i} \hat d_{\bsigma_0}(p,q) \leq \hat d_{\tilde\g}(p,q) \leq  A_i \hat d_{\bsigma_0}(p,q), \qquad p,q \in U_i.
 \]
 which means that the identity map  $\id \colon (\R \times \Sigma,\hat d_{\bsigma_0}) \to (\R \times \Sigma, \hat d_{\tilde \g})$ is bi-Lipschitz on $U_i$. 
\end{proof}

\begin{customthm}{\ref{thm: Globally Hyperbolic Are Integral Current Spaces}}
 Let $(M,\g)$ be a globally hyperbolic spacetime with (smooth) time function $\tau$ of Proposition~\ref{prop: Globally Hyperbolic Are Integral Current Spaces}. Suppose $M$ admits Cauchy hypersurfaces on which the ambient Lorentzian metric restricts as a complete Riemannian metric and $(M,\hat d_\tau)$ is complete as a metric space. Then $(M,\hat{d}_\tau)$ is a local integral current space.
If $M$ is compact then it is an integral current space.
\end{customthm}

\begin{proof}
 By Proposition~\ref{prop: Globally Hyperbolic Are Integral Current Spaces} $M =\R \times \Sigma$ and the identity $\id \colon (M,\hat d_{\bsigma_0}) \to (M, \hat d_{\g})$ is locally bi-Lipschitz. Recall that $\hat d_{\bsigma_0}$ is the null distance of the product spacetime $\M_{\bsigma_0}$ and thus carries a natural (local) integral current space structure by Lemma~\ref{lem:intcurrent}. By Theorem~\ref{thm:lengthspace} it is a locally compact length space, because it carries the manifold topology. Moreover, since $(\Sigma,\bsigma_0)$ is complete, so is the Riemannian product $(\R \times \Sigma, d\tau^2 + \bsigma_0)$ and hence the Lorentzian product $(M,\hat d_{\bsigma_0})$ is complete by Corollary~\ref{cor:complete}. Hence $(M,\hat{d}_{\bsigma_0})$ is a proper metric space (see Remark~\ref{rmrk:proper}). Hence the locally Lipschitz identity $\id \colon (M,\hat d_{\bsigma_0}) \to (M, \hat d_{\g})$ is Lipschitz on bounded sets by Remark~\ref{rmrk:proper}. Similarly, the assumption that $(M,\hat d_\tau)$ is complete implies that it is proper, and the inverse map $\id \colon (M, \hat d_{\g}) \to (M,\hat d_{\bsigma_0})$ is also Lipschitz on bounded sets.
 Hence by Theorem~\ref{thm:localintcurrent} we can push forward the (local) integral current structure via the locally bi-Lipschitz identity to $(M,\hat d_{\tilde \g})$.
 
 The compact case follows immediately from Proposition~\ref{prop: Globally Hyperbolic Are Integral Current Spaces} and Theorem~\ref{thm:intcurrent}.
\end{proof}

\begin{rmrk}
 Recall that a time-oriented Lorentzian manifold $(M,\g)$ is called globally hyperbolic if and only if it is causal and if for every pair of points $p,q \in M$ the set $J^+(p) \cap J^-(q)$ is compact (in fact, Hounnonkpe and Minguzzi recently showed the surprising result that the causal condition is not even needed for noncompact manifolds with spacetime dimensions larger than three \cite{HM}*{Theorem 2.8}). Lorentzian products $I \times \Sigma$ and warped products $I \times_f \Sigma$ are globally hyperbolic if and only if the Riemannian manifold $(\Sigma,\bsigma)$ is complete (see \cite{BEE}*{Theorem 3.66}). This is the setting in which we studied both cases earlier, hence Theorem~\ref{thm: Globally Hyperbolic Are Integral Current Spaces} naturally generalizes Lemma~\ref{lem:intcurrent} and Theorem~\ref{thm: Warp are integral currents} on the integral current structure of $I\times\Sigma$ and $I\times_f\Sigma$, respectively.
\end{rmrk}

\begin{rmrk}
 We believe Theorem~\ref{thm: Globally Hyperbolic Are Integral Current Spaces} can be extended to weaker notions of causality or related local assumptions. In particular, the regularity assumptions may be weakened. 
 
 Moreover, it should be straightforward to drop the completeness assumptions in Theorem~\ref{thm: Globally Hyperbolic Are Integral Current Spaces} (as well as in the earlier (warped) product results Theorem~\ref{thm: Warp are integral currents} and Lemma~\ref{lem:intcurrent}) by using a definition of local integral current spaces that is based on the locally integral currents of Lang~\cite{L} (see also our earlier Remark~\ref{rmrk:Lang}).
\end{rmrk}
 
%%%%%%%%%%%%%%%%%%%%%%%%%%%%%%%%%%%%%%%%%%%%%%%%%%%%%%%%%%%%%%%%%%%%%%%%%%%%%%%%%%%%%%%%%%%%%%%%%%%%%%%%%%%%%%%%%%%%%%%%%%
 
 \section{Convergence of warped product spacetimes}\label{sect-Convergence}
 
%%%%%%%%%%%%%%%%%%%%%%%%%%%%%%%%%%%%%%%%%%%%%%%%%%%%%%%%%%%%%%%%%%%%%%%%%%%%%%%%%%%%%%%%%%%%%%%%%%%%%%%%%%%%%%%%%%%%%%%%%%

 By Section~\ref{sect-Integral Current Spaces} we now know that warped product spacetimes (of low regularity) and globally hyperbolic spacetimes with complete Cauchy surfaces are (local) integral current spaces with respect to the null distance. We are thus in a position to study the limits of such sequences of spacetimes under GH and SWIF convergence. In this section we initiate the study of spacetime convergence for the class of warped products. In particular, we provide a general statement for uniformly converging warping functions (see Theorem~\ref{thm:fuconv}) and address the distinct limiting behavior for non-uniformly converging sequences in Section~\ref{subsec:examples}. 

\subsection{Uniform convergence for warped product spacetimes}

 In this subsection we state sufficient conditions on a sequence of warping functions $f_j$, $f_j \rightarrow f_{\infty}$, which guarantee that the sequence of corresponding warped product spacetimes with associated null distance structure converges in the GH and SWIF topology to the warped product spacetime with warping function $f_{\infty}$. The strategy of the proof is to show a general pointwise convergence result and then use the bi-Lipschitz bounds of Section~\ref{sect-Integral Current Spaces} to obtain the uniform, GH, and SWIF convergence result via a compactness result of Huang, Lee and Sormani (discussed in Section~\ref{subsec: GH and SWIF results}).
 
\begin{prop}\label{prop:pointwise convergence}
 Let $I$ be an interval and $(\Sigma,\bsigma)$ be a connected Riemannian manifold. Suppose $(f_j)_j$ is a sequence of bounded continuous functions $f_j \colon I \to (0,\infty)$ (uniformly bounded away from $0$) and $M_j = I \times_{f_j} \Sigma$ are warped products with Lorentzian metric tensors
 \[
  \g_j = -dt^2 + f_j(t)^2 \bsigma.
 \]
 Assume that $(f_j)_j$ converges uniformly to a limit function
 \[ f_\infty(t) = \lim_{j\to\infty} f_j(t). \]
 Then the corresponding null distances $\hat d_j$ of $\g_j$, $j \in \N \cup \{\infty\}$, with respect to the canonical time function $\tau(t,x)=t$ converge pointwise, i.e., for any $p,q \in M = I \times \Sigma$ it follows that
  \begin{align*}
   \lim_{j \to \infty} \hat d_j (p,q) = \hat d_\infty (p,q).
  \end{align*}
\end{prop}

\begin{proof}
 Due to the uniform convergence $f_j \to f_\infty$, it follows that $f_\infty$ is continuous and bounded as well. In particular, $f_\infty$ is bounded away from $0$ by some constant $f_{\infty,\min}>0$.
 
 Let $p,q \in M$ and $\varepsilon \in (0,\frac{f_{\infty,\min}}{4})$. Then, uniform convergence of $(f_j)_j$ implies for all $j$ sufficiently large, say, $j \geq j_0$,
 \begin{align}\label{fjuniform}
  \| f_j - f_\infty \|_{\infty} < \varepsilon.
 \end{align}
 
 We start with a general observation on suitably broken causal curves which holds for all $\g_j$ and the limit $\g_\infty$. For each $j \geq j_0$ we construct in Step 1 a particular set of subintervals related to a (broken) causal curve, which will be used in Step 2 and 3 to estimate the difference of $\hat d_j(p,q)$ to $\hat d_\infty(p,q)$.
 
\textit{Step 1: Construction of subintervals.} Fix $j \geq j_0$.
 Let $\beta_j$ be a (possibly broken) $\g_j$-causal curve such that
 \begin{align}\label{beta}
  \hat L (\beta_j) < \hat d_j (p,q) + \varepsilon.
 \end{align}
 Consider the time projection of $\beta_j$, more precisely, the compact interval $I_{j} := t(\operatorname{im}(\beta_j))$. Since $f_\infty$ is uniformly continuous on $I_{j}$ there exists a $\delta>0$ such that for all $s,t \in I_j$
 \begin{align}\label{uniform}
  |t-s| < 2\delta \Longrightarrow |f_\infty(t) - f_\infty(s)|<\varepsilon.
 \end{align}
 We can cover $I_{j}$ by finitely many $\delta$-intervals $(t_{2i}-\delta,t_{2i}+\delta)$ (numbered with increasing time). Choose furthermore $t_{2i+1} \in (t_{2i}-\delta,t_{2i}+\delta) \cap (t_{2i+2}-\delta,t_{2i+2}+\delta) \neq \emptyset$. Add the breaking points of $\beta_j$, i.e., the times $t_j = t(\beta_j(s_j))$ of the parameter values $s_j$ where $\beta_j$ changes from past- to future-directed and vice versa. Now, we consider all preimages $s_j$ of points on $\beta_j$ with times $t_j$ on $\beta_j$, and number the (possibly repeating) times according to their appearance on $\beta_j$, i.e., such that
 \[
  t(\beta_{j}(s_i)) = t_i.
 \]
 Without loss of generality we can assume that all breaking points have even parameter indices.
 Recall that due to the construction of the $s_i$'s and $t_i$'s and \eqref{uniform} we know, in particular, that
 \[
  f_\infty([t_{2i},t_{2i+2}]) \subseteq [f_\infty(t_{2i+1})-\varepsilon, f_\infty(t_{2i+1})+\varepsilon].
 \]
 Since $j \geq j_0$, \eqref{fjuniform} implies
 \[
  f_j([t_{2i},t_{2i+2}]) \subseteq [f_\infty(t_{2i+1})-2\varepsilon, f_\infty(t_{2i+1})+2\varepsilon].
 \]
 Note that for every $j \in \N \cup \{\infty\}$ we have
 \[
  \hat L(\beta_j) = \sum_i \hat d_j (\beta_{j}(s_{2i}),\beta_{j}(s_{2i-2})) = \sum_i |t_{2i} - t_{2i-2}|.
 \]
 In order to obtain upper and lower bounds of $\hat d_j (p,q)$ involving $\hat d_\infty(p,q)$ and $\varepsilon$, we estimate all summands $\hat d_j$ by employing Proposition~\ref{prop:biLip} in a clever way.
 
 \textit{Step 2: Upper bound.}
 Fix $j \geq j_0$ and consider the associated path $\beta_j$ and subintervals (indexed by $i$) from Step 1. Consider $\beta_\infty$ as constructed above with respect to $\g_\infty$ and $f_\infty$. We aim at estimating all parts $\hat d_\infty (\beta_\infty(s_{2i}),\beta_\infty(s_{2i-2}))$ by $\hat d_j (\beta_\infty(s_{2i}),\beta_\infty(s_{2i-2}))$, for all $j$ sufficiently large. To this end we note that
 \begin{align}\label{gjbsigma}
  \g_j &= -dt^2 + f_j(t)^2 \bsigma = -dt^2 + \frac{f_j(t)^2}{(f_\infty(t_{2i-1})-\varepsilon)^2} (f_\infty(t_{2i-1})-2\varepsilon)^2 \bsigma.
 \end{align}
 Here,
 \[
  \tilde \bsigma = (f_\infty(t_{2i-1})-2\varepsilon)^2 \bsigma
 \]
 is a conformal Riemannian metric, and by construction $\g_j \preccurlyeq \etab_{\tilde\bsigma}$ and $\g_\infty \preccurlyeq \etab_{\tilde\bsigma}$. Thus we can apply the right inequality of \eqref{dfestimate} in Proposition~\ref{prop:biLip} to obtain
 \begin{align*}
  \hat d_j (\beta_\infty& (s_{2i}),\beta_\infty(s_{2i-2})) \\
  &\leq \max \left\{ 1, \max_{t\in[t_{2i-2},t_{2i}]} \frac{f_j(t)}{f_\infty(t_{2i-1})-2\varepsilon} \right\} \hat d_{\tilde\bsigma} (\beta_\infty(s_{2i}),\beta_\infty(s_{2i-2})) \\
  &\leq \frac{f_\infty(t_{2i-1})+2\varepsilon}{f_\infty(t_{2i-1})-2\varepsilon} \, \hat d_{\tilde\bsigma} (\beta_\infty(s_{2i}),\beta_\infty(s_{2i-2})).
 \end{align*}
 Note that in the second line it is enough to consider the maximum on $[t_{2i-2},t_{2i}]$ because the only relevant cases in the proof of Proposition~\ref{prop:biLip} are Cases 1 and 3 where $\beta_{\infty}$ is causal with respect to $\etab_{\tilde \bsigma}$ by construction.
 Hence also
 \begin{align*}
  \hat d_{\tilde\bsigma} (\beta_\infty(s_{2i}),\beta_\infty(s_{2i-2})) = \hat d_\infty (\beta_\infty(s_{2i}),\beta_\infty(s_{2i-2})),
 \end{align*}
 and thus together
 \begin{align*}
  \hat d_j (\beta_\infty(s_{2i}),\beta_\infty(s_{2i-2})) \leq \frac{f_\infty(t_{2i-1})+2\varepsilon}{f_\infty(t_{2i-1})-2\varepsilon} \, \hat d_\infty (\beta_\infty(s_{2i}),\beta_\infty(s_{2i-2})).
 \end{align*}
 The constant can be estimated globally by
 \[
  \frac{f_\infty(t_{2i-1})+2\varepsilon}{f_\infty(t_{2i-1})-2\varepsilon} = 1 + \frac{4\varepsilon}{f_\infty(t_{2i-1})-2\varepsilon} \leq 1 + \frac{4\varepsilon}{f_{\infty,\min}-2\varepsilon} \leq 1 + \frac{8\varepsilon}{f_{\infty,\min}}.
 \]
 Summing up we have
 \begin{align*}
  \hat d_j (p,q) &\leq \sum_i \hat d_j (\beta_{\infty}(s_{2i}),\beta_{\infty}(s_{2i-2})) \\
  &\leq \left( 1 + \frac{8\varepsilon}{f_{\infty,\min}} \right) \sum_i \hat d_\infty (\beta_\infty(s_{2i}),\beta_\infty(s_{2i-2})) \\
  &= \left( 1 + \frac{8\varepsilon}{f_{\infty,\min}} \right) \hat L(\beta_{\infty})
 \end{align*}
 Thus by the choice of $\beta_\infty$, more precisely \eqref{beta}, this yields an upper bound for $\hat d_j$, 
 \begin{align}
  \hat d_j (p,q) &\leq  \left( 1 + \frac{8\varepsilon}{f_{\infty,\min}} \right) \left( \hat d_\infty(p,q) + \varepsilon \right) \nonumber \\
  &\leq \hat d_\infty(p,q) + \varepsilon \left( 1 +  \frac{8\varepsilon}{f_{\infty,\min}} + \frac{8}{f_{\infty,\min}} \hat d_\infty(p,q) \right). \label{upperbound}
 \end{align}
 Since $p,q$ is assumed to be fixed, the second summand is small for all sufficiently large $j$.
 
 \textit{Step 3: Lower bound.}
 Fix $j \geq j_0$ and consider the associated path $\beta_j$ and subintervals (indexed by $i$) from Step 1. As in the case of the upper bound we can write $\g_j$ and $\g_\infty$ in terms of $\tilde\bsigma$ via \eqref{gjbsigma}. Hence for all parts of $\beta_j$ we know that 
 \[
  \hat d_j(\beta_j(s_{2i}),\beta_j(s_{2i-2})) = |t_{2i}-t_{2i-2}| = \hat d_{\tilde\bsigma} (\beta_j(s_{2i}),\beta_j(s_{2i-2})).
 \]
 Moreover, since only the Cases 1 and 3 are relevant on $[t_{2i-2},t_{2i}]$ in the proof of Proposition~\ref{prop:biLip} we obtain that
 \begin{align*}
  \hat d_{\tilde\bsigma} (\beta_j&(s_{2i}), \beta_j(s_{2i-2})) \\
  &\geq \frac{1}{\max \left\{ 1, \max_{t \in [t_{2i-2},t_{2i}]} \frac{f_\infty(t)}{f_\infty (t_{2i-1})-2\varepsilon} \right\}} \, \hat d_\infty (\beta_j(s_{2i}),\beta_j(s_{2i-2})) \\
  &\geq \frac{f_\infty (t_{2i-1})-2\varepsilon}{f_\infty(t_{2i-1})+\varepsilon} \, \hat d_\infty (\beta_j(s_{2i}),\beta_j(s_{2i-2}))
 \end{align*}
 Similar to the previous case we know that
 \[
  \frac{f_\infty (t_{2i-1})-2\varepsilon}{f_\infty(t_{2i-1})+\varepsilon} = 1 - \frac{3\varepsilon}{f_\infty(t_{2i-1})+\varepsilon} \geq 1 - \frac{3\varepsilon}{f_{\infty,\min}}
 \]
 and thus 
 \[
  \hat d_j (\beta_j(s_{2i}), \beta_j(s_{2i-2})) \geq \left(1 - \frac{3\varepsilon}{f_{\infty,\min}}\right) \hat d_\infty (\beta_j(s_{2i}),\beta_j(s_{2i-2})).
 \]
 Due to the choice of $\beta_j$ in \eqref{beta} and the triangle inequality we obtain the pointwise lower bound
 \begin{align*}
  \hat d_j (p,q) &> \hat L(\beta_j) - \varepsilon \\
  &= \sum_i \hat d_{\tilde\bsigma} (\beta_j(s_{2i}), \beta_j(s_{2i-2})) - \varepsilon \\
  &\geq \left(1 - \frac{3\varepsilon}{f_{\infty,\min}}\right) \sum_i \hat d_\infty (\beta_j(s_{2i}),\beta_j(s_{2i-2})) - \varepsilon \\
  &\geq \left(1 - \frac{3\varepsilon}{f_{\infty,\min}}\right) \hat d_\infty (p,q) - \varepsilon,
 \end{align*}
 and hence
 \begin{align}\label{lowerbound}
  \hat d_j(p,q) \geq \hat d_\infty (p,q) - \varepsilon \left( 1 + \frac{3}{f_{\infty,\min}} \hat d_\infty (p,q) \right).
 \end{align}
 
  Combining the bounds \eqref{upperbound} and \eqref{lowerbound} we thus obtain the desired pointwise convergence
 \[
  \lim_{j\to\infty} \hat d_j(p,q) = \hat d_\infty(p,q). \qedhere
 \]
\end{proof}

 In what follows we establish the bi-Lipschitz bounds needed to for Theorem~\ref{HLS-thm}. While continuity of the warping functions $f_j$ and completeness of $\Sigma$ is crucial for the convergence result, they are not necessary to obtain these bounds.

\begin{prop}\label{prop:lambda}
 Let $(\Sigma,\bsigma)$ be a connected Riemannian manifold, and $I$ be an interval. Suppose $(f_j)_j$ is a sequence of uniformly bounded functions $f_j \colon I \to \R$ (bounded away from $0$) and $M_j = I \times_{f_j} \Sigma$ are warped products with Lorentzian metric tensors
 \[
  \g_j = -dt^2 + f_j(t)^2 \bsigma.
 \]
 Then there exists a constant $\lambda>1$ such that for any $j \in \N$
 \begin{align}\label{lestimate}
  \frac{1}{\lambda} \leq \frac{\hat d_j(p,q)}{\hat d_\bsigma(p,q)} \leq \lambda
 \end{align}
 on $M$, where $\hat d_j$ denotes the null distance corresponding to $f_j$ with respect to the canonical time function $\tau(t)=t$ and $\hat d_\bsigma$ denotes the null distance with respect to the reference Lorentzian product metric $\etab_\bsigma = -dt^2 +\bsigma$.
\end{prop}

\begin{proof}
 By assumption of uniform boundedness of $(f_j)_j$ there exists $\lambda>1$ such that
 \[
  \frac{1}{\lambda} \leq f_j(t) \leq \lambda, \qquad t \in I, \, j \in \N.
 \]
 Proposition~\ref{prop:biLip} then implies the bi-Lipschitz bounds
  \[
   \frac{1}{\lambda} \hat d_\bsigma(p,q) \leq \hat d_j(p,q) \leq \lambda \hat d_\bsigma(p,q), \qquad p,q \in M, \, j\in\N,
   \]
  which implies \eqref{lestimate} for $p\neq q$.
\end{proof}

\begin{rmrk}
 Instead of $\hat d_\bsigma$ we could have used any other reference null distance of which we know it yields a (local) integral current space, for instance, the null distance corresponding to Minkowski space or any of the $\g_j$ themselves. 
\end{rmrk}

 Using the bi-Lipschitz bounds \eqref{lestimate}, a theorem of Huang, Lee and Sormani \cite{HLS}*{Theorem A.1} implies convergence of a subsequence. Even the rate of convergence can be estimated (see the restated Theorem~\ref{HLS-thm} and Remark~\ref{rmrk-Boundary Allowed in HLS} in the Background Section~\ref{sect-Background} of this paper for more details).
 
\begin{cor}\label{cor:conv}
 Let $I$ be a closed interval, $(\Sigma,\bsigma)$ a connected compact Riemannian manifold and $(f_j)_j$ be given as in Proposition~\ref{prop:lambda}. Then there exists a subsequence $(\hat d_{j_k})_k$ of $(\hat d_j)_j$ and a length metric $d_\infty$ satisfying \eqref{lestimate} such that this subsequence converges uniformly to $d_\infty$, i.e.,
 \[
  \sup_{p,q \in M} | \hat d_{j_k}(p,q) - d_\infty(p,q)| \to 0, \qquad \text{as}~ k\to\infty.
 \]
 Moreover, the metric spaces $(M,\hat d_{j_k})_k$ converge to $(M,d_\infty)$ with respect to the  Gromov--Hausdorff distance, i.e.,
 \[
  d_{\mathrm{GH}} ((M,\hat d_{j_k}),(M,d_\infty)) \to 0, \qquad \text{as}~ k\to\infty,
 \]
 and the integral current spaces $(M,\hat d_{j_k},T)$ (with respect to the integral current structure of $\hat d_\bsigma$ obtained in Theorem~\ref{thm: Warp are integral currents}) converge to the integral current space $(M,d_\infty,T)$ with respect to the Sormani--Wenger intrinsic flat distance, i.e.,
  \[
  \pushQED{\qed} 
   d_\Fm ((M,\hat d_{j_k},T),(M,d_\infty,T)) \to 0, \qquad \text{as}~ k\to\infty. \qedhere
   \popQED
  \]
\end{cor}
 
 At this point we know nothing about the limiting distance function $d_\infty$, not even whether it arises as a null distance related to some kind of causal structure on $M$. We thus assume, in addition, uniform convergence and continuity of $(f_j)_j$ to be able to employ Proposition~\ref{prop:pointwise convergence}.
 
\begin{customthm}{\ref{thm:fuconv}}
 Let $I$ be a closed interval and $(\Sigma,\bsigma)$ be a connected compact Riemannian manifold. Suppose $(f_j)_j$ is a sequence of continuous functions $f_j \colon I \to (0,\infty)$ (uniformly bounded away from $0$) and $M_j = I \times_{f_j} \Sigma$ are warped products with Lorentzian metric tensors
 \[
  \g_j = -dt^2 + f_j(t)^2 \bsigma.
 \]
 Assume that $(f_j)_j$ converges uniformly to a limit function
 \[ f_\infty(t) = \lim_{j\to\infty} f_j(t). \]
 Then the corresponding null distances $\hat d_j$ of $\g_j$, $j \in \N$, with respect to the canonical time function $\tau(t,x)=t$, converge uniformly to $\hat d_\infty$ on $M$.
 Moreover, $(M,\hat d_j)$ converge to $(M,\hat d_\infty)$ in the Gromov--Hausdorff and Sormani--Wenger intrinsic flat topology, with $T$  the canonical integral current obtained in Theorem~\ref{thm: Warp are integral currents}.
\end{customthm}

\begin{proof}
 We show that $\hat d_\infty$ is the metric $d_\infty$ obtained in Corollary~\ref{cor:conv}. Note that uniform boundedness is implied by the assumption of uniform convergence $f_j \to f_\infty$, and that continuity and boundedness of $f_\infty>0$ follows as well.

 More precisely, Corollary~\ref{cor:conv} implies that a subsequence $(\hat d_{j_k})_k$ converges to some metric $d_\infty$ uniformly. Since we have shown pointwise convergence $\hat d_j \to \hat d_\infty$ in Proposition~\ref{prop:pointwise convergence}, we know that $d_\infty = \hat d_\infty$ is the uniform limit of $(\hat d_{j_k})_k$. Assume, for contradiction, that the full sequence does \emph{not} converge uniformly. Hence there must exist a subsequence $(\hat d_{j_l})_l$ so that
 \begin{align}
    d_\mathrm{unif}(\hat{d}_{j_l},\hat d_\infty) \ge C > 0 \label{ContradictionHyp}
 \end{align}
 for all $l \in \N$ (we omit $M$ since the space does not change).
 This subsequence still satisfies all the properties we have shown for the original pointwise convergent sequence and hence by the argument above we know that we can take a further subsequence so that
 \begin{align*}
    d_\mathrm{unif}(\hat{d}_{j_{l_n}},\hat d_\infty)  \rightarrow 0,
 \end{align*}
 which contradicts \eqref{ContradictionHyp} for infinitely many $l$. Hence we find the desired uniform convergence of $\hat{d}_j$ to $\hat d_\infty$. The GH and SWIF convergence follows from Theorem~\ref{HLS-thm} since uniform convergence implies GH and SWIF convergence.
\end{proof}
 
\begin{rmrk}
 Note that the proof of Proposition~\ref{prop:pointwise convergence} and Theorem~\ref{thm:fuconv} crucially relies on the continuity of the warping functions $f_j$ and that of the limit warping function $f_\infty$. In Section~\ref{subsec:examples} we will see that uniform convergence and continuity are not always necessary to obtain uniform, GH or SWIF convergence of $\hat d_j$ to $\hat d_\infty$, but none of those assumptions can be dropped for a general result of the type of Theorem~\ref{thm:fuconv}.
\end{rmrk}

\begin{rmrk}
 Note that in Theorem~\ref{thm:fuconv} we assume $\Sigma$ is compact (possibly with boundary). In the case of a noncompact sequence this result can be used to show convergence on compact sets in the standard way.
\end{rmrk} 

\subsection{Examples with non-uniformly converging warping functions}\label{subsec:examples}
 
 In this subsection we explore a few examples of sequences of warped product spacetimes in order to exhibit the variety of behavior we can observe in the limit under uniform, GH, and SWIF convergence. In particular, the sequences of warping functions $f_j$ we are investigating do not satisfy all assumptions in Theorem~\ref{thm:fuconv} and only converge pointwise and in the $L^p$ sense to semicontinuous limit functions $f_\infty$. In the first example we construct a sequence of warped product spacetimes with lower semicontinuous limit warping function which does not converge to the null distance of the limiting warped product spacetime. On the contrary, in the second example the (pointwise) limit warping function is upper semicontinuous and we do still obtain a uniform (and GH and SWIF) convergence result of the corresponding metric spaces. Finally, in the third example a sequence of warping functions converges to a degenerate limit warping function and while we show that the GH and SWIF limits still exist, they disagree.

\begin{ex}[Pointwise convergence of $(f_j)_j$ is not enough]\label{Ex: Necessity of Lower Bound}
 Fix a constant $h_0 \in (0,1)$ and a smooth increasing function $h \colon [0,1] \to (0,1]$ satisfying $h(t)=h_0$ for $t \in [0,\frac{1}{2}]$ and $h(1)=1$ with $h'(1)=0$. Consider the sequence of smooth warping functions $f_j:[0,2]\to [h_0,1]$ (see also Figure~\ref{fig:NecessityofLowerBound}), given by
 \[
 f_j(t)=
 \begin{cases}
 h(jt) & t\in [0, \frac{1}{j}],
 \\ 1 &t\in (\frac{1}{j}, 2].
 \end{cases}
\]

 For any fixed connected compact Riemannian manifold $(\Sigma,\bsigma)$, the sequence $(f_j)_j$ defines a sequence of smooth warped product Lorentzian metrics $\g_j$  with induced null distances $\hat{d}_j$ on the manifold\ $M = [0,2] \times \Sigma$.

 The pointwise limit of $(f_j)_j$ is the bounded function (see Figure~\ref{fig:NecessityofLowerBound})
 \[
  f_\infty (t) = \begin{cases}
                             h_0 & t=0, \\
                             1 & t\in (0,2],
                            \end{cases}
 \]
 with the corresponding null distance $\hat d_\infty = \hat d_\bsigma$.

 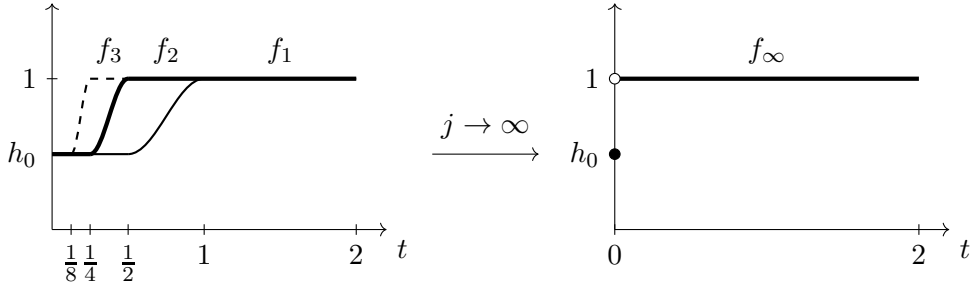
\begin{figure}[h]
  \centering
 \begin{tikzpicture}[scale=2]
  \draw[->] (0,0) -- (2.2,0) node[anchor=north west] {$t$};
  \draw[->] (0,0) -- (0,1.5) node[anchor=south east] {};
  
  \foreach \x/\xtext in {0.125/\frac{1}{8}, 0.25/\frac{1}{4}, 0.5/\frac{1}{2}, 1/1, 2/2, 3.7/0, 5.7/2}
    \draw[shift={(\x,0)}] (0pt,1pt) -- (0pt,-1pt) node[below] {$\xtext$};

  \draw[thick] (1,1) -- (2,1);
  \node[above, outer sep=2pt] at (1.5,1) {$f_1$};
  \draw[thick] (0,1/2) -- (1/2,1/2);
  \node[left, outer sep=2pt] at (0,.5) {$h_0$};
  \draw[thick] (1/2,1/2) cos (3/4,3/4) sin (1,1);
  
  \draw[ultra thick] (0,1/2) -- (1/4,1/2);
  \draw[ultra thick] (1/2,1) -- (2,1);
  \draw[ultra thick] (1/4,1/2) cos (3/8,3/4) sin (1/2,1);
  \node[above, outer sep=2pt] at (.75,1) {$f_2$};
  
  \draw[thick, dashed] (0,1/2) -- (1/8,1/2);
  \draw[thick, dashed] (1/4,1) -- (2,1);
  \draw[thick, dashed] (1/8,1/2) cos (3/16,3/4) sin (1/4,1);
  \node[above, outer sep=2pt] at (3/8,1) {$f_3$};
  
  \node[left, outer sep=2pt] at (0,1) {$1$};
  \draw[shift={(0,1)}] (-1pt,0pt) -- (1pt,0pt);
  
  \draw[->] (2.5,.5) -- (3.2,.5) node[anchor=north west] {};
  \node[above, outer sep=2pt] at (2.85,0.5) {$j\to\infty$}; 
  
  \draw[->] (3.7,0) ++(0,0) -- +(2.2,0) node[anchor=north west] {$t$};
  \draw[->] (3.7,0) ++(0,0)-- (3.7,1.5) node[anchor=south east] {};
  \draw[fill] (3.7,0) ++(0,.5) circle [radius=0.035];
  \node[left, outer sep=2pt] at (3.7,.5) {$h_0$};
  \node[left, outer sep=2pt] at (3.7,1) {$1$};
  \draw[ultra thick] (3.7,0) ++(0,1) -- +(2,0);
  \draw[draw=black, fill=white] (3.7,0) ++(0,1) circle [radius=0.035];
  \node[above, outer sep=2pt] at (4.7,1) {$f_\infty$};
 \end{tikzpicture}
 \caption{Sequence of pointwise converging warping functions $f_j$ in Example~\ref{Ex: Necessity of Lower Bound}.}
 \label{fig:NecessityofLowerBound}
 \end{figure} 

 A uniform limit (even of a subsequence) of $(f_j)_j$ clearly does not exist, and hence Corollary~\ref{cor:conv} and Theorem~\ref{thm:fuconv} are not applicable. We will show that $d_\infty=\hat d_\bsigma$ is \emph{not} a limit of any subsequence of null distances $(\hat d_j)_j$.

 In fact, we will show that $(\hat d_j)_j$ converges to a another metric space with metric $d_0$. 

 \textbf{Claim.} The sequence of null distances $\hat d_j$induced by the warping functions $f_j$ converges uniformly to the metric $d_0 \neq \hat d_\infty = \hat d_\bsigma$, for points $p=(t(p),p_\Sigma),q=(t(q),q_\Sigma) \in [0,2] \times \Sigma$ given by
 \begin{align*}
  d_0(p,q) := \min\left\{\hat d_\bsigma(p,q), t(p)+t(q) + h_0 \, \hat d_\bsigma(J_{\etab_\bsigma}^-(p)\cap  \Sigma_0,J_{\etab_\bsigma}^-(q)\cap  \Sigma_0)\right\}
 \end{align*}
 where $\Sigma_{t} = \{t\}\times \Sigma$.
 In addition $(M,\hat d_j)$ converges in the GH and SWIF sense to $(M,d_0)$.
\end{ex}

\begin{proof}
 First we notice that since $f_j$ is uniformly bounded from above and below we have Lipschitz bounds on $\hat d_j$ by Proposition \ref{prop:biLip}, and hence by Theorem \ref{HLS-thm} we know that a subsequence must converge in the uniform, GH and SWIF sense to a metric $d_\infty$ satisfying the same bi-Lipschitz bounds (see Corollary~\ref{cor:conv}). In what follows we show that $d_\infty=d_0$ is the pointwise limit of $(\hat d_j)_j$, and subsequently prove uniform convergence. Notice that it is clear that $d_0 \neq \hat d_\infty$ since they only agree when $t(p)$ and $t(q)$ are positive.
 
 Let $p=(t(p),p_\Sigma)$ and $q=(t(q),q_\Sigma)$ for $p_\Sigma,q_\Sigma \in \Sigma$. First notice that since $h_0 \le f_j \le 1$ we have by Proposition~\ref{prop:biLip} that
 \begin{align}\label{biLip1}
  h_0 \hat d_{\bsigma}(p,q) \le \hat d_j(p,q) \le \hat d_{\bsigma}(p,q).
 \end{align}

 \textit{Case 1: $t(p)=t(q)=0$.} In this case, $d_0(p,q) = h_0 \, d_\bsigma(p_\Sigma,q_\Sigma)$ and we show that $\hat d_j \to d_0$ pointwise. We can construct a piecewise null curve $\beta_j$ with respect to $\g_j$ which remains within the strip $S_j := [0,\frac{1}{2j}] \times \Sigma$ in order to calculate
 \begin{align*}
  \hat d_j(p,q) \le \hat L_j(\beta_j) = h_0 d_{\bsigma}(p_\Sigma,q_\Sigma), % = d_0(p,q).
 \end{align*}
 so together with \eqref{biLip1} we immediately have that \[ \lim_{j\to\infty}\hat d_j(p,q) = h_0 d_{\bsigma}(p_\Sigma,q_\Sigma) = d_0(p,q)\] is satisfied.

 \textit{Case 2: $t(p)>0$ and $t(q)=0$.} In this case,
  \[ d_0(p,q) = t(p) + h_0 \, \hat d_\bsigma(J_{\etab_\bsigma}^-(p)\cap  \Sigma_0,q). \]
 The proof of pointwise convergence is obtained by curves minimizing their contributions in both the $\Sigma_0$ and the $(0,2] \times \Sigma$ region, i.e.,
 \begin{align}\label{d0min}
  d_0(p,q) = \inf_{q' \in \{0\} \times \Sigma} \left( \hat d_\bsigma(p,q') + h_0 d_\bsigma(q'_\Sigma,q_\Sigma) \right).
 \end{align}
 Note that if $q \in J^-_{\etab_\bsigma}(p)$ then the minimum is obtained by choosing $q'=q$ because the second term vanishes and the first one is minimal due to \eqref{biLip1}. If $q \not \in J^-_{\etab_\bsigma}(p)$ then both terms are minimal for $q'$ at the intersection of the light cone with $\Sigma$. It remains to show that \eqref{d0min} is indeed the limit of $(\hat d_j)_j$.

 We consider two cases: (a) If $q \in J^-_{\etab_\bsigma}(p)$, then for any $j \in \N$ also $q \in J^-_j(p)$ (since $f_j \leq f_{j+1} \leq f_\infty$ and thus $\g_\infty \preccurlyeq \g_{j+1} \preccurlyeq \g_j$ by Example~\ref{ex:order}) and therefore
 \[
  \hat d_j(p,q) = t(p) - t(q) = t(p) = d_0(p,q).
 \]
 In particular, $d_\infty(p,q) := \lim_{j \to \infty} \hat d_j(p,q) = d_0(p,q)$.
 
 (b) Suppose now that $q \not\in J^-_{\etab_\bsigma}(p)$. For any $q' \in (J^-_{\etab_\bsigma}(p) \setminus I^-_{\etab_\bsigma}(p)) \cap (\{0\} \times \Sigma)$ there exists a $\g_\infty$-null curve $\beta^1$ such that
 \[
  \hat L(\beta^1) = t(p) - t(q') = t(p) = \hat d_\bsigma(p,q') = d_0(p,q'),
 \]
 Since $q' \in J^-_{\M_\bsigma}(p) \subseteq J^-_j(p)$, also
 \begin{align}\label{dj}
  \hat d_j(p,q') = t(p).  
 \end{align}
 Moreover, for every such $q'$ there exists a broken causal curve $\beta^2_j$ in $[0,\frac{1}{j}] \times \Sigma$ such that by Lemma~\ref{d:warped}
 \[
   \hat L(\beta^2_j) \leq \hat d_j(q',q) + \frac{1}{j} = h_0 d_\bsigma(q'_\Sigma,q_\Sigma) + \frac{1}{j}.
 \]
 Thus by the triangle inequality
 \begin{align*}
  \hat d_j(p,q) 
   &\leq \inf_{q' \in \Sigma_0} \left( \hat d_j(p,q') + \hat d_j(q',q) \right) \\
   &\leq t(p) + h_0 \inf_{q' \in \Sigma_0} \hat d_\bsigma(q'_\Sigma,q_\Sigma) + \frac{1}{j} \\
   &= d_0(p,q) + \frac{1}{j},
 \end{align*}
 and therefore
 \begin{align}\label{limsup}
  \limsup_{j\to\infty} \hat d_j(p,q) \leq d_0(p,q).
 \end{align}
 
 On the other hand, there is a broken causal curve $\beta_j$ from $p$ to $q$ with respect to $\g_j$ that satisfies
 \[
  \hat L(\beta_j) \leq \hat d_j(p,q) + \frac{1}{j}
 \]
 Due to the continuity of $\beta_j$ and $t$ (and since $t(p)>0=t(q)$), for $j$ sufficiently large, there exists a last point $q_j$ along the image of $\beta_j$ in $M$ satisfying $t(q_j)=\frac{1}{j}$. Then
 \[
  \hat d_j(q_j,q) \geq t(q_j) - t(q) = \frac{1}{j}.
 \]
 We add $q_j$ to the break points $\{x_i\}_{i=0}^N$ of $\beta_j$. Thus
 \begin{align*}
  \hat d_j(p,q) + \frac{1}{j} &\geq \hat L(\beta_j) = \sum_i |t(x_i)-t(x_{i+1})| \\
                         &\geq t(p) - t(q_j) + \hat d_j(q_j,q)  \\
                         &\stackrel{\eqref{biLip1}}{\geq} t(p) - \frac{1}{j} + h_0 \hat d_\bsigma(q_j,q) \\
                         &= t(p) - \frac{1}{j} + h_0 d_\bsigma(q_{j,\Sigma},q_{\Sigma}) \\
                         &\geq d_0(p,q) - \frac{1}{j},
 \end{align*}
 and therefore
 \begin{align}\label{liminf}
  \liminf_{j\to\infty} \hat d_j(p,q) \geq d_0(p,q).
 \end{align}
 Together, \eqref{limsup} and \eqref{liminf} imply the desired result
 \[
  \lim_{j\to\infty} \hat d_j(p,q) = d_0(p,q)
 \]
 for all $p,q$ in Case 2 pointwise.
 
 \textit{Case 3: $t(p), t(q)>0$.} 
We distinguish two cases. If $q \in J^\pm_{\etab_\bsigma}(p) \subseteq J^\pm_j(p)$ (or vice versa), then it follows immediately that
 \[
  \hat d_j(p,q) = |t(p)-t(q)| = d_0(p,q).
 \]
 Next we assume that $q \not\in J^\pm_{\etab_\bsigma}(p)$.
Since $\etab_\bsigma  \preccurlyeq\g_j$ we know by Lemma~\ref{lem: biLip Inequality} that $\hat d_j(p,q) \le \hat d_\bsigma(p,q)$. For any $p' \in J_{\etab_\bsigma}^-(p)\cap  \Sigma_0$ and $q' \in J_{\etab_\bsigma}^-(q)\cap  \Sigma_0$ the triangle inequality implies
 \begin{align*}
     \hat d_j(p,q) &\le \hat d_j(p,p')+\hat d_j(p',q')+\hat d_j(q',q),
 \end{align*}
 and hence by Case 1 and Case 2 we find
 \begin{align*}
     \limsup_{j \rightarrow \infty}\hat{d}_j(p,q) \le t(p) + h_0 \hat{d}_{\bsigma}(p',q') + t(q),
 \end{align*}
 which implies
 \begin{align}\label{d0est0}
      \limsup_{j \rightarrow \infty}\hat{d}_j(p,q) \le d_0(p,q).
 \end{align}
 It remains to show that the value is attained. If $J_{\etab_\bsigma}^-(p) \cap J_{\etab_\bsigma}^-(q) \cap [\frac{1}{j},2] \neq \emptyset$ then for $j$ (and all larger indices)
 \[
  \hat d_j(p,q) = \hat d_\bsigma(p,q),
 \]
 and thus furthermore
 \begin{align}\label{d0est1}
  \lim_{j\to\infty} \hat d_j(p,q) = \hat d_\bsigma(p,q) \geq d_0(p,q),
 \end{align}
and we are done. If, on the other hand, $J_{\etab_\bsigma}^-(p) \cap J_{\etab_\bsigma}^-(q) \cap (0,2] = \emptyset$, then for $j$ sufficiently large such that $t(p),t(q)>\frac{1}{j}$ we consider two points $p_j \in J_{\etab_\bsigma}^-(p)\cap  \Sigma_{\frac{1}{j}}$, $q_j \in J_{\etab_\bsigma}^-(q)\cap  \Sigma_{\frac{1}{j}}$
and the auxiliary function $\tilde f_j \colon [0,2] \to \R$, defined by
 \[
  \tilde f_j(t) = \begin{cases}
                                 h_0 & \text{if } t\leq \frac{1}{j}, \\
                                 1 & \text{if } t < \frac{1}{j}.
                                \end{cases}
 \]
This function $\tilde f_j$ defines a bounded warping function that satisfies $\tilde f_j \leq f_j$, and hence by Example~\ref{ex:order} we know that
 \[
  \hat d_{\tilde j}(p,q) \leq \hat d_j(p,q).
 \]
 Let $\beta_j$ be a curve that connects $p$ to $p_j$ to $q_j$ (the latter within $[0,\frac{1}{j}] \times \Sigma$) to $q$. Then
 \begin{align*}
  \inf_{p_j,q_j} \hat L(\beta_j) &= t(p) - \frac{1}{j} + h_0 \inf_{p_j,q_j} \hat d_\bsigma (p_j,q_j) + t(q) - \frac{1}{j} \\
  &= t(p) + t(q) - \frac{2}{j} + h_0 \hat d_\bsigma (J_{\etab_\bsigma}^-(p)\cap  \Sigma_{\frac{1}{j}},J_{\etab_\bsigma}^-(q)\cap  \Sigma_{\frac{1}{j}})
 \end{align*}
 and 
 \begin{align*}
 \hat{d}_{\tilde{j}} = \min\{ \hat{d}_{\bsigma}(p,q), \inf_{p_j,q_j} \hat L(\beta_j)\}
 \end{align*}
 thus
 \begin{align}\label{d0est2}
  &\liminf_{j\to\infty} \hat d_j(p,q) 
  \\&\quad \geq \min\{ \hat{d}_{\bsigma}(p,q),t(p) + t(q) + h_0 \hat d_\bsigma (J_{\etab_\bsigma}^-(p)\cap  \Sigma_{0},J_{\etab_\bsigma}^-(q)\cap  \Sigma_{0})\} \nonumber 
  \\ &\quad = d_0(p,q).\nonumber 
 \end{align}
Together, \eqref{d0est0}--\eqref{d0est2} establish that also if $q \not\in J^\pm_{\etab_\bsigma}(p)$ (or vice versa) we have
  \begin{align*}
      \lim_{j \rightarrow \infty}\hat{d}_j(p,q) =d_0(p,q).
 \end{align*}
 
 Since a subsequence of $\hat{d}_j$ must converge in the uniform, GH, and SWIF sense and we have now shown pointwise convergence of $\hat{d}_j(p,q) \rightarrow d_0(p,q)$ in all cases, we can conclude that $(\hat d_j)_j$ itself must converge in the uniform, GH, and SWIF sense to $d_0$. For the sake of contradiction assume that $\hat{d}_j$ does not converge in the uniform, GH, or SWIF sense to $d_0$ and hence there must exist a subsequence $(\hat d_{j_l})_l$ that is bounded away from $(M,d_0)$ in the uniform, GH and SWIF sense. However, the subsequence $(\hat d_{j_l})_l$ still satisfies all the properties shown for the original Example~\ref{Ex: Necessity of Lower Bound}, and thus we can take a further subsequence $(\hat d_{j_{l_n}})_n$ of $(\hat d_{j_l})_l$ must converge in the uniform, GH and SWIF topology---and yet again to the only possible limit $d_0$, a contradiction.
\end{proof}

\begin{rmrk}
 The above Example~\ref{Ex: Necessity of Lower Bound} shows that pointwise convergence of $(f_j)_j$ is generally not enough to obtain the desired convergence of the corresponding null distances $(\hat d_j)_j$ to the null distance of the limiting function $f_\infty$. This illustrates our assumption of uniform convergence in Theorem \ref{thm:fuconv}. The setting can easily be extended to an interval $[-2,2]$ by extending $h$ as an even function on $[-2,0]$. Similarly, it can be extended to $\R \times \Sigma$ by extending it with constant value $1$.

 Besides, note that $L^p$ convergence, for $p\geq 1$ is also not enough, since the $L^p$ limit of $(f_j)_j$ is
 \[ f^p_\infty(t) = 1 \qquad \text{a.e.} \]
 with the same corresponding null distance $\hat d_\infty^p = \hat d_\bsigma$ as the pointwise limit.
\end{rmrk}

 While in Example~\ref{Ex: Necessity of Lower Bound} we have seen that the convergence does not hold if the null cone of $\hat d_\infty$ is narrower than those of the sequence $(\hat{d}_j)_j$, this is not the case when the null cone of $\hat{d}_{\infty}$ is wider than the null cone of the sequence $(\hat{d}_j)_j$. Indeed, we will be able to deduce that the sequence does converge to the null distance of a warped product spacetime in such a setting. As such, uniform convergence of $(f_j)_j$ is not necessary to obtain uniform convergence of the corresponding null distances.

\begin{ex}[Uniform convergence of $(f_j)_j$ is not necessary]\label{Ex: Upper Bound Unecessary}
 Let $h_0 >1$ and consider a smooth function $h \colon [0,1] \to [1,h_0]$ defined as in Example~\ref{Ex: Necessity of Lower Bound} except that it is now decreasing instead of increasing between $h(t)=h_0$ for $t\in[0,\frac{1}{2}]$ and $h(1)=1$, with $h'(1)=0$. Consider the sequence of smooth functions $f_j(t)\colon [0,2]\to [1,h_0]$ (see also Figure~\ref{fig:UpperBoundUnnecessary}), given by
 \begin{align*}
 f_j(t)=
 \begin{cases}
 h(jt) & t\in[0, \frac{1}{j}],
 \\ 1 &t\in (\frac{1}{j}, 2].
 \end{cases}
\end{align*}

 For any fixed connected compact Riemannian manifold $(\Sigma, \bsigma)$, the sequence $(f_j)_j$ defines a sequence of smooth Lorentzian metrics $\g_j$ with induced null distances $\hat{d}_j$ on the manifolds $M= [0,2]\times \Sigma$.
 
 The pointwise limit of $(f_j)_j$ is the bounded function (see Figure~\ref{fig:UpperBoundUnnecessary})
 \[
  f_\infty(t) = \begin{cases}
                 h_0 & t = 0, \\
                 1 & t \in (0,2],
                \end{cases}
 \]
 with induced null Lorentzian metric $\g_\infty$ and corresponding null distance $\hat d_\infty = \hat d_\bsigma$.

 \begin{figure} [h]
  \centering
  \begin{tikzpicture}[scale=2]
  \draw[->] (0,0) -- (2.2,0) node[anchor=north west] {$t$};
  \draw[->] (0,0) -- (0,2) node[anchor=south east] {};
  
  \foreach \x/\xtext in {0.125/\frac{1}{8}, 0.25/\frac{1}{4}, 0.5/\frac{1}{2}, 1/1, 2/2, 3.7/0, 5.7/2}
    \draw[shift={(\x,0)}] (0pt,1pt) -- (0pt,-1pt) node[below] {$\xtext$};

  \draw[thick] (1,1) -- (2,1);
  \node[above, outer sep=2pt] at (1.5,1) {$f_1$};
  \draw[thick] (0,1+1/2) -- (1/2,1+1/2);
  \node[left, outer sep=2pt] at (0,1.5) {$h_0$};
    \draw[shift={(0,1)}] (-1pt,0pt) -- (1pt,0pt);
    \node[left, outer sep=2pt] at (0,1) {$1$};
  \draw[thick] (1/2,1+1/2) cos (3/4,1+1/4) sin (1,1);
  
  \draw[ultra thick] (0,1+1/2) -- (1/4,1+1/2);
  \draw[ultra thick] (1/2,1) -- (2,1);
  \draw[ultra thick] (1/4,1+1/2) cos (3/8,1+1/4) sin (1/2,1);
  \node[below, outer sep=2pt] at (.75,1) {$f_2$};
  
  \draw[thick, dashed] (0,1+1/2) -- (1/8,1+1/2);
  \draw[thick, dashed] (1/4,1) -- (2,1);
  \draw[thick, dashed] (1/8,1+1/2) cos (3/16,1+1/4) sin (1/4,1);
  \node[below, outer sep=2pt] at (3/8,1) {$f_3$};
  
  \draw[->] (2.5,.5) -- (3.2,.5) node[anchor=north west] {};
  \node[above, outer sep=2pt] at (2.85,0.5) {$j\to\infty$}; 
  
  \draw[->] (3.7,0) ++(0,0) -- +(2.2,0) node[anchor=north west] {$t$};
  \draw[->] (3.7,0) ++(0,0)-- (3.7,2) node[anchor=south east] {};
  \draw[fill] (3.7,0) ++(0,1.5) circle [radius=0.035];
  \node[left, outer sep=2pt] at (3.7,1.5) {$h_0$};
  \node[left, outer sep=2pt] at (3.7,1) {$1$};
  \draw[ultra thick] (3.7,0) ++(0,1) -- +(2,0);
  \draw[draw=black, fill=white] (3.7,0) ++(0,1) circle [radius=0.035];
  \node[above, outer sep=2pt] at (4.7,1) {$f_\infty$};
  \end{tikzpicture}
  \caption{Sequence of pointwise converging warping functions $f_j$ of Example~\ref{Ex: Upper Bound Unecessary}.}
  \label{fig:UpperBoundUnnecessary}
 \end{figure}
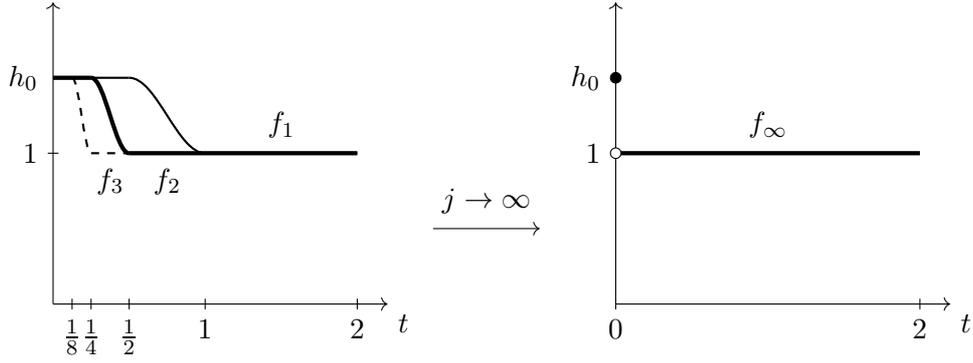 

 \textbf{Claim.} The null distances $\hat{d}_j$ converge uniformly to $\hat{d}_{\infty}$, and $(M,\hat d_j)$ converges in the GH and SWIF sense to $(M, \hat d_{\infty})$.
\end{ex}

\begin{proof}
 First we notice that since $f_j$ is uniformly bounded from above and below we have Lipschitz bounds on $\hat{d}_j$ by Proposition \ref{prop:biLip} and hence by Theorem \ref{HLS-thm} we know that a subsequence must converge in the uniform, GH and SWIF sense to a metric $d$ which satisfies the same bi-Lipschitz bounds. Our goal now is to show that the full sequence $(\hat d_j)_j$ converges pointwise to $d = \hat{d}_{\infty}$, which implies uniform convergence using the same the proof as in Example~\ref{Ex: Necessity of Lower Bound}.

 Let $p=(t(p),p_\Sigma)$ and $q=(t(q),q_\Sigma)$ in $M$, where $p_\Sigma,q_\Sigma$ is the projection onto $\Sigma$. First notice that since $1 \le f_j \le h_0$ we have by Proposition~\ref{prop:biLip} that
 \begin{align}\label{biLip2}
  \hat d_\infty(p,q) = \hat{d}_{\bsigma}(p,q) \le \hat d_j(p,q) \le h_0 \hat d_{\bsigma}(p,q).
 \end{align}
 and thus
 \begin{align*}
  \hat d_\infty(p,q) \leq \liminf_{j \rightarrow \infty} \hat d_j(p,q).
 \end{align*}
 It remains to be shown that
 \begin{align}\label{supbound}
  \limsup_{j \rightarrow \infty} \hat d_j(p,q) \le \hat{d}_{\infty}(p,q).
 \end{align}

\textit{Case 1: $t(p), t(q) >0$.}
 For $j$ sufficiently large, $t(p)$ and $t(q)$ are larger than $\frac{1}{j}$. For any such $j$ there exists a broken causal curve $\beta_j$ such that
 \[
  \hat L_\infty(\beta_j) \leq \hat d_\infty(p,q) + \frac{1}{j}.
 \]
 Without loss of generality we can assume that $\operatorname{im} (\beta_j) \subseteq [\frac{1}{j},2] \times \Sigma$, which implies that $\hat L_j(\beta_j) = \hat L_\infty(\beta_j)$ and therefore, for all $j$ sufficiently large,
\begin{align*}
 \hat d_j(p,q) \leq \hat{L}_{\infty}(\beta_j) \leq \frac{1}{j} + \hat{d}_{\infty}(p,q),
\end{align*}
 which implies \eqref{supbound}.
 
\textit{Case 2: $t(p)=t(q)=0$.}
 Let $p_j = (\frac{1}{j},p_\Sigma)$ and $q_j=(\frac{1}{j},q_\Sigma)$. By the triangle inequality, and the fact that $p_j \in J^+(p)$, $q_j \in J^+(q)$ (with respect all $\g_i$ and also $\g_\infty$)
 \begin{align}
  \hat d_j(p,q) &\leq \hat d_j(p,p_j) + \hat d_j(p_j,q_j) + \hat d_j(q_j,q) \nonumber \\
               &= \frac{2}{j} + \hat d_j(p_j,q_j). \label{case22}
 \end{align}
 Moreover, we can connect $p_j$ and $q_j$ by broken $\g_\infty$-causal curves $\beta_j$ in $[\frac{1}{j},2] \times \Sigma$ (which have the same null length as those in $[0,2]\times \Sigma$), hence by Lemma~\ref{lem:nulld1}
 \[
  \hat d_j(p_j,q_j) \leq \hat d_\bsigma(p_j,q_j) = d_\bsigma (p_\Sigma,q_\Sigma) = d_\infty(p,q).
 \]
 Thus together with \eqref{case22} we have
 \[
  \hat d_j(p,q) \leq \frac{2}{j} + d_\infty(p,q),
 \]
 which implies \eqref{supbound}. 

 \textit{Case 3: $t(p)>0, \, t(q)=0$.} For any $j$ sufficiently large such that $t(p) > \frac{1}{j}$ we define $q_j=(\frac{1}{j},q_\Sigma)$. Then due to the triangle inequality,
 \begin{align*}
  \hat d_j(p,q) &\leq \hat d_j(p,q_j) + \hat d_j(q,q_j) \\
                &\leq \hat d_\bsigma(p,q_j) + \frac{1}{j},
\end{align*}
 since we can connect $p$ and $q_j$ by broken $\g_\infty$-causal curves in $[\frac{1}{j},2] \times \Sigma$ which have the same null length as those in $[0,2]\times\Sigma$. Applying again the triangle inequality for $\hat d_\bsigma$ implies
\begin{align*}
 \hat d_j(p,q) &\leq \hat d_\bsigma(p,q) + \hat d_\bsigma(q_j,q) + \frac{1}{j} \\
               &= \hat d_\infty(p,q) + \frac{2}{j},
\end{align*}
 and thus the desired estimate \eqref{supbound}.

 As in the previous Example~\ref{Ex: Necessity of Lower Bound} we can show that the pointwise limit $\hat d_\infty$ is the uniform limit of $(\hat d_j)_j$.
\end{proof}

\begin{rmrk}
 As Example~\ref{Ex: Necessity of Lower Bound}, also Example~\ref{Ex: Upper Bound Unecessary} can be extended to $[-2,2]\times \Sigma$ or even $\R \times \Sigma$. Note that considering the $L^p$ limit $f^p_\infty(t) = 1$ a.e.\ also yields uniform convergence of null distances, since $\hat d^p_\infty$ is the same as the corresponding pointwise null distance $\hat d_\infty = \hat d_\bsigma$.
\end{rmrk}

 In Example~\ref{Ex: Necessity of Lower Bound} and Example~\ref{Ex: Upper Bound Unecessary} we have seen that uniform convergence, GH and SWIF hold and that all limiting distance functions are the same (although, as in Example~\ref{Ex: Necessity of Lower Bound}, not necessarily related to the null distance of the limiting warped product). In the last example we construct a sequence of warped product spacetimes for which the GH and SWIF limits disagree. For general metric spaces we know that the GH and SWIF limits can disagree (see examples in \cite{LS,S,SW}), so it is important to have examples where this happens for spacetimes equipped with the null distance. Note that for such a counterexample the conditions of Corollary~\ref{cor:conv} must be violated.

\begin{ex}[GH and SWIF limits need not agree]\label{GHandSWIFdisagree}
 Let $h \colon [0,1]\to[0,1]$ be a smooth increasing function that satisfies $h(0)=0$, $h(1)=1$ and $h'(0)=h'(1)=0$. Consider the sequence $(f_j)_j$ of smooth increasing functions $f_j \colon [0,2] \to [\frac{1}{j},1]$ (see Figure~\ref{fig:GHandSWIFdisagree}), given by
 \begin{align*}
  f_j(t) = \begin{cases}
            \frac{1}{j} & t \in [0,1-\frac{1}{j}], \\
            \frac{1}{j} + \frac{j-1}{j} \, h\left(jt-(j-1)\right) & t \in (1-\frac{1}{j},1), \\
            1 & t \in [1,2].
           \end{cases}
 \end{align*}
 For any fixed connected compact Riemannian manifold $(\Sigma,\bsigma)$, the sequence $(f_j)_j$ defines a sequence of smooth Lorentzian warped product metrics  $\g_j$ with induced null distances $\hat d_j$ on the manifold $M = [0,2] \times \Sigma$.  
 
 Since $f_{j+1} \leq f_{j}$ for all $j$, by Example~\ref{ex:order} we have $\g_{j} \preccurlyeq \g_{j+1}$ and 
 \begin{align}\label{stackd}
  \hat d_{j+1} (p,q) \leq \hat d_{j} (p,q), \qquad p,q \in M.
 \end{align}
 The pointwise limit of $(f_j)_j$ is the bounded discontinuous function (see also Figure~\ref{fig:GHandSWIFdisagree})
 \[
  f_\infty(t) = \begin{cases}
                0 & t \in [0,1), \\
                1 & t \in [1,2],
               \end{cases}
 \]
 which clearly does \emph{not} induce a nondegenerate warped product metric on $M$, and hence no null distance. We investigate whether the sequence $(\hat d_j)_j$ converges to a limiting distance function with respect to different notions of convergence.
 
 \begin{figure}
  \centering
  \begin{tikzpicture}[scale=2]
  \draw[->] (0,0) -- (2.2,0) node[anchor=north west] {$t$};
  \draw[->] (0,0) -- (0,1.5) node[anchor=south east] {};
  
  \foreach \x/\xtext in {0/0, 1/1, 0.5/\frac{1}{2}, 0.666/\frac{2}{3}, 2/2, 4.7/1, 5.7/2}
    \draw[shift={(\x,0)}] (0pt,1pt) -- (0pt,-1pt) node[below] {$\xtext$};
  \foreach \t/\ttext in {1/1, 0.5/\frac{1}{2}, 0.333/\frac{1}{3}~}
 \draw[shift={(0,\t)}] node[left] {$\ttext$};
    
  \draw[thick] (0,1) -- (2,1);
  \node[above, outer sep=2pt] at (1.5,1) {$f_1$};
  
  \draw[ultra thick] (0,1/2) -- (1/2,1/2);
  \draw[ultra thick] (1,1) -- (2,1);
  \draw[ultra thick] (1/2,1/2) cos (3/4,3/4) sin (1,1);
  \node[above, outer sep=2pt] at (0.5,0.5) {$f_2$};
  
  \draw[thick, dashed] (0,1/3) -- (2/3,1/3);
  \draw[thick, dashed] (1,1) -- (2,1);
  \draw[thick, dashed] (2/3,1/3) cos (5/6,2/3) sin (1,1);
  \node[right, outer sep=2pt] at (0.75,0.4) {$f_3$};
  
  \draw[->] (2.5,.5) -- (3.2,.5) node[anchor=north west] {};
  \node[above, outer sep=2pt] at (2.85,0.5) {$j\to\infty$}; 
  
  \draw[->] (3.7,0) ++(0,0) -- +(2.2,0) node[anchor=north west] {$t$};
  \draw[->] (3.7,0) ++(0,0)-- +(0,1.5) node[anchor=south east] {};
  \draw[shift={(3.7,1)}] (-1pt,0pt) -- (1pt,0pt) node[left] {$1$};
  \draw[ultra thick] (3.7,0) ++(0,0) -- +(1,0);
  \draw[draw=black, fill=white] (3.7,0) ++(1,0) circle [radius=0.035];
  \draw[ultra thick] (3.7,0) ++(1,1) -- +(1,0);
  \draw[fill] (3.7,0) ++(1,1) circle [radius=0.035];
  \node[above, outer sep=2pt] at (5.2,1) {$f_\infty$};
  \end{tikzpicture}
  \caption{Sequence of warping functions $f_j$ of Example~\ref{GHandSWIFdisagree}.}
  \label{fig:GHandSWIFdisagree}
 \end{figure}
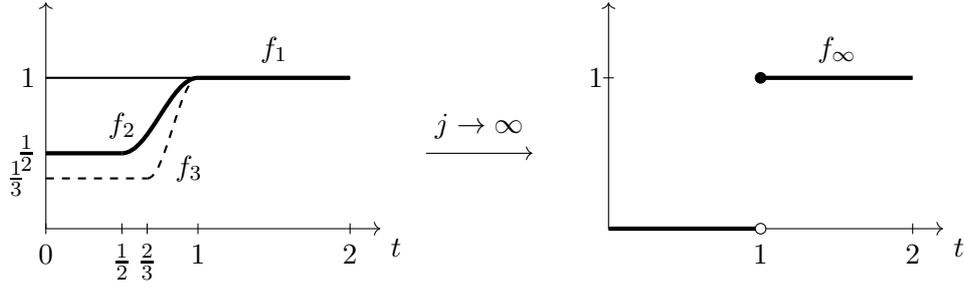
 
 Fix $p_0 \in \Sigma$ and let
 \begin{align*}
  M_\infty &= ([0,1] \times \{p_0\}) \sqcup (1,2] \times \Sigma),
  \\M'&= \{1\}\times \{p_0\} \sqcup ((1,2]\times \Sigma) ,
 \end{align*}
 with distance function
 \[
  d_\infty(p,q) = \begin{cases}
                   \min \{ \hat d_\bsigma (p,q), t(p)-1+t(q)-1 \} & t(p)> 1,t(q)> 1, \\
                   |t(p)-t(q)| & \text{else}.
                  \end{cases}
 \]
 
 \textbf{Claim.} We have that
 \begin{align*}
 (M, \hat d_j) &\GHto (M_{\infty},d_\infty), \\
 (M, \hat d_j) &\Fto (X, d_\infty),
 \end{align*}
 where $X \subseteq M'$, possibly the zero space.
\end{ex}
 
\begin{proof}
 By Example \ref{ex:order} we know that $(\hat d_j)_j$ is a decreasing sequence of null distances on $M = [0,2] \times \Sigma$, which is bounded below by $0$. Thus the Monotone Convergence Theorem and \eqref{stackd} implies that $\hat d_j$ converges pointwise to a function $d_0$ on $M \times M$ such that for any $p,q \in M$
 \begin{align}\label{ineq2}
  d_0(p,q) \leq \hat d_j (p,q) \leq \hat d_1 (p,q)=\hat d_\bsigma(p,q).
 \end{align}
 If we can show that $d_0$ is continuous, then Dini's Theorem implies uniform convergence $\hat d_j \to d_0$ on $M \times M$. We distinguish several cases to prove this. Let $p=(t(p),p_\Sigma),q=(t(q),q_\Sigma) \in [0,2] \times \Sigma$.
 
 \textit{Case 1: $t(p), t(q) < 1$.} By choosing $j$ large enough we can ensure that $t(p),t(q) < 1-\frac{1}{j}$ and hence we know that by Proposition~\ref{prop:biLip} and Remark~\ref{rmrk:biLip} (note that case 3 cannot occur because $\etab_\bsigma = \g_1 \preccurlyeq \g_j$)
 \[
  \hat d_j(p,q) = \max \left\{ |t(p)-t(q)|, \tfrac{1}{j} \hat d_\bsigma(p,q) \right\},
 \]
 since we can always restrict our attention to broken causal curves with image in $[0,1-\frac{1}{j}] \times \Sigma$. Since $\frac{1}{j} \hat d_\bsigma(p,q)\to 0$, we obtain a pointwise limit
 \[
  \lim_{j \to \infty} \hat d_j(p,q) = |t(p)-t(q)|.
 \]

 \textit{Case 2: $t(p), t(q)\ge 1$.} Let $p'_j = (1-\frac{1}{j},p_\Sigma)$ and $q'_j = (1-\frac{1}{j},q_\Sigma)$. Then due to the triangle inequality
 \begin{align*}
  |t(p)-t(q)| \leq \hat d_j(p,q) &\leq \hat d_j(p,p'_j) + \hat d_j(p'_j,q'_j) + \hat d_j(q'_j,q) \\
                       &= |t(p)-1|+ \frac{1}{j} d_\bsigma(p_\Sigma,q_\Sigma) + |t(q)-1| ,
 \end{align*}
 thus
 \[ \limsup_{j\to\infty} \hat d_j(p,q) \leq |t(p)-1| + |t(q)-1|. \]
 in this case. Since by \eqref{ineq2},
 \[ \limsup_{j\to\infty} \hat d_j(p,q) \leq \hat d_\bsigma(p,q), \]
 we obtain
 \begin{align}\label{limsupest}
     \limsup_{j\to\infty} \hat d_j(p,q) \leq \min\left\{\hat d_\bsigma(p,q),|t(p)-1| + |t(q)-1| \right\}.
 \end{align}
 
 It remains to be shown that these values are indeed attained. If $\hat{d}_\bsigma(p,q)\le|t(p)-1|+|1-t(q)|$ then we note that a $\hat{d}_{\bsigma}$ (almost) length-minimizing curve is contained in $[1,2]\times \Sigma$ with the same  null length with respect to $\g_j$. Furthermore, for any piecewise causal curve $\beta$ which enters the region $[0,1]\times \Sigma$ the null length with respect to $\g_j$ is at least as long as 
 \begin{align*}
   \hat L_j(\beta) &\ge |t(p)-1|+|t(q)-1| \ge \hat{d}_\bsigma(p,q).
 \end{align*}
 Since these two classes exhaust all possible curves we find
 \begin{align*}
  \hat{d}_j(p,q) \ge \hat{d}_\bsigma(p,q).
 \end{align*}
 If, on the other hand, $\hat{d}_\bsigma(p,q)\ge|t(p)-1|+|1-t(q)|$ then any curve which remains in $[1,2]\times \Sigma$ will have $\g_j$ null length equal to $\hat{d}_\bsigma(p,q)$ and any curve which enters the region $[0,1]\times \Sigma$  will have $\g_j$ null length at least as long as $|t(p)-1|+|t(q)-1|$ so that 
 \begin{align}\label{Case6LowerBoundPart2}
  \hat{d}_j(p,q)& \ge |t(p)-1|+|1-t(q)|.
 \end{align}
 Hence we obtain a limiting statement involving the right hand side of \ref{limsupest}, and hence together
 \begin{align*}
  \lim_{j \rightarrow \infty} \hat{d}_j(p,q) =  \min \left\{ \hat{d}_\bsigma(p,q),|t(p)-1|+|1-t(q)|\right\}.
 \end{align*}

 \textit{Case 3: $t(p)\geq 1,\,t(q) < 1$.} For sufficiently large $j$ we know that $t(q) < 1 -\frac{1}{j}$. We include the auxiliary point $p' = (t(q),p_\Sigma)$. Then
 \begin{align*}
  |t(p)-t(q)| \leq \hat d_j(p,q) &\leq \hat d_j(p,p') + \hat d_j(p',q) \\
                       &=|t(p)-t(q)| + \frac{1}{j} \hat d_\bsigma (p',q),
 \end{align*}
 which yields
 \[
  \lim_{j\to\infty} \hat d_j(p,q) = |t(p)-t(q)|.
 \]

 Combining Cases 1 to 3 we see that $\hat d_j \to d_0$ uniformly and, therefore, $d_0$ is also continuous. Note that symmetry and the triangle inequality follow for $d_0$ by taking limits but $d_0$ is clearly not definite on $[0,1] \times \Sigma$. Hence we identify points whose $d_0$ distance is $0$ in order to define the metric space $M_{\infty}$ with the distance function $d_{\infty}$. Now by the observation in \cite{BBI}*{Example 7.4.4} we find that
 \begin{align}
     (M,\hat{d}_j) \GHto (M_{\infty},d_{\infty}).\label{GHConvEq}
 \end{align}
 Finally, we show that the SWIF limit exists and is a proper subset of the GH limit. To this end we note that  by \cite{SW}*{Lemma 2.43}
 \begin{align*}
    \mass \left((M,\hat{d}_j)\right)& \le \frac{2^m}{\omega_m} \mathcal{H}^m(M,\hat{d}_j),
    \\\mass \left((\partial M,\hat{d}_j)\right)& \le \frac{2^{m-1}}{\omega_{m-1}} \mathcal{H}^{m-1}(\partial M,\hat{d}_j),
 \end{align*}
 where $\omega_m$, $m \in \N$, is the volume of the unit ball in $\R^m$ and $\mathcal{H}^m(X,d)$ is the $m$-dimensional Hausdorff measure of $(X,d)$. Due to \eqref{ineq2} we see that for any $U \subseteq M$
 \begin{align*}
   \diam_{\hat{d}_j}(U) &\le \diam_{\hat{d}_{\bsigma}}(U),
 \end{align*}
 which implies
 \begin{align*}
  \mathcal{H}^m(M,\hat{d}_j) &\le \mathcal{H}^m(M,\hat{d}_{\bsigma}) \le V_0,
  \\ \mathcal{H}^{m-1}(\partial M,\hat{d}_j) &\le \mathcal{H}^{m-1}(\partial M,\hat{d}_{\bsigma}) \le A_0.
 \end{align*}
 By \cite{SW}*{Theorem 3.20} we know that a SWIF limit $(X,d_{\infty})$ of the sequence $(M,\hat d_j)_j$ exists where $X$ is possibly the zero space. Since the SWIF limit must be $n+1$-dimensional we know that $X \subseteq M'$ and hence a proper subset of the GH limit.
\end{proof}
 
\begin{rmrk}
 In Example \ref{GHandSWIFdisagree} we note that the expected SWIF limit is $X = (M', d_{\infty})$. We do not go through the work of explicitly showing this result here since new machinery which can be used to estimate the SWIF distance between spacetimes with the null distance should be developed first. Since a lot of the literature on estimating the SWIF distance has involved Riemannian manifolds (see \cite{BAB, AS, Lak, LS, LeeS}) it is necessary to first investigate extensions of these results for Lorentzian manifolds equipped with the null distance. Nonetheless, our results on warped product spacetimes are a first major step towards understanding spacetime convergence with respect to the null distance, and highlight the potential of this notion in geometric analysis and general relativity.
\end{rmrk}

%%%%%%%%%%%%%%%%%%%%%%%%%%%%%%%%%%%%%%%%%%%%%%%%%%%%%%%%%%%%%%%%%%%%%%%%%%%%%%%%%%%%%%%%%%%%%%%%%%%%

\section*{Funding}

This work was supported by the National Science Foundation [DMS 1612049 to B.A.]; and the Dutch Research Council [VI.Veni.192.208 to A.B.]. A.B. also gratefully acknowledges support from the Simons Center for Geometry and Physics, Stony Brook University, to attend two workshops in 2018 and 2019 mentioned below.

\section*{Acknowledgements}

The authors are grateful to the organizers (Piotr Chru\'{s}ciel, Richard Schoen, Christina Sormani, Mu-Tao Wang, and Shing-Tung Yau) of the SCGP Workshop ``Mass in General Relativity" which took place in March 2018. This is where the authors first met. B.A. is grateful to Xiaochun Rong for the invitation to speak at the Rutgers Geometry/Topology Seminar where the authors met for a second time and initiated the collaboration on this project. A.B. is also grateful to the organizers (Mike Anderson, Jeff Jauregui, Philippe LeFloch, Christina Sormani) for the invitation to the SCGP Workshop ``Convergence and Low Regularity in General Relativity" in April 2019. In particular, both authors would like to thank Christina Sormani for continuing encouragement and discussions. The authors are also grateful to Michael Kunzinger and Roland Steinbauer for their interest and careful reading of a previous version. Finally, many thanks to the referee for useful comments and suggestions.

%%%%%%%%%%%%%%%%%%%%%%%%%%%%%%%%%%%%%%%%%%%%%%%%%%%%%%%%%%%%%%%%%%%%%%%%%%%%%%%%%%%%%%%%%%%%%%%%%%%% 


\begin{thebibliography}{99}
\providecommand{\url}[1]{\texttt{#1}}
\providecommand{\urlprefix}{URL }

\bib{AB}{article}{
   author={Alexander, S.B.},
   author={Bishop, R.L.},
   title={Lorentz and semi-Riemannian spaces with Alexandrov curvature
   bounds},
   journal={Comm. Anal. Geom.},
   volume={16},
   date={2008},
   number={2},
   pages={251--282},
%    issn={1019-8385},
%    review={\MR{2425468}},
}

\bib{BA1}{article}{
    AUTHOR = {Allen, B.},
     TITLE = {I{MCF} and the stability of the {PMT} and {RPI} under {$L^2$}
              convergence},
   JOURNAL = {Ann. Henri Poincar\'{e}},
%  FJOURNAL = {Annales Henri Poincar\'{e}. A Journal of Theoretical and
%              Mathematical Physics},
    VOLUME = {19},
      YEAR = {2018},
    NUMBER = {4},
     PAGES = {1283--1306},
      ISSN = {1424-0637},
%    MRCLASS = {53C20 (35R01 53C21 53C44 81Q05 81Q12 81Q80)},
%   MRNUMBER = {3775157},
% MRREVIEWER = {Otis Chodosh},
%        DOI = {10.1007/s00023-017-0641-7},
%        URL = {https://doi.org/10.1007/s00023-017-0641-7},
}

\bib{BA2}{article}{
    AUTHOR = {Allen, B.},
     TITLE = {Stability of the {PMT} and {RPI} for asymptotically hyperbolic
              manifolds foliated by {IMCF}},
   JOURNAL = {J. Math. Phys.},
    VOLUME = {59},
      YEAR = {2018},
    NUMBER = {8},
     PAGES = {082501, 18},
      ISSN = {0022-2488},
%    MRCLASS = {53C20 (53C44)},
%   MRNUMBER = {3844923},
% MRREVIEWER = {Otis Chodosh},
%        DOI = {10.1063/1.5035275},
%        URL = {https://doi.org/10.1063/1.5035275},
}

\bib{BA3}{article}{
    AUTHOR = {Allen, B.},
     TITLE = {Sobolev stability of the positive mass theorem and
              {R}iemannian {P}enrose inequality using inverse mean curvature
              flow},
   JOURNAL = {Gen. Relativity Gravitation},
%  FJOURNAL = {General Relativity and Gravitation},
    VOLUME = {51},
      YEAR = {2019},
    NUMBER = {5},
     PAGES = {Art. 59, 32},
      ISSN = {0001-7701},
%    MRCLASS = {53C20 (53C44)},
%   MRNUMBER = {3946444},
%        DOI = {10.1007/s10714-019-2542-1},
%        URL = {https://doi.org/10.1007/s10714-019-2542-1},
}

\bib{BA4}{article}{
  Author={Allen, B.},
  Title={Stability of the positive mass theorem using inverse mean  curvature flow},
  journal={Arxiv e-prints},
  date={2018},
  eprint={https://arxiv.org/abs/1807.08822},
}

\bib{BAB}{article}{
    AUTHOR = {Allen, B.},
    AUTHOR= {Bryden, E.},
     TITLE = {Sobolev bounds and convergence of {R}iemannian manifolds},
   JOURNAL = {Nonlinear Anal.},
%  FJOURNAL = {Nonlinear Analysis. Theory, Methods \& Applications. An
%              International Multidisciplinary Journal},
    VOLUME = {185},
      YEAR = {2019},
     PAGES = {142--169},
      ISSN = {0362-546X},
%    MRCLASS = {58D17},
%   MRNUMBER = {3925055},
% MRREVIEWER = {Nikolai K. Smolentsev},
%        DOI = {10.1016/j.na.2019.03.001},
%        URL = {https://doi.org/10.1016/j.na.2019.03.001},
}


\bib{AS}{article}{
   author={Allen, B.},
   author={Sormani, C.},
   title={Contrasting various notions of convergence in geometric analysis},
   journal={Pacific J. Math.},
   volume={303},
   year={2019},
   number={1},
   pages={1--46},
   % archivePrefix = "arXiv",
  % eprint = {1803.06582},
}
 
\bib{AK}{article}{
    AUTHOR = {Ambrosio, L.},
    AUTHOR = {Kirchheim, B.},
     TITLE = {Currents in metric spaces},
   JOURNAL = {Acta Math.},
  %FJOURNAL = {Acta Mathematica},
    VOLUME = {185},
      YEAR = {2000},
    NUMBER = {1},
     PAGES = {1--80},
      ISSN = {0001-5962},
  %   CODEN = {ACMAA8},
%    MRCLASS = {49Q15 (49Q20)},
%   MRNUMBER = {MR1794185 (2001k:49095)},
% MRREVIEWER = {Giovanni Bellettini},
}

\bib{AW}{article}{
   author={An, X.},
   author={Wong, W.W.Y.},
   title={Warped product space-times},
   journal={Classical Quantum Gravity},
   volume={35},
   date={2018},
   number={2},
   pages={025011, 25},
%    issn={0264-9381},
%    review={\MR{3755944}},
%    doi={10.1088/1361-6382/aa8af7},
}

\bib{A1}{article}{
   author={Anderson, M.T.},
   title={On long-time evolution in general relativity and geometrization of
   3-manifolds},
   journal={Comm. Math. Phys.},
   volume={222},
   date={2001},
   number={3},
   pages={533--567},
%    issn={0010-3616},
%    review={\MR{1888088}},
%    doi={10.1007/s002200100527},
}

% \bib{MR1982778}{article}{
%    author={Anderson, M.T.},
%    title={Regularity for Lorentz metrics under curvature bounds},
%    journal={J. Math. Phys.},
%    volume={44},
%    date={2003},
%    number={7},
%    pages={2994--3012},
%    issn={0022-2488},
%    review={\MR{1982778}},
%    doi={10.1063/1.1580199},
% }

\bib{A2}{article}{
   author={Anderson, M.T.},
   title={Cheeger-Gromov theory and applications to general relativity},
   conference={
      title={The Einstein equations and the large scale behavior of
      gravitational fields},
   },
   book={
      publisher={Birkh\"{a}user, Basel},
   },
   date={2004},
   pages={347--377},
%    review={\MR{2098921}},
}

\bib{AnB}{article}{
   author={Andersson, L.},
   author={Burtscher, A.Y.},
   title={On the asymptotic behavior of static perfect fluids},
   journal={Ann. Henri Poincar\'{e}},
   volume={20},
   date={2019},
   number={3},
   pages={813--857},
   issn={1424-0637},
%    review={\MR{3916963}},
%    doi={10.1007/s00023-018-00758-z},
}

\bib{AGH}{article}{
   author={Andersson, L.},
   author={Galloway, G.J.},
   author={Howard, R.},
   title={The cosmological time function},
   journal={Classical Quantum Gravity},
   volume={15},
   date={1998},
   number={2},
   pages={309--322},
%    issn={0264-9381},
%    review={\MR{1606594}},
%    doi={10.1088/0264-9381/15/2/006},
}

\bib{AH}{article}{
   author={Andersson, L.},
   author={Howard, R.},
   title={Comparison and rigidity theorems in semi-Riemannian geometry},
   journal={Comm. Anal. Geom.},
   volume={6},
   date={1998},
   number={4},
   pages={819--877},
%    issn={1019-8385},
%    review={\MR{1664893}},
%    doi={10.4310/CAG.1998.v6.n4.a8},
}

\bib{BLSS}{article}{
   author={Barnes, A.P.},
   author={Lefloch, P.G.},
   author={Schmidt, B.G.},
   author={Stewart, J.M.},
   title={The Glimm scheme for perfect fluids on plane-symmetric Gowdy
   spacetimes},
   journal={Classical Quantum Gravity},
   volume={21},
   date={2004},
   number={22},
   pages={5043--5074},
%    issn={0264-9381},
%    review={\MR{2103241}},
%    doi={10.1088/0264-9381/21/22/003},
}

\bib{BEE}{book}{
   author={Beem, J.K.},
   author={Ehrlich, P.E.},
   author={Easley, K.L.},
   title={Global Lorentzian geometry},
   series={Monographs and Textbooks in Pure and Applied Mathematics},
   volume={202},
   edition={2},
   publisher={Marcel Dekker, Inc., New York},
   date={1996},
   pages={xiv+635},
%    isbn={0-8247-9324-2},
%    review={\MR{1384756}},
}

\bib{BS}{article}{
    AUTHOR = {Bernal, A.N.},
    author = {S\'{a}nchez, M.},
     TITLE = {Smoothness of time functions and the metric splitting of globally hyperbolic spacetimes},
   JOURNAL = {Comm. Math. Phys.},
 % FJOURNAL = {Communications in Mathematical Physics},
    VOLUME = {257},
      YEAR = {2005},
    NUMBER = {1},
     PAGES = {43--50},
%       ISSN = {0010-3616},
%    MRCLASS = {53C50 (83C20)},
%   MRNUMBER = {2163568},
% MRREVIEWER = {Paul E. Ehrlich},
%        DOI = {10.1007/s00220-005-1346-1},
%        URL = {https://doi.org/10.1007/s00220-005-1346-1},
}

\bib{BSu}{article}{
   author={Bernard, P.},
   author={Suhr, S.},
   title={Lyapounov functions of closed cone fields: from Conley theory to
   time functions},
   journal={Comm. Math. Phys.},
   volume={359},
   date={2018},
   number={2},
   pages={467--498},
   issn={0010-3616},
%    review={\MR{3783554}},
%    doi={10.1007/s00220-018-3127-7},
}

\bib{BF}{article}{
    AUTHOR = {Bray, H.},
    AUTHOR= {Finster, F.},
     TITLE = {Curvature estimates and the positive mass theorem},
   JOURNAL = {Comm. Anal. Geom.},
%  FJOURNAL = {Communications in Analysis and Geometry},
    VOLUME = {10},
      YEAR = {2002},
    NUMBER = {2},
     PAGES = {291--306},
      ISSN = {1019-8385},
%    MRCLASS = {53C21 (53C24)},
%   MRNUMBER = {1900753},
% MRREVIEWER = {Yuguang Shi},
%        DOI = {10.4310/CAG.2002.v10.n2.a3},
%        URL = {https://doi.org/10.4310/CAG.2002.v10.n2.a3},
}


\bib{BKS}{article}{
  author={Bryden, E.},
  author={Khuri, M.},
  author={Sormani, C.},
  title={Stability of the spacetime positive mass theorem in spherical symmetry},
  journal={Arxiv e-prints},
  date={2019},
  eprint={https://arxiv.org/abs/1906.11352},
}

\bib{BBI}{book}{
   author={Burago, D.},
   author={Burago, Y.},
   author={Ivanov, S.},
   title={A course in metric geometry},
   series={Graduate Studies in Mathematics},
   volume={33},
   publisher={American Mathematical Society, Providence, RI},
   date={2001},
   pages={xiv+415},
%    isbn={0-8218-2129-6},
%    review={\MR{1835418}},
%    doi={10.1090/gsm/033},
}

\bib{BGP}{article}{
   author={Burago, Y.},
   author={Gromov, M.},
   author={Perel\cprime man, G.},
   title={A. D. Aleksandrov spaces with curvatures bounded below},
   language={Russian, with Russian summary},
   journal={Uspekhi Mat. Nauk},
   volume={47},
   date={1992},
   number={2(284)},
   pages={3--51, 222},
   issn={0042-1316},
   translation={
      journal={Russian Math. Surveys},
      volume={47},
      date={1992},
      number={2},
      pages={1--58},
      issn={0036-0279},
   },
%    review={\MR{1185284}},
%    doi={10.1070/RM1992v047n02ABEH000877},
}

\bib{B}{article}{
   author={Burtscher, A.Y.},
   title={Length structures on manifolds with continuous Riemannian metrics},
   journal={New York J. Math.},
   volume={21},
   date={2015},
   pages={273--296},
%    issn={1076-9803},
%    review={\MR{3358543}},
}

\bib{Bp}{article}{
   author={Burtscher, A.Y.},
   title={Initial data and black holes for matter models},
   conference={
      title={Hyperbolic problems: theory, numerics, applications},
   },
   book={
      series={AIMS Ser. Appl. Math.},
      volume={10},
      publisher={Am. Inst. Math. Sci. (AIMS), Springfield, MO},
   },
   date={2020},
   pages={336--345},
}

\bib{BGH}{article}{
    AUTHOR = {Burtscher, A.},
    AUTHOR = {Garc\'{\i}a-Heveling, L.},
%     TITLE = {Spacetime intrinsic flat convergence},
     Journal={in preparation},
}

\bib{BKTZ}{article}{
   author={Burtscher, A.},
   author={Kiessling, M.K.-H.},
   author={Tahvildar-Zadeh, A.S.},
   title={Weak second Bianchi identity for spacetimes with timelike singularities},
   journal={Arxiv e-prints},
  % archivePrefix = "arXiv",
  % eprint = {1901.00813},
   date={2019},
   eprint={https://arxiv.org/abs/1901.00813},
}

\bib{BLF}{article}{
   author={Burtscher, A.Y.},
   author={LeFloch, P.G.},
   title={The formation of trapped surfaces in spherically-symmetric
   Einstein-Euler spacetimes with bounded variation},
   journal={J. Math. Pures Appl. (9)},
   volume={102},
   date={2014},
   number={6},
   pages={1164--1217},
%    issn={0021-7824},
%    review={\MR{3277438}},
%    doi={10.1016/j.matpur.2014.10.003},
}

\bib{ACP}{article}{
   author={Cabrera Pacheco, A.J.},
   title={On the stability of the positive mass theorem for asymptotically
   hyperbolic graphs},
   journal={Ann. Global Anal. Geom.},
   volume={56},
   date={2019},
   number={3},
   pages={443--463},
%    issn={0232-704X},
%    review={\MR{4009763}},
%    doi={10.1007/s10455-019-09674-9},
}

\bib{CC}{article}{
   author={Cheeger, J.},
   author={Colding, T.H.},
   title={Lower bounds on Ricci curvature and the almost rigidity of warped products},
   journal={Ann. of Math. (2)},
   volume={144},
   date={1996},
   number={1},
   pages={189--237},
%    issn={0003-486X},
%    review={\MR{1405949}},
%    doi={10.2307/2118589},
}

\bib{CC1}{article}{
   author={Cheeger, J.},
   author={Colding, T.H.},
   title={On the structure of spaces with Ricci curvature bounded below. I},
   journal={J. Differential Geom.},
   volume={46},
   date={1997},
   number={3},
   pages={406--480},
%    issn={0022-040X},
%    review={\MR{1484888}},
}


\bib{CL}{article}{
   author={Chen, G.-Q.G.},
   author={Li, S.},
   title={Weak continuity of the {C}artan structural system on semi-{R}iemannian manifolds with lower regularity},
   journal={Arxiv e-prints},
   date={2019},
   eprint={https://arxiv.org/abs/1905.02661}
}

\bib{Chr}{article}{
   author={Christodoulou, D.},
   title={Bounded variation solutions of the spherically symmetric
   Einstein-scalar field equations},
   journal={Comm. Pure Appl. Math.},
   volume={46},
   date={1993},
   number={8},
   pages={1131--1220},
%    issn={0010-3640},
%    review={\MR{1225895}},
%    doi={10.1002/cpa.3160460803},
}

\bib{CG}{article}{
   author={Chru\'{s}ciel, P.T.},
   author={Grant, J.D.E.},
   title={On Lorentzian causality with continuous metrics},
   journal={Classical Quantum Gravity},
   volume={29},
   date={2012},
   number={14},
   pages={145001, 32},
%    issn={0264-9381},
%    review={\MR{2949547}},
%    doi={10.1088/0264-9381/29/14/145001},
}

\bib{CGM}{article}{
    AUTHOR = {Chru\'{s}ciel, P.T.},
    author = {Grant, J.D.E.},
    author = {Minguzzi, E.},
     TITLE = {On differentiability of volume time functions},
   JOURNAL = {Ann. Henri Poincar\'{e}},
 % FJOURNAL = {Annales Henri Poincar\'{e}. A Journal of Theoretical and Mathematical Physics},
    VOLUME = {17},
      YEAR = {2016},
    NUMBER = {10},
     PAGES = {2801--2824},
%       ISSN = {1424-0637},
%    MRCLASS = {58C35},
%   MRNUMBER = {3546988},
% MRREVIEWER = {Juan Manuel P\'{e}rez-Pardo},
%        DOI = {10.1007/s00023-015-0448-3},
%        URL = {https://doi.org/10.1007/s00023-015-0448-3},
}

\bib{DafL}{article}{
   author={Dafermos, M.},
   author={Luk, J.},
   title={The interior of dynamical vacuum black holes I: The $C^0$-stability of the {K}err {C}auchy horizon},
   journal={Arxiv e-prints},
   date={2017},
   eprint={https://arxiv.org/abs/1710.01722},
}

\bib{DM}{article}{
   author={D'Andrea, F.},
   author={Martinetti, P.},
   title={On Pythagoras theorem for products of spectral triples},
   journal={Lett. Math. Phys.},
   volume={103},
   date={2013},
   number={5},
   pages={469--492},
%    issn={0377-9017},
%    review={\MR{3041755}},
%    doi={10.1007/s11005-012-0598-x},
}


\bib{E}{article}{
   author={Eschenburg, J.-H.},
   title={The splitting theorem for space-times with strong energy
   condition},
   journal={J. Differential Geom.},
   volume={27},
   date={1988},
   number={3},
   pages={477--491},
%    issn={0022-040X},
%    review={\MR{940115}},
}

\bib{FF}{article}{
    AUTHOR = {Federer, H.}
    AUTHOR = {Fleming, W.H.},
     TITLE = {Normal and integral currents},
   JOURNAL = {Ann. of Math. (2)},
 % FJOURNAL = {Annals of Mathematics. Second Series},
    VOLUME = {72},
      YEAR = {1960},
     PAGES = {458--520},
      ISSN = {0003-486X},
%    MRCLASS = {28.80 (53.45)},
%   MRNUMBER = {MR0123260 (23 \#A588)},
% MRREVIEWER = {L. C. Young},
}

\bib{FelF}{article}{
    AUTHOR = {Finster, F.},
     TITLE = {A level set analysis of the {W}itten spinor with applications
              to curvature estimates},
   JOURNAL = {Math. Res. Lett.},
%  FJOURNAL = {Mathematical Research Letters},
    VOLUME = {16},
      YEAR = {2009},
    NUMBER = {1},
     PAGES = {41--55},
      ISSN = {1073-2780},
%    MRCLASS = {53C21 (53C27 53C80 58J60)},
%   MRNUMBER = {2480559},
% MRREVIEWER = {J. Carlos D\'{\i}az-Ramos},
%        DOI = {10.4310/MRL.2009.v16.n1.a5},
%        URL = {https://doi.org/10.4310/MRL.2009.v16.n1.a5},
}

\bib{FK}{article}{
    AUTHOR = {Finster, F.}
    AUTHOR= {Kath, I.},
     TITLE = {Curvature estimates in asymptotically flat manifolds of
              positive scalar curvature},
   JOURNAL = {Comm. Anal. Geom.},
  %FJOURNAL = {Communications in Analysis and Geometry},
    VOLUME = {10},
      YEAR = {2002},
    NUMBER = {5},
     PAGES = {1017--1031},
      ISSN = {1019-8385},
%    MRCLASS = {53C21 (53C20 53C27 83C57)},
%   MRNUMBER = {1957661},
% MRREVIEWER = {Bernd Ammann},
%        DOI = {10.4310/CAG.2002.v10.n5.a6},
%        URL = {https://doi.org/10.4310/CAG.2002.v10.n5.a6},
}

\bib{F}{article}{
   author={Flores, J.L.},
   title={The Riemannian and Lorentzian splitting theorems},
   conference={
      title={Topics in modern differential geometry},
   },
   book={
      series={Atlantis Trans. Geom.},
      volume={1},
      publisher={Atlantis Press, [Paris]},
   },
   date={2017},
   pages={1--20},
%    review={\MR{3588304}},
}

\bib{FS}{article}{
   author={Flores, J.L.},
   author={S\'{a}nchez, M.},
   title={The causal boundary of wave-type spacetimes},
   journal={J. High Energy Phys.},
   date={2008},
   number={3},
   pages={036, 43},
   issn={1126-6708},
 %  review={\MR{2391084}},
 %  doi={10.1088/1126-6708/2008/03/036},
}

\bib{G1}{article}{
   author={Galloway, G.J.},
   title={Splitting theorems for spatially closed space-times},
   journal={Comm. Math. Phys.},
   volume={96},
   date={1984},
   number={4},
   pages={423--429},
%    issn={0010-3616},
%    review={\MR{775039}},
}

\bib{G2}{article}{
   author={Galloway, G.J.},
   title={The Lorentzian splitting theorem without the completeness assumption},
   journal={J. Differential Geom.},
   volume={29},
   date={1989},
   number={2},
   pages={373--387},
%    issn={0022-040X},
%    review={\MR{982181}},
}

\bib{GLS}{article}{
   author={Galloway, G.J.},
   author={Ling, E.},
   author={Sbierski, J.},
   title={Timelike completeness as an obstruction to $C^0$-extensions},
   journal={Comm. Math. Phys.},
   volume={359},
   date={2018},
   number={3},
   pages={937--949},
%    issn={0010-3616},
%    review={\MR{3784536}},
%    doi={10.1007/s00220-017-3019-2},
}

\bib{GL}{article}{
   author={Graf, M.},
   author={Ling, E.},
   title={Maximizers in Lipschitz spacetimes are either timelike or null},
   journal={Classical Quantum Gravity},
   volume={35},
   date={2018},
   number={8},
   pages={087001, 6},
%    issn={0264-9381},
%    review={\MR{3789523}},
%    doi={10.1088/1361-6382/aab259},
}

\bib{GST}{book}{
   author={Groah, J.},
   author={Smoller, J.},
   author={Temple, B.},
   title={Shock wave interactions in general relativity},
   series={Springer Monographs in Mathematics},
   note={A locally inertial Glimm scheme for spherically symmetric
   spacetimes},
   publisher={Springer, New York},
   date={2007},
   pages={viii+152},
%    isbn={978-0-387-35073-8},
%    isbn={0-387-35073-X},
%    review={\MR{2271028}},
%    doi={10.1007/978-0-387-44602-8},
}

\bib{G}{book}{
   author={Gromov, M.},
   title={Metric structures for Riemannian and non-Riemannian spaces},
   series={Modern Birkh\"auser Classics},
   edition={Reprint of the 2001 English edition},
%    note={Based on the 1981 French original;
%    With appendices by M. Katz, P. Pansu and S. Semmes;
%    Translated from the French by Sean Michael Bates},
   publisher={Birkh\"auser Boston, Inc., Boston, MA},
   date={2007},
   pages={xx+585},
%    isbn={978-0-8176-4582-3},
%    isbn={0-8176-4582-9},
%    review={\MR{2307192}},
}

\bib{GP}{article}{
   author={Grove, K.},
   author={Petersen, P.},
   title={Manifolds near the boundary of existence},
   journal={J. Differential Geom.},
   volume={33},
   date={1991},
   number={2},
   pages={379--394},
%    issn={0022-040X},
%    review={\MR{1094462}},
}

\bib{H}{article}{
   author={Harris, S.G.},
   title={A triangle comparison theorem for Lorentz manifolds},
   journal={Indiana Univ. Math. J.},
   volume={31},
   date={1982},
   number={3},
   pages={289--308},
%    issn={0022-2518},
%    review={\MR{652817}},
%    doi={10.1512/iumj.1982.31.31026},
}

\bib{H}{article}{
   author={Heinonen, J.},
   title={Geometric embeddings of metric spaces},
   date={January 2003},
 %  pages={44 p.},
   journal={Lectures in the Finnish Graduate School of Mathematics, University of Jyv{\"a}skyl{\"a}},
   eprint={http://www.math.jyu.fi/research/reports/rep90.pdf},
}

\bib{HM}{article}{
   author={Hounnonkpe, R.A.},
   author={Minguzzi, E.},
   title={Globally hyperbolic spacetimes can be defined without the `causal'
   condition},
   journal={Classical Quantum Gravity},
   volume={36},
   date={2019},
   number={19},
   pages={197001, 9},
%    issn={0264-9381},
%    review={\MR{4016706}},
%    doi={10.1088/1361-6382/ab3f11},
}

\bib{HL}{article}{
    AUTHOR = {Huang, L.-H.},
    Author={ Lee, D.A.},
     TITLE = {Stability of the positive mass theorem for graphical
              hypersurfaces of {E}uclidean space},
   JOURNAL = {Comm. Math. Phys.},
%  FJOURNAL = {Communications in Mathematical Physics},
    VOLUME = {337},
      YEAR = {2015},
    NUMBER = {1},
     PAGES = {151--169},
      ISSN = {0010-3616},
%    MRCLASS = {53C20 (53C24)},
%   MRNUMBER = {3324159},
% MRREVIEWER = {Changwei Xiong},
%        DOI = {10.1007/s00220-014-2265-9},
%        URL = {https://doi.org/10.1007/s00220-014-2265-9},
}

\bib{HLS}{article}{
   author={Huang, L.-H.},
   author={Lee, D.A.},
   author={Sormani, C.},
   title={Intrinsic flat stability of the positive mass theorem for
   graphical hypersurfaces of Euclidean space},
   journal={J. Reine Angew. Math.},
   volume={727},
   date={2017},
   pages={269--299},
%    issn={0075-4102},
%    review={\MR{3652253}},
%    doi={10.1515/crelle-2015-0051},
}

\bib{JL}{article}{
   author={Jauregui, J.L.},
   author={Lee, D.A.},
   title={Lower semicontinuity of ADM mass under intrinsic flat convergence},
   journal={Arxiv e-prints},
   year={2019},
   eprint={https://arxiv.org/abs/1903.00916},
}

\bib{Kok}{article}{
   author={Kokkendorff, S.L.},
   title={On the existence and construction of stably causal Lorentzian
   metrics},
   journal={Differential Geom. Appl.},
   volume={16},
   date={2002},
   number={2},
   pages={133--140},
%    issn={0926-2245},
%    review={\MR{1893904}},
%    doi={10.1016/S0926-2245(02)00063-3},
}

\bib{KS}{article}{
   author={Kunzinger, M.},
   author={S\"{a}mann, C.},
   title={Lorentzian length spaces},
   journal={Ann. Global Anal. Geom.},
   volume={54},
   date={2018},
   number={3},
   pages={399--447},
   issn={0232-704X},
%    review={\MR{3867652}},
%    doi={10.1007/s10455-018-9633-1},
}

\bib{Lak}{article}{
    AUTHOR = {Lakzian, S.},
     TITLE = {On diameter controls and smooth convergence away from
              singularities},
   JOURNAL = {Differential Geom. Appl.},
  FJOURNAL = {Differential Geometry and its Applications},
    VOLUME = {47},
      YEAR = {2016},
     PAGES = {99--129},
      ISSN = {0926-2245},
   %MRCLASS = {53C20 (53C44)},
  %MRNUMBER = {3504922},
 %      DOI = {10.1016/j.difgeo.2016.01.003},
    %   URL = {https://doi.org/10.1016/j.difgeo.2016.01.003},
}

\bib{LS}{article}{
    AUTHOR = {Lakzian, S.},
    author = {Sormani, C.},
     TITLE = {Smooth convergence away from singular sets},
   JOURNAL = {Comm. Anal. Geom.},
 % FJOURNAL = {Communications in Analysis and Geometry},
    VOLUME = {21},
      YEAR = {2013},
    NUMBER = {1},
     PAGES = {39--104},
      ISSN = {1019-8385},
%    MRCLASS = {53C23},
%   MRNUMBER = {3046939},
% MRREVIEWER = {Luca Granieri},
%        DOI = {10.4310/CAG.2013.v21.n1.a2},
%        URL = {https://doi.org/10.4310/CAG.2013.v21.n1.a2},
}

\bib{L}{article}{
   author={Lang, U.},
   title={Local currents in metric spaces},
   journal={J. Geom. Anal.},
   volume={21},
   date={2011},
   number={3},
   pages={683--742},
   issn={1050-6926},
%    review={\MR{2810849}},
%    doi={10.1007/s12220-010-9164-x},
}

\bib{LW}{article}{
   author={Lang, U.},
   author={Wenger, S.},
   title={The pointed flat compactness theorem for locally integral
   currents},
   journal={Comm. Anal. Geom.},
   volume={19},
   date={2011},
   number={1},
   pages={159--189},
   issn={1019-8385},
%    review={\MR{2818408}},
%    doi={10.4310/CAG.2011.v19.n1.a5},
}

\bib{DL}{article}{
    AUTHOR = {Lee, D.A.},
     TITLE = {On the near-equality case of the positive mass theorem},
   JOURNAL = {Duke Math. J.},
%  FJOURNAL = {Duke Mathematical Journal},
    VOLUME = {148},
      YEAR = {2009},
    NUMBER = {1},
     PAGES = {63--80},
      ISSN = {0012-7094},
%    MRCLASS = {53C21 (53C44 53C80)},
%   MRNUMBER = {2515100},
% MRREVIEWER = {Gabjin Yun},
%        DOI = {10.1215/00127094-2009-021},
%        URL = {https://doi.org/10.1215/00127094-2009-021},
}

\bib{LeeS}{article}{
    AUTHOR = {Lee, D.A.},
    AUTHOR = {Sormani, C.},
     TITLE = {Stability of the positive mass theorem for rotationally
              symmetric {R}iemannian manifolds},
   JOURNAL = {J. Reine Angew. Math.},
%  FJOURNAL = {Journal f\"{u}r die Reine und Angewandte Mathematik. [Crelle's
%              Journal]},
    VOLUME = {686},
      YEAR = {2014},
     PAGES = {187--220},
      ISSN = {0075-4102},
%    MRCLASS = {53C35 (58J60)},
%   MRNUMBER = {3176604},
% MRREVIEWER = {Vincent Minerbe},
%        DOI = {10.1515/crelle-2012-0094},
%        URL = {https://doi.org/10.1515/crelle-2012-0094},
}

\bib{LLF}{article}{
   author={Le Floch, B.},
   author={LeFloch, P.G.},
   title={On the global evolution of self-gravitating matter. Nonlinear
   interactions in Gowdy symmetry},
   journal={Arch. Ration. Mech. Anal.},
   volume={233},
   date={2019},
   number={1},
   pages={45--86},
%    issn={0003-9527},
%    review={\MR{3974638}},
%    doi={10.1007/s00205-018-01354-5},
}

\bib{LM}{article}{
   author={LeFloch, P.G.},
   author={Mardare, C.},
   title={Definition and stability of Lorentzian manifolds with
   distributional curvature},
   journal={Port. Math. (N.S.)},
   volume={64},
   date={2007},
   number={4},
   pages={535--573},
%    issn={0032-5155},
%    review={\MR{2374400}},
%    doi={10.4171/PM/1794},
}

\bib{LFS}{article}{
   author={LeFloch, P.G.},
   author={Sormani, C.},
   title={The nonlinear stability of rotationally symmetric spaces with low
   regularity},
   journal={J. Funct. Anal.},
   volume={268},
   date={2015},
   number={7},
   pages={2005--2065},
%    issn={0022-1236},
%    review={\MR{3315585}},
%    doi={10.1016/j.jfa.2014.12.012},
}

\bib{Lic}{article}{
   author={Lichnerowicz, A.},
   title={Ondes de choc et hypoth\`{e}ses de compressibilit\'{e} en magn\'{e}tohydrodynamique relativiste},
   journal={Comm. Math. Phys.},
   volume={12},
   date={1969},
   number={2},
   pages={145--174},
}

\bib{MM}{article}{
   author={Mantica, C.A.},
   author={Molinari, L.G.},
   title={Generalized Robertson-Walker spacetimes -- A survey},
   journal={International Journal of Geometric Methods in Modern Physics},
   volume={14},
   issue={03},
   date={2017},
   pages={1730001},
%    doi={10.1142/S021988781730001X},
}

\bib{McC}{article}{
   author={McCann, R.J.},
   title={Displacement convexity of {B}oltzmann's entropy characterizes the strong energy condition from general relativity},
   journal={to appear in Camb. J. Math., Arxiv e-prints},
   year={2018},
   eprint={https://arxiv.org/abs/1808.01536},
}

\bib{Min}{article}{
   author={Minguzzi, E.},
   title={Causality theory for closed cone structures with applications},
   journal={Rev. Math. Phys.},
   volume={31},
   date={2019},
   number={5},
   pages={1930001, 139},
%    issn={0129-055X},
%    review={\MR{3955368}},
%    doi={10.1142/S0129055X19300012},
}

\bib{MS}{article}{
   author={Minguzzi, E.},
   author={S\'{a}nchez, M.},
   title={The causal hierarchy of spacetimes},
   conference={
      title={Recent developments in pseudo-{R}iemannian geometry},
   },
   book={
      series={ESI Lect. Math. Phys.},
      publisher={Eur. Math. Soc., Z\"{u}rich},
   },
   date={2008},
   pages={299--358},
%    review={\MR{2436235}},
%    doi={10.4171/051-1/9},
}

\bib{MSu}{article}{
   author={Minguzzi, E.},
   author={Suhr, S.},
   title={Some regularity results for Lorentz-Finsler spaces},
   journal={Ann. Global Anal. Geom.},
   volume={56},
   date={2019},
   number={3},
   pages={597--611},
%    issn={0232-704X},
%    review={\MR{4009771}},
%    doi={10.1007/s10455-019-09681-w},
}

\bib{MoSu}{article}{
   author={Mondino, A.},
   author={Suhr, S.},
   title={An optimal transport formulation of the Einstein equations of general relativity},
   journal={Arxiv e-prints},
   date={2018},
   eprint={https://arxiv.org/abs/1810.13309},
}

\bib{N}{article}{
   author={Newman, R.P.A.C.},
   title={A proof of the splitting conjecture of S.-T. Yau},
   journal={J. Differential Geom.},
   volume={31},
   date={1990},
   number={1},
   pages={163--184},
%    issn={0022-040X},
%    review={\MR{1030669}},
}

\bib{No}{article}{
   author={Noldus, J.},
   title={A Lorentzian Gromov-Hausdorff notion of distance},
   journal={Classical Quantum Gravity},
   volume={21},
   date={2004},
   number={4},
   pages={839--850},
   issn={0264-9381},
%    review={\MR{2036128}},
%    doi={10.1088/0264-9381/21/4/007},
}

\bib{O}{book}{
	author={O'Neill, B.},
	title={Semi-Riemannian Geometry: with application to relativity},
	series={Pure and Applied Mathematics},
	publisher={Academic Press, Elsevier, San Diego, CA},
	date={1983},
}

\bib{RS}{article}{
   author={Rendall, A.D.},
   author={St\aa hl, F.},
   title={Shock waves in plane symmetric spacetimes},
   journal={Comm. Partial Differential Equations},
   volume={33},
   date={2008},
   number={10-12},
   pages={2020--2039},
%    issn={0360-5302},
%    review={\MR{2475328}},
%    doi={10.1080/03605300802421948},
}

\bib{Sae}{article}{
   author={S\"{a}mann, C.},
   title={Global hyperbolicity for spacetimes with continuous metrics},
   journal={Ann. Henri Poincar\'{e}},
   volume={17},
   date={2016},
   number={6},
   pages={1429--1455},
%    issn={1424-0637},
%    review={\MR{3500220}},
%    doi={10.1007/s00023-015-0425-x},
}

\bib{SakS}{article}{
    AUTHOR = {Sakovich, A.},
    AUTHOR = {Sormani, C.},
     TITLE = {Almost rigidity of the positive mass theorem for
              asymptotically hyperbolic manifolds with spherical symmetry},
   JOURNAL = {Gen. Relativity Gravitation},
%  FJOURNAL = {General Relativity and Gravitation},
    VOLUME = {49},
      YEAR = {2017},
    NUMBER = {9},
     PAGES = {Art. 125, 26},
      ISSN = {0001-7701},
%    MRCLASS = {53C20},
%   MRNUMBER = {3691759},
% MRREVIEWER = {Gordon Craig},
%        DOI = {10.1007/s10714-017-2291-y},
%        URL = {https://doi.org/10.1007/s10714-017-2291-y},
}

\bib{SakS2}{article}{
    AUTHOR = {Sakovich, A.},
    AUTHOR = {Sormani, C.},
     TITLE = {Spacetime intrinsic flat convergence},
     Journal={to appear on arXiv},
}

\bib{Sb}{article}{
   author={Sbierski, J.},
   title={The $C^0$-inextendibility of the Schwarzschild spacetime and the
   spacelike diameter in Lorentzian geometry},
   journal={J. Differential Geom.},
   volume={108},
   date={2018},
   number={2},
   pages={319--378},
%    issn={0022-040X},
%    review={\MR{3763070}},
%    doi={10.4310/jdg/1518490820},
}

\bib{SWo}{article}{
   author={Sorkin, R.D.},
   author={Woolgar, E.},
   title={A causal order for spacetimes with $C^0$ Lorentzian metrics: proof
   of compactness of the space of causal curves},
   journal={Classical Quantum Gravity},
   volume={13},
   date={1996},
   number={7},
   pages={1971--1993},
%    issn={0264-9381},
%    review={\MR{1400951}},
%    doi={10.1088/0264-9381/13/7/023},
}

\bib{S}{article}{
   author={Sormani, C.},
   title={Scalar curvature and intrinsic flat convergence},
   conference={
      title={Measure theory in non-smooth spaces},
   },
   book={
      series={Partial Differ. Equ. Meas. Theory},
      publisher={De Gruyter Open, Warsaw},
   },
   date={2017},
   pages={288--338},
%    review={\MR{3701743}},
}

\bib{S-Ob}{article}{
   author={Sormani, C.},
   title={Spacetime intrinsic flat convergence},
   journal={Oberwolfach Report for the Workshop ID 1832: Mathematical General Relativity},  
   date={2018},
   pages={1--3},
   url={https://arxiv.org/pdf/1805.08886.pdf}
}

\bib{SStav}{article}{
    AUTHOR = {Sormani, C.},
    AUTHOR = {Stavrov Allen, I.},
     TITLE = {Geometrostatic manifolds of small {ADM} mass},
   JOURNAL = {Comm. Pure Appl. Math.},
%  FJOURNAL = {Communications on Pure and Applied Mathematics},
    VOLUME = {72},
      YEAR = {2019},
    NUMBER = {6},
     PAGES = {1243--1287},
      ISSN = {0010-3640},
%    MRCLASS = {53C20 (83C22)},
%   MRNUMBER = {3948557},
%        DOI = {10.1002/cpa.21807},
%        URL = {https://doi.org/10.1002/cpa.21807},
}

\bib{SV}{article}{
   author={Sormani, C.},
   author={Vega, C.},
   title={Null distance on a spacetime},
   journal={Classical Quantum Gravity},
   volume={33},
   date={2016},
   number={8},
   pages={085001, 29},
%    issn={0264-9381},
%    review={\MR{3476515}},
%    doi={10.1088/0264-9381/33/7/085001},
}

\bib{SW}{article}{
   author={Sormani, C.},
   author={Wenger, S.},
   title={The intrinsic flat distance between Riemannian manifolds and other
   integral current spaces},
   journal={J. Differential Geom.},
   volume={87},
   date={2011},
   number={1},
   pages={117--199},
%    issn={0022-040X},
%    review={\MR{2786592}},
}

\bib{TZ}{article}{
   author={Tahvildar-Zadeh, A.S.},
   title={On the static spacetime of a single point charge},
   journal={Rev. Math. Phys.},
   volume={23},
   date={2011},
   number={3},
   pages={309--346},
%    issn={0129-055X},
%    review={\MR{2793478}},
%    doi={10.1142/S0129055X11004308},
}

\bib{Ta}{article}{
   author={Taub, A.H.},
   title={Space-times with distribution-valued curvature tensors},
   journal={J. Math. Phys.},
   volume={21},
   date={1980},
   number={6},
   pages={1423--1431},
%    issn={0022-2488},
%    review={\MR{574705}},
%    doi={10.1063/1.524568},
}

\bib{W}{article}{
   author={Wong, W.W.-Y.},
   title={A comment on the construction of the maximal globally hyperbolic
   Cauchy development},
   journal={J. Math. Phys.},
   volume={54},
   date={2013},
   number={11},
   pages={113511, 8},
%    issn={0022-2488},
%    review={\MR{3154377}},
%    doi={10.1063/1.4833375},
}

\bib{Y}{article}{
   author={Yun, J.-G.},
   title={Volume comparison for Lorentzian warped products with integral
   curvature bounds},
   journal={J. Geom. Phys.},
   volume={57},
   date={2007},
   number={3},
   pages={903--912},
%    issn={0393-0440},
%    review={\MR{2275199}},
%    doi={10.1016/j.geomphys.2006.07.001},
}

\bib{Z}{article}{
   author={Zeghib, A.},
   title={Geometry of warped products},
   journal={Arxiv e-prints},
   year={2011},
   eprint={https://arxiv.org/abs/1107.0411},
}

\end{thebibliography}
\end{document}